\documentclass[reqno,10pt]{amsart}
\usepackage{amssymb}
\usepackage{amsmath}
\usepackage{amsthm}
\usepackage[usenames]{color}
\usepackage{graphicx}
\usepackage{cite}
\usepackage{bbm}

\usepackage[colorlinks,linkcolor=blue,anchorcolor=red,citecolor=blue]{hyperref}
\usepackage[margin=1in]{geometry} 
\usepackage{marginnote}

\usepackage{color}

\newtheorem{thm}{Theorem}[section]
\newtheorem{cor}[thm]{Corollary}
\newtheorem{lem}[thm]{Lemma}
\newtheorem{prop}[thm]{Proposition}
\newtheorem{rem}{Remark}[section]

\numberwithin{equation}{section}

\newcommand{\R}{\mathbb{R}}

\newcommand{\CE}{\mathcal{E}}

\newcommand{\de}{\delta}

\newcommand{\vertiii}[1]{{\left\vert\kern-0.25ex\left\vert\kern-0.25ex\left\vert #1
		\right\vert\kern-0.25ex\right\vert\kern-0.25ex\right\vert}}

\begin{document}

\title[Time-periodic solutions to 3D VPB system]
{Time-periodic solutions of the Vlasov-Poisson-Boltzmann system with a general external force in $\mathbb{R}^3$}

\author[R.-J. Duan]{Renjun Duan}
\address[RJD]{Department  of Mathematics, The Chinese University of Hong Kong,  Shatin, Hong Kong}
\email{rjduan@math.cuhk.edu.hk}

\author[J. Ni] {Jinkai Ni}
\address[JKN]{Department of Applied Mathematics, The Hong Kong Polytechnic University, Hong Kong}
\email{jinkaini123@gmail.com}

\author[T. Yang] {Tong Yang}
\address[TY]{Department of Applied Mathematics, The Hong Kong Polytechnic University, Hong Kong}
\email{t.yang@polyu.edu.hk}

\begin{abstract}
This paper is concerned with the time-periodic solutions of the Vlasov-Poisson-Boltzmann (VPB) system with a given time-periodic external force in the whole space $\mathbb{R}^3$, where the force is allowed to be non-potential. For the purpose, we first study the Cauchy problem around global Maxwellians, proving the global existence of small-amplitude solutions in a hybrid function space that combines the low-frequency Besov framework for the forced Boltzmann equation with a corresponding control of the self-consistent electric field.  One key point occurs to the nonlinear Vlasov force $-\nabla_{x}\phi \cdot \nabla_{v}f + \frac{1}{2}v \cdot \nabla_{x}\phi f$ at low frequencies. Rather than treating it as a generic source term, we exploit the Poisson equation and macroscopic balance laws to recover the structural cancellation required for the VPB semi-group estimates, which combined with high-frequency energy estimates and weighted microscopic propagation, yields a closed global well-posedness theory. We further prove uniform stability of the difference of two small-amplitude solutions driven by the same force with respect to distinct initial data. As a consequence of those results, for the time-periodic problem where the external force is time-periodic, we are able to make use of Serrin's method to establish the existence of a unique time-periodic solution of the same period together with its large time asymptotic stability. Our result also gives the existence and stability of stationary solutions when the external force is time-independent.
\end{abstract}

\date{\today}
	
\subjclass[2020]{35Q20, 35Q83, 35B10, 35B35, 35B40}
%35Q20  	Boltzmann equations
%35Q83  	Vlasov equations 
%35B10  	Periodic solutions to PDEs
%35B35  	Stability in context of PDEs
%35B40  	Asymptotic behavior of solutions to PDEs

\keywords{Vlasov-Poisson-Boltzmann system, time-periodic solution, time-periodic force, global well-posedness, asymptotic stability}

\maketitle
\thispagestyle{empty}
	
%\tableofcontents

\setcounter{equation}{0}
 \indent 
 \allowdisplaybreaks

\section{Introduction}
\subsection{Problem}
The three-dimensional Vlasov-Poisson-Boltzmann (VPB) system for a single species of particles under the influence of an external force reads as
\begin{equation}\label{I1.1}
\left\{
\begin{aligned}
& \partial_tF+v\cdot\nabla_{x} F+(\nabla_{x}\phi+E )\cdot\nabla_v F=Q(F,F),\\
&\Delta_{x}\phi=\int_{\mathbb R^3}F {\rm d}v-\bar\rho(x),
\end{aligned}
\right.
\end{equation}
with the initial data
\begin{align}\label{I1.2}
F(0,x,v)=F_0(x,v).    
\end{align}
Here, the unknown $F = F(t, x, v)$ is a non-negative function representing the number density of gas particles that have position $x=(x_1,x_2,x_3)\in \mathbb R^3$ and velocity $v=(v_1,v_2,v_3)\in\mathbb R^3$ at time $t > 0$. The given external force field is denoted as $E = E(t, x)\in \mathbb R^3$, which is assumed to be independent of velocity $v$. The bilinear collision operator $Q$ for the hard-sphere model is defined by
\begin{align*}
Q(F,G)=\int_{\mathbb R^3}\!\int_{\mathbb S^2}  |(v-v_*)\cdot\omega|(F^\prime G_{*}^\prime-F G_{*})\,  {\rm d}\omega {\rm d}v_*,
\end{align*}
where 
\begin{align*}
F=F(t,x,v),\quad  F^\prime=F(t,x,v^\prime),\quad G_{*}=G(t,x,v_{*}),\quad G_{*}^\prime=G(t,x,v_{*}^\prime),
\end{align*}
with
\begin{align*}
v^\prime=v-[(v-v_*)\cdot\omega]\omega,\quad v_{*}^\prime= v_{*}   +[(v-v_{*})\cdot\omega]\omega,\quad \omega\in \mathbb{S}^2.
\end{align*}
In addition, the unknown potential function $\phi=\phi(t,x)$ that generates the self-consistent electric field $\nabla_{x}\phi$ in \eqref{I1.1}$_1$ is coupled with $F(t,x,v)$ via the Poisson equation \eqref{I1.1}$_2$, where $\bar\rho(x)$ denotes the stationary background density satisfying
\begin{align*}
\bar\rho(x)\rightarrow \rho_{\infty},\quad \text{as}\quad |x|\rightarrow+\infty,    
\end{align*}
for a positive constant state $\rho_{\infty}>0$. Throughout this paper, we restrict our attention to the constant background density
\begin{align*}
\bar\rho(x) \equiv\rho_{\infty}=1,\quad x\in\mathbb R^3.
\end{align*}

We aim at studying the VPB system \eqref{I1.1} under the assumption that the external force field $E(t,x)$ is periodic in time with the period $T > 0$, namely, 
$$
E(t,x)=E(t + T,x)
$$
for any $t\in \mathbb{R}$ and $x\in \mathbb{R}^3$. Note that this is also applicable to the stationary case whenever $E = E(x)$ is further independent of time $t$, for which the period can be an arbitrary positive constant. Under the influence of prescribed time-periodic forces, we seek for the solution to the following time-periodic problem:
 \begin{equation} \label{G1.3}
\left\{\begin{aligned}
& \partial_t F_{T}+ v\cdot \nabla_{x} F_{T}+(\nabla_{x}\phi_{T}+E)\cdot\nabla_{v}F_{T}=Q(F_{T},F_{T}) \quad \text{in}\quad \mathbb R\times \mathbb R_{x}^3\times \mathbb R_{v}^3,\\
&\Delta_{x}\phi_{T}=\int_{\mathbb{R}^3}F_{T}{\rm d}v-1 \quad \text{in}\quad \mathbb R\times\mathbb R^3_{x},\\
&F_{T}(t,x,v)=F_{T}(t+T,x,v), \quad \text{for}\quad (t,x,v)\in \mathbb R\times \mathbb R_x^3\times \mathbb R_v^3, \\
& \phi_{T}(t,x)=\phi_{T}(t+T,x),\quad \text{for}\quad (t,x)\in \mathbb R\times \mathbb R_x^3. 
 \end{aligned}
 \right.
\end{equation}   
If the external force field disappears, i.e., $E\equiv 0$,    then $(F,\phi)\equiv(M,0)$ is a trivial equilibrium state solution to the VPB system \eqref{I1.1}--\eqref{I1.2},
where the global Maxwellian $M$ is defined by  
\begin{align}\label{def.M}  
M = (2\pi)^{-3/2} e^{-|v|^2/2},  
\end{align} 
and $M$ has zero bulk velocity, unit density, and unit temperature. Then, the main goal of this paper is to investigate the time-periodic problem \eqref{G1.3} in a neighborhood of $(M,0)$; before stating the main results, related works are reviewed in the following subsection.

\subsection{Previous works}
There is a large body of literature on the Cauchy problem and on the large-time behavior of the Boltzmann equation and related kinetic models.
For the Boltzmann equation without external forces, the global existence of renormalized weak solutions for large initial data was established by DiPerna-Lions \cite{DL-AM-1989}. See also Hamdache \cite{Hamdache-ARMA-1992} and Arkeryd–Maslova \cite{AM-JSP-1994} for initial-boundary value problems. 
The question of convergence to equilibrium has been studied by several different methods.
Desvillettes-Villani \cite{DV-IM-2005} obtained almost exponential convergence under additional regularity assumptions, while Gualdani-Mischler-Mouhot \cite{GMM-MSM-2017} developed a factorization method to achieve optimal exponential convergence.
In the perturbative regime around a Maxwellian, Ukai’s spectral approach \cite{Ukai-1974,Ukai-1976} initiated the classical global theory. The spectral analysis of Ellis-Pinsky \cite{EP-JMPA-1975}, Kawashima’s compensation method \cite{Ks-JJAM-1990}, and the subsequent decay theory in the whole space have clarified the diffusive role of the low-frequency part of the linearized semi-group. Another fundamental method is the nonlinear energy method based on the macro-micro decomposition. In this direction, Guo \cite{GY-iumj-2004} developed a robust energy framework around global Maxwellians, while Liu-Yang-Yu \cite{LYY-PD-2004} developed the energy method around local Maxwellians for studying the stability of wave patterns in the context of conservation laws.

Time-periodic solutions are classical objects in the theory of dissipative equations. Serrin \cite{Serrin-1959-ARMA}
introduced an influential stability-based construction for periodic solutions of the Navier–Stokes equations.
Related periodic or stationary theories for fluid equations were developed by Valli \cite{Valli-1983},
Valli–Zajaczkowski \cite{VZ-CMP-1986}, Feireisl et al. \cite{EMPS-ARMA-1999},  Beirão da Veiga \cite{Bb-ARMA-2005}, Kagei-Tsuda \cite{KaTs}, Tsuda\cite{Tk-ARMA-2016}.
For the full compressible Navier–Stokes–Fourier system in the whole space, Deguchi \cite{Deguchi-2026} recently emphasized a low-frequency homogeneous Besov framework for the three-dimensional problem. Such a viewpoint is also relevant for kinetic models, as in the whole space, the obstruction to periodic forcing is associated with the slow decay of the lowest spatial frequencies.

For the Boltzmann equation with a time-periodic source, Ukai \cite{Ukai-2006} constructed periodic solutions near the normalized Maxwellian by using a fixed-point argument over an integral extending from the infinite past.
The construction is based on time-decay estimates for the linearized Boltzmann semi-group. In dimensions $n = 3$ and $4$, a zero-average condition on the macroscopic part of the source is needed to obtain sufficient decay; whereas in dimensions $n \geq 5$, stronger dispersion removes this restriction. 
Ukai and the third author \cite{UY-AA-2006} later obtained related global and time-periodic results in the space \(L^2 \cap L^\infty_\beta\). These works treat inhomogeneous source terms, not external force fields. The latter case is more delicate because, after perturbation around a global Maxwellian, a velocity-independent force produces the terms \(-E \cdot \nabla_v f\), \(\frac{1}{2} E \cdot v\, f\), and \(E \cdot v \sqrt{M}\), which involve   velocity differentiation, velocity growth and  a pure forcing term.
Recently,
in the source-driven setting, the authors of this paper along with Lei \cite{DLN-2026-arXiv} removed the zero-average restriction on $S_0$ by developing a global dynamical framework based on low-frequency Besov estimates.
For the externally forced Boltzmann equation, the first author, Ukai, the third author, and Zhao \cite{DUYZ-CMP-2008} obtained optimal decay estimates for the linearized problem with a time-dependent force and applied them to the nonlinear Cauchy problem and the time-periodic problem in dimensions $n\ge5$. 
The three-dimensional case remained open because the low-frequency time convolution is non-integrable. Recently, this difficulty was overcome by the authors of this paper \cite{DN-2026} via the low-frequency framework $L_v^2(\dot B_{2,\infty}^{1/2})$ and Serrin's construction \cite{Serrin-1959-ARMA}. In particular, measuring the force in $\mathcal C(\mathbb R;\dot B^{-3/2}_{2,\infty}\cap\dot H^N)$ results in a two-derivative low-frequency gain, which controls $E\cdot v\sqrt M$ without imposing time decay on $E$. Moreover, our argument also provides stationary solutions for small time-independent forces, including non-potential rotational fields.

We next turn to the VPB  system, which couples the Boltzmann collision dynamics with the self-consistent electric field determined by the Poisson equation. 
The perturbative theory near Maxwellians was initiated by Guo \cite{Gy-CPAM-2002}; related classical solutions and large-time behavior were further studied in \cite{YYZ-ARMA-2006,YZ-CMP-2006,DYZ-DCDS-2006}, while initial-boundary value and weak solution issues were treated by Mischler \cite{Mischler-CMP-2000}. 
Compared with the force-free Boltzmann equation, the Poisson coupling changes the low-frequency dynamics in an essential way: the electric field enters the lowest-order energy estimate and slows the dispersion of the system.
This effect was quantified by the first author of this paper and Strain \cite{RS-ARMA-2011}, who established optimal time-decay estimates for the one-species VPB system in $\mathbb{R}^3$. In particular, the sharp $L^2$-decay rate is $(1+t)^{-1/4}$, which is slower than the $(1+t)^{-3/4}$ rate for the Boltzmann equation without external forces. The analysis combines Fourier-based estimates for the linearized VPB semi-group with nonlinear Lyapunov-type energy inequalities. A key structural feature is the refined macroscopic projection into the $\mathbf{P}_0$- and $\mathbf{P}_1$-components (cf.~\eqref{G2.5}), which separates the hyperbolic density mode from the momentum–temperature modes and incorporates the electric field directly into the energy functional. This structure is consistent with the general hypocoercive framework for collisional kinetic equations; see, for instance, \cite{Villani-MAMS-2009,SG-CPDE-2006,SG-ARMA-2008,MN-nonlinearity-2006}.

Motivated by these developments, it is natural to consider the time-periodic problem for the three-dimensional VPB system \eqref{G1.3} with a given external force. This problem involves two distinct low-frequency mechanisms: the external force induces velocity-growth and velocity-derivative terms in the forced Boltzmann equation, while the Poisson coupling introduces a self-consistent electric field and leads to slower, field-dominated dynamics in the VPB system. Consequently, the analysis must bridge the low-frequency treatment of external forces developed in \cite{DUYZ-CMP-2008,DN-2026} with the field-energy structure and the $\mathbf{P}_0$–$\mathbf{P}_1$ macroscopic decomposition introduced in \cite{RS-ARMA-2011}. Our purpose in the current paper is to study this time-periodic problem \eqref{G1.3} for the forced VPB system in the physically relevant setting of the whole three-dimensional space.

\subsection{Main results}
In this paper, in order to study the time-periodic problem \eqref{G1.3}, we need to mainly focus on the Cauchy problem \eqref{I1.1} and \eqref{I1.2} for the external force field $E(t,x)$ which is bounded for all time $t\geq 0$, including the specific case where it is time-periodic. Recall \eqref{def.M}. We define the perturbation $f = f(t,x,v)$ as 
\begin{align*} 
F = M + \sqrt{M}  f.
\end{align*}  
Then, under the formulation above, the Cauchy problem \eqref{I1.1} and \eqref{I1.2}  for the perturbation $f=f(t,x,v)$ and $\phi=\phi(t,x)$ can be re-written as
\begin{equation}\label{I1.4}
\left\{
\begin{aligned}
& \partial_t f+v\cdot\nabla_x f-\mathcal{L}f-\nabla_{x}\phi\cdot v\sqrt M\\
&\quad=\Gamma(f,f)+G_\phi(f)+G_E(f)+E\cdot v\sqrt M,\quad(t,x,v)\in \mathbb R^+\times\mathbb R^3_{x}\times \mathbb R_{v}^3,    \\
&\Delta_{x}\phi= \int_{\mathbb R^3}\sqrt{M}f{\rm d}v, \quad(t,x)\in \mathbb R^+\times\mathbb R^3_{x},\\
&f(0,x,v)=f_0(x,v):=\frac{F_0(x,v)-M}{\sqrt{M}}, \quad(x,v)\in \mathbb R^3_{x}\times \mathbb R_{v}^3,
\end{aligned}
\right.
\end{equation}
where
\begin{align*}
G_\phi(f)=-\nabla_{x}\phi\cdot\nabla_v f+\frac12(v\cdot \nabla_{x}\phi)f,
\qquad
G_E(f)=-E\cdot\nabla_v f+\frac12(v\cdot E)f.    
\end{align*}
In equation \eqref{I1.4}$_1$, the linearized collision operator $\mathcal{L}$ and the nonlinear collision operator $\Gamma$ are respectively denoted by
\begin{gather*}
 \mathcal{L} g={ \frac{1}{\sqrt{M}}} \big[ Q(M,\sqrt{M}g)+Q (\sqrt{M}g,M) \big],  \\
 \Gamma(g_1,g_2)= \frac{1}{\sqrt{M}}Q\big(\sqrt{M}g_1,\sqrt{M}g_2\big).
\end{gather*}

%To study the time-periodic solution of \eqref{G1.3}, we initially consider the Cauchy problem \eqref{I1.4}.
For later use, we introduce the following energy norm associated with the
perturbation \(f\):
\begin{align}\label{G1.5}
\|f\|_{\mathcal E^{s,N}}
:=&\
\|f\|_{L^2_v(\dot B^s_{2,\infty}\cap \dot H^N_x)}
+\|\langle v\rangle f\|_{L^2_v(\dot H^1_x\cap\dot H^{N-1}_x)}
+\|\{\mathbf I-\mathbf P\}f\|_{L^2_{x,v}}
\nonumber\\
&\quad
+\sum_{\substack{1\le |\beta|\le N\\ |\alpha|+|\beta|\le N}}
\|\partial_x^\alpha\partial_v^\beta\{\mathbf I-\mathbf P\}f\|_{L^2_{x,v}} .
\end{align}

\begin{rem} 
The norm \eqref{G1.5} is adapted to the localized VPB semi-group estimates and to the force terms in the nonlinear equation. Different from the Boltzmann settings in \cite{DN-2026,DLN-2026-arXiv}, our approach is based on dyadic localized energy estimates for the VPB system. Moreover, the velocity weight may be explicitly retained in the energy to control the nonlinear terms involving \(vf\) and \(\nabla_v f\).

No separate propagation estimate for
\(\langle v\rangle^{\frac{1}{2}}\{\mathbf I-\mathbf P\}f\) is needed. Indeed, by the interpolation inequality in Lemma \ref{LA.8}, one has
\begin{align*} 
\|\langle v\rangle^{\frac12} {\{\mathbf I-\mathbf P\}f}\|_{L^2_v(\dot B^{\frac12}_{2,\infty})}
\lesssim
\| {\{\mathbf I-\mathbf P\}f}\|_{L^2_{x,v}}
+
\|\langle v\rangle {\{\mathbf I-\mathbf P\}}f\|_{L^2_v(\dot H_{x}^1 )} .
\end{align*}
Therefore, the half velocity weight required in the Besov estimates of the force terms is recovered from the microscopic \(L^2\)-norm and the weighted first-order spatial energy already included in \eqref{G1.5}. This is why
\(\|\langle v\rangle^{\frac{1}{2}}(\mathbf I-\mathbf P)f\|_{L^2_{x,v}}\) is not taken as an independent component of the total energy.
\end{rem}

The first main result concerning the Cauchy problem \eqref{I1.4} is stated as follows.

\begin{thm}\label{Th1}
    Let $F_0(x,v)=M+\sqrt{M}f_0(x,v)\geq 0$ with
    \begin{align*}
    f_0\in  L_v^2(\dot B_{2,\infty}^{\frac{1}{2}}\cap\dot H^N),\quad \langle v\rangle f_0\in L_v^2(\dot H^1\cap \dot H^{N-1}),\quad \{\mathbf{I}-\mathbf{P}\}f_0\in H^N_{x,v},
    \end{align*}
    and $\Delta_{x}\phi_0=\langle f_0, \sqrt{M}\rangle$, $ \nabla_{x}\phi_0\in \dot B^{\frac{1}{2}}_{2,\infty}\cap\dot H^N$,    
 for an integer $N\geq 3$, then there exists a constant $\delta_0>0$ such that if the external force $E$ and the initial data $f_0$ satisfy
\begin{equation}\label{G1.6}
\|f_0\|_{\CE^{\frac{1}{2},N}}+\|\nabla_{x}\phi_{0}\|_{\dot B_{2,\infty}^\frac{1}{2}\cap\dot H^N}+\|E (t)\|_{\mathcal{C} (\mathbb R; \dot B_{2,\infty}^{-\frac{3}{2}}\cap \dot H^N)  }\leq \delta_0,
\end{equation}
then the Cauchy problem \eqref{I1.4} admits a unique global solution $(f(t,x,v),\phi(t,x))$ that satisfies $F(t,x,v)=M+\sqrt{M}f(t,x,v)\geq 0$ with

\begin{equation*}
\left\{
\begin{aligned}
&  f\in \mathcal{C}\big([0,\infty);L_v^2(\dot B_{2,\infty}^{\frac{1}{2}}\cap \dot H^N)\big),\quad \langle v\rangle f\in \mathcal{C}\big([0,\infty);L_v^2( \dot H^1\cap \dot H^{N-1})\big),  \\
&\{\mathbf{I}-\mathbf{P}\}f\in \mathcal{C}\big([0,\infty);H^N_{x,v}\big),\quad \nabla_{x}\phi \in \mathcal{C}\big([0,\infty); \dot B_{2,\infty}^{\frac{1}{2}}\cap \dot H^N\big) ,
\end{aligned}
\right.
\end{equation*}
and
\begin{align}\label{G1.7}
\|f(t)\|_{\CE^{\frac{1}{2},N}}+\|\nabla_{x} \phi(t)\|_{\dot B_{2,\infty}^{\frac{1}{2}}\cap \dot H^N }
\leq  C_0 \left(\|f_0\|_{\CE^{\frac{1}{2},N}}
+\|
\nabla_{x} \phi_0\|_{\dot B_{2,\infty}^{\frac{1}{2}}\cap \dot H^N}+\|E (t)\|_{\mathcal{C} (\mathbb R; \dot B_{2,\infty}^{-\frac{3}{2}}\cap \dot H^N)  }\right), 
\end{align}
for any $t\geq 0$. Here $C_0$ is a positive constant independent of $t$. 
\end{thm}

\begin{rem}
Theorem \ref{Th1} extends the forced Boltzmann theory of \cite{DN-2026} to the more complex VPB system. Indeed, the total force in the present VPB system is
\begin{align*}
E(t,x)+\nabla_x\phi(t,x),\quad \text{with}\quad
\Delta_x\phi=\langle f,\sqrt M\rangle_v ,    
\end{align*}
where $E$ is the prescribed force as in \cite{DN-2026}, whereas $\nabla_x\phi$ is a self-consistent potential component. Hence our result may be regarded as the forced Boltzmann theory coupled with an additional Poisson field.
\end{rem}

We next derive uniform stability estimates for the Cauchy problem \eqref{I1.4}, which will be used to construct time-periodic solutions to the time-periodic problem \eqref{G1.3}.

\begin{thm}\label{Th2}
Let $0<\varepsilon<\frac{1}{2}$ and let $N\geq 4$ be an integer.
For $i=1,2$, let $F^{(i)}_0(x,v)=M+\sqrt{M}f^{(i)}_0(x,v)\geq 0$,  and let $(f^{(i)},\phi^{(i)})$   be the corresponding global solution  to the Cauchy problem \eqref{I1.4} with initial data $\big(f^{(i)}_0(x,v),\phi_0^{(i)}(x)\big)$. Suppose that the external force  $E$ and both initial data $\big(f^{(1)}_0,\phi_0^{(1)}\big) $ and $\big(f^{(2)}_0,\phi_0^{(2)}\big)$ satisfy \eqref{G1.6} in Theorem \ref{Th1} with $\delta_0>0$. Then, we have the following results:
\begin{itemize}
  \item[(i)]
If the initial data  $\big(f^{(1)}_0,\phi_0^{(1)}\big) $ and $\big(f^{(2)}_0,\phi_0^{(2)}\big)$  satisfy
\begin{align}\label{G1.8}
f^{(1)}_0 - f^{(2)}_0\in L_v^2(\dot B_{2,\infty}^{s_0}),\qquad\nabla_{x} \phi_0^{(1)}-\nabla_{x}\phi_0^{(2)}\in \dot B_{2,\infty}^{s_0},
\end{align}
for some $s_0\in\big(-\frac{1}{2},\frac{1}{2}\big]$, then we obtain the following time-weighted estimate
\begin{align}\label{G1.9}
&\sup_{t\geq 0}(1 + t)^{\frac{s - s_0}{2}}\Big(\big\|(f^{(1)} - f^{(2)})(t)\big\|_{L_v^2(\dot B_{2,\infty}^{s}\cap\dot H^{N-1})}+\big\|(\nabla_{x}\phi^{(1)} -\nabla_{x} \phi^{(2)})(t)\big\|_{ \dot B_{2,\infty}^{s}\cap\dot H^{N-1}}\Big)\nonumber\\
&+\sup_{t\geq 0} (1+t)^{\frac{1-\varepsilon-s_0}{2}}  \big \|\{\mathbf{I}-\mathbf{P}\}  (f^{(1)}-f^{(2)})(t)\big\|_{L_{x,v}^2}\nonumber\\
 &+\sup_{t\geq 0} (1+t)^{\frac{1-\varepsilon-s_0}{2}} 
 \big\|\langle v\rangle (f^{(1)}-f^{(2)}) (t)\big\|_{L_v^2(\dot H^1\cap\dot H^{N-2})}\nonumber\\
& +\sup_{t\geq 0} (1+t)^{\frac{1-\varepsilon-s_0}{2}}   \sum_{\substack{ 1\leq |\beta|\leq N-1 \\|\alpha|+|\beta| \leq N-1}}\big\|\partial^\alpha_{x}\partial^\beta_{v}\{\mathbf{I}-\mathbf{P}\}( f^{(1)}-f^{(2)})(t) \big\|_{L_{x,v}^2 }   \nonumber\\
\leq\,& C_1\Big(\|f^{(1)}_0 - f^{(2)}_0\|_{\CE^{s_0,N-1}}+\big\|(\nabla_{x}\phi^{(1)}_0 -\nabla_{x} \phi^{(2)}_0)(t)\big\|_{ \dot B_{2,\infty}^{s_0}\cap\dot H^{N-1}}\Big),
\end{align}
for any $s\in\big[-\frac{1}{2}+\varepsilon,1-\varepsilon \big]$ with $s\geq s_0$.

\item[(ii)]
If the initial data  $\big(f^{(1)}_0,\phi_0^{(1)}\big) $ and $\big(f^{(2)}_0,\phi_0^{(2)}\big)$  satisfy
\begin{align}\label{DNYG1.9.aa} 
f^{(1)}_0 - f^{(2)}_0\in L_v^2(\dot B_{2,\infty}^{s}),\qquad\nabla_{x} \phi_0^{(1)}-\nabla_{x}\phi_0^{(2)}\in \dot B_{2,\infty}^{s},
\end{align}
with $s\in \big( \frac{1}{2},1-\varepsilon    \big]$, then we have the following uniform  estimate
\begin{align}\label{DNYG1.9}
&\sup_{t\geq 0} \Big(\big\|(f^{(1)} - f^{(2)})(t)\big\|_{\CE^{s,N-1}}+\big\|(\nabla_{x}\phi^{(1)} -\nabla_{x} \phi^{(2)})(t)\big\|_{ \dot B_{2,\infty}^{s}\cap\dot H^{N-1}}\Big)\nonumber\\
\leq\,& C_1\Big(\|f^{(1)}_0 - f^{(2)}_0\|_{\CE^{s,N-1}}+\big\|(\nabla_{x}\phi^{(1)}_0 -\nabla_{x} \phi^{(2)}_0)(t)\big\|_{ \dot B_{2,\infty}^{s}\cap\dot H^{N-1}}\Big).
\end{align}
\end{itemize}
Here, in both \eqref{G1.9} and \eqref{DNYG1.9}, $C_1$ is a positive constant independent of $t$. 
\end{thm}

\begin{rem}
The estimate \eqref{DNYG1.9} holds under the common smallness assumption \eqref{G1.6} together with \eqref{DNYG1.9.aa} and does not require the additional
low-frequency assumption \eqref{G1.8}. In particular, the estimate with $s=1-\varepsilon$ is needed in the proof of
Theorem \ref{Th1.3}.
\end{rem}

\begin{rem}
The stability proof of Theorem \ref{Th2} differs from that in \cite{DN-2026} only at low frequencies. Instead of using the spectral estimates of the linearized Boltzmann semi-group directly, we treat the VPB low-frequency part by a dyadic energy method for
$\big(\widetilde f,\nabla_x \widetilde\phi\big)$.
The high-frequency estimates, including the time-weighted bounds for $\langle v\rangle \widetilde f$ and $\nabla_v\widetilde f$, follow the same strategy as in \cite{DN-2026}.
\end{rem}

Based on Theorems \ref{Th1} and \ref{Th2}, we further study the time-periodic problem \eqref{G1.3}. For the sake of simplicity, we define $F_T = M+\sqrt{M}f_{T}$. Then, we reformulate the problem \eqref{G1.3} as
\begin{equation}\label{G1.10}
\left\{\begin{aligned}
& \partial_t f_{T}+ v\cdot \nabla_{x} f_{T}-\mathcal{L}f_{T}{ -\nabla_{x}\phi_{T}\cdot v\sqrt{M}}\\
&\quad =\Gamma(f_{T},f_{T})-(E+\nabla_{x}\phi_{T})\cdot\nabla_v f_{T}+ \frac{1}{2}(E+\nabla_{x}\phi_{T})\cdot vf_{T}+E\cdot v\sqrt{M},\quad \text{in}\quad \mathbb R\times \mathbb R_{x}^3\times \mathbb R_{v}^3,\\
&\Delta_{x}\phi_{T}=\int_{\mathbb{R}^3}\sqrt{M}f_{T}{\rm d}v \quad \text{in}\quad \mathbb R\times\mathbb R^3_{x},\\
&f_{T}(t,x,v)=f_{T}(t+T,x,v)\quad \text{for}\quad (t,x,v)\in \mathbb R\times \mathbb R_x^3\times \mathbb R_v^3, \\
&\phi_{T}(t,x)=\phi_{T}(t+T,x),\quad \text{for}\quad (t,x)\in \mathbb R\times \mathbb R_x^3. 
\end{aligned} \right.
\end{equation} 
where $E(t,x )=E(t+T,x )$ for $(t,x)\in \mathbb R\times \mathbb R_x^3 $. Finally, we focus on the existence and stability of the time-periodic problem \eqref{G1.10}.

\begin{thm} \label{Th1.3}
Let $T > 0$ and let $N\geq 4$ be an integer. There exists a small constant $\delta > 0$ such that if the time-periodic force term $E$ with the time period $T > 0$ satisfies
\begin{equation}\label{cond.E}
\|E (t)\|_{\mathcal{C} (\mathbb{R}; \dot B_{2,\infty}^{-\frac{3}{2}}\cap \dot H^N )}\leq \delta,
\end{equation}
then we have the following results:
\begin{itemize}
    \item (Existence of time-periodic solutions) The problem \eqref{G1.10} admits a unique time-periodic solution $f_T$ of the same time period $T$ satisfying $F_T = M+\sqrt{M}f_{T}\geq 0$ and
    \begin{equation}\label{fT.bdd}
    \sup_{t\in \R}\|f_T(t)\|_{\CE^{\frac{1}{2},N}}+\sup_{t\in \mathbb R} \|\nabla_{x}{\phi_{ {T}}(t)}\|_{\dot B_{2,\infty}^{\frac{1}{2}}\cap\dot H^N}\leq C_2\delta.
     \end{equation}
    \item (Stability of time-periodic solutions) Consider the Cauchy problem \eqref{I1.4}. If initial data $f_0(x,v)$ satisfy   $F_0 = M+\sqrt{M}f_0\geq 0$, 
    \begin{align*}   
\|f_0\|_{\CE^{\frac{1}{2},N}}+\|\nabla_{x}\phi_0\|_{\dot B_{2,\infty}^{\frac{1}{2}}\cap\dot H^N}\leq \delta,
    \end{align*}
     with $\delta>0$ sufficiently small, and
    \begin{align*}
    f_0 - f_T(0) \in L^2_v(L^p),\qquad {\nabla_{x}\phi_0-\nabla_{x}\phi_{T}(0)\in L^p}
    \end{align*}
    with $\frac{3}{2}< p\leq 2$ and $\Delta_{x}\phi_0=\langle f_0,\sqrt{M}\rangle$, then the Cauchy problem  \eqref{I1.4}  admits a unique global solution $(f,\phi)$ satisfying $F = M+\sqrt{M}f\geq 0$,
     \begin{align}\label{G1.15.addthm}
   & \|f(t)\|_{\mathcal{C} ([0,\infty);L_v^2(\dot B_{2,\infty}^{\frac{1}{2}}\cap \dot H^N))}+\|(\langle v\rangle f(t),\nabla_v f(t ))\|_{\mathcal{C} ([0,\infty);L_v^2( \dot H^1\cap \dot H^{N-1}))}\nonumber\\
    &\quad+\|\nabla_{x}\phi(t)\|_{\mathcal{C} ([0,\infty); \dot B_{2,\infty}^{\frac{1}{2}}\cap \dot H^N)}   \leq C_2\delta,
    \end{align} 
    and the time decay estimate:
    \begin{align}\label{G1.15}
&\,\|(f - f_T)(t)\|_{L_v^2(\dot H^{s})}+\|(\nabla_{x}\phi-\nabla_{x}\phi_{T})(t)\|_{\dot H^s}\nonumber\\
\leq&\, C_2 (1 + t)^{-\frac{s}{2}-\frac{3}{2}(\frac{1}{p}-\frac{1}{2})}\Big(\|(f - f_T)(0)\|_{L_v^2( L^p\cap\dot H^N)} +\|\langle v\rangle (f - f_T)(0)\|_{L_v^2(\dot H^1\cap \dot H^{N-2})}   \nonumber\\
&+ {\|\{\mathbf{I}-\mathbf{P}\}(f - f_T)(0)\|_{L_{x,v}^2}}+\sum_{\substack{ 1\leq |\beta|\leq N-1 \\|\alpha|+|\beta| \leq N-1}}\|\partial^\alpha_{x}\partial^\beta_{v}\{\mathbf{I}-\mathbf{P}\}(f - f_T)(0)\|_{L_{x,v}^2 }\nonumber\\
&+\|(\nabla_{x}\phi- \nabla_{x}\phi_T)(0)\|_{  L^p\cap\dot H^N} \Big),
    \end{align}
    for any $s\in\big[-\frac{1}{2}+\varepsilon,1-\varepsilon\big]$ with $0<\varepsilon<\frac{1}{2}$ and $\frac{s}{2}+\frac{3}{2}\big(\frac{1}{p}-\frac{1}{2}\big)>0$.
\end{itemize} 
Here, in all three estimates \eqref{fT.bdd}, \eqref{G1.15.addthm} and \eqref{G1.15}, $C_2$ is a positive constant independent of $t$. 
\end{thm}

\begin{rem}
As a consequence of Theorem \ref{Th1.3}, the associated time-periodic linearized  VPB problem \eqref{G1.10} can be analyzed within the same framework. More precisely, by linearizing around the time-periodic profile and applying the stability estimate stated above, one establishes the existence and asymptotic stability of the corresponding linearized time-periodic solution, in the spirit of \cite[Theorem 1.4]{DLN-2026}.
\end{rem}

\begin{rem}
  
The admissible range of $s$ is different from those in \cite{DLN-2026} and \cite{DN-2026}. In the KFP problem studied in \cite{DLN-2026}, the low-frequency spectral analysis gives
$
s\in\big[-\frac12+\varepsilon,\frac32-\varepsilon\big],
$
whereas in the forced Boltzmann problem \cite{DN-2026}, one has
$
s\in\big[-\frac32+\varepsilon,1-\varepsilon\big],
$
For the present VPB system, the low-frequency part is treated by a localized energy method for the coupled pair
$
(f,\nabla_x \phi),
$
and the available range becomes
$
s\in\big[-\frac12+\varepsilon,1-\varepsilon\big].
$

\end{rem}

\subsection{Our contributions and beyond}
\begin{itemize}
    \item The first contribution is a global Cauchy theory for the hard-sphere VPB system \eqref{I1.4} in $\mathbb R^3$ with a small general external force 
$E\in \mathcal C\big( \mathbb R;\dot B^{-3/2}_{2,\infty}\cap\dot H^N\big)$ with  $N\ge4$. 
The solution is constructed near the global Maxwellian in a hybrid norm  
\begin{align*}
(f,\nabla_{x}\phi)\in L^2_v(\dot B^{1/2}_{2,\infty}\cap\dot H^N)\times  \dot B^{1/2}_{2,\infty}\cap\dot H^N ,    
\end{align*}
where velocity weights, velocity derivatives, and microscopic propagation estimates, as mentioned in \cite{DN-2026}, are utilized. As pointed out in \cite{DN-2026}, this class of forces includes genuinely rotational, non-conservative time-independent fields; hence the prescribed force $E $ here is not restricted to the potential case. Moreover, it is possible to extend the current results for the hard-sphere model to other types of angular cutoff interaction potentials that are either hard or even soft; we would not pursue such an extension in the current work. \\

\item The second contribution  lies in the low-frequency treatment of the nonlinear self-consistent field.
We prove VPB semi-group estimates for the pair
$(f,\nabla_{x}\phi)$ instead of just for $f$ alone.
At the level of dyadic blocks, the $\mathbf P_1$ source  exhibits a two-derivative low-frequency gain  as follows:
\begin{align*}
 \dot B_{2,\infty}^{-3/2}\rightarrow   \dot B_{2,\infty}^{1/2},
\end{align*}
whereas divergence sources acquire one spatial factor prior to the time integration, and microscopic sources possess the usual additional microscopic gain.
These estimates are applied to the exact
identity (see \eqref{G3.10}):
\begin{align*}
 \mathbf{P}_1\Big[\nabla_{x}\phi\cdot\Big( \frac{1}{2}vf - \nabla_{v}f\Big)\Big]  
=&\,  (\rho_{f} \nabla_{x}\phi)\cdot v\sqrt{M}+ \frac13(b\cdot \nabla_{x}\phi)(|v|^2-3)\sqrt{M} \nonumber\\
:=&\,\mathcal{I}_1+\mathcal{I}_2.   
\end{align*}
In the term $\mathcal{I}_1$, $\rho_{f}\nabla_{x}\phi$ is substituted with
\begin{align*}
\rho_{f} \nabla_{x}\phi=\nabla_x\cdot\big(\nabla_{x}\phi\otimes \nabla_{x}\phi-\frac12|\nabla_{x}\phi|^2{\rm Id}\big),   
\end{align*}
and in the term $\mathcal{I}_2$, $b\cdot\nabla_{x}\phi$ is decomposed as
\begin{align*}
 b \cdot\nabla_{x}\phi=       [\nabla_x\Delta_x^{-1}\nabla_x\cdot b+{({\rm Id}-\nabla_x\Delta_x^{-1}\nabla_x\cdot )b}]\cdot\nabla_{x}\phi:=(  b_{\parallel}+b_{\perp})\cdot\nabla_{x}\phi. 
\end{align*}
In this manner, the low-frequency estimate is closed even without assuming $f\in L_{x,v}^2$ or $\nabla_{x}\phi\in L^2$.\\

\item  The third point is the construction of time-periodic and stationary solutions under general external forces. Once the global Cauchy theory becomes available, Serrin's method \cite{Mp-1991-Nonlinearity,Serrin-1959-ARMA}  enables us  to construct a  $T$-periodic solution whenever the  general force $E(t,x)$ is $T$-periodic. The force does not need to be a potential force, so this result is particularly applicable to rotational external fields.  Moreover, similar to the stationary argument in \cite[Theorem 1.4]{DLN-2026}, the time-independent case leads to small stationary solutions driven by general external forces. Moreover, the background density $\bar\rho(x)$ can be assumed to be a small smooth perturbation around an arbitrary positive constant $\rho_{\infty}$, since the perturbed force field $\nabla_x \Delta_x^{-1} (\bar\rho-\rho_{\infty})$ can be regarded as an additional potential part of the given force field $E(t,x)$.

\end{itemize}

\subsection{Strategies in our proofs}
The proof is constructed based on the dynamical scheme of \cite{DN-2026}. First, a global existence theory is established for a given general external force. Then, a stability estimate is derived for the difference of two solutions under the same force. Finally, Serrin's argument is applied to the time-\(T\) map. The main new point is that in the VPB system, the self-consistent field is part of the unknowns. 
Therefore, the Boltzmann semi-group and the Boltzmann energy estimates used in \cite{DN-2026} need to be replaced by their VPB counterparts, and the nonlinear Vlasov force $-\nabla_{x}\phi\cdot\nabla_{v}f+\frac{1}{2}(v\cdot\nabla_{x}\phi)f$ has to be dealt with through the Poisson structure rather than being treated as a general source term.

\begin{itemize}
    \item 
The Cauchy problem is divided into low- and high-frequency components in physical space. Following the low-frequency framework of \cite{DN-2026,DLN-2026-arXiv}, which was in turn motivated by Deguchi’s treatment of the three-dimensional Navier-Stokes-Fourier system \cite{Deguchi-2026}, we do not measure the low-frequency component in the full $L^2_{x,v}$-norm. Instead, we use the homogeneous Besov space $L^2_v(\dot{B}^{1/2}_{2,\infty})$. This choice is natural in three dimensions, since   $1/|x|$ belongs to $\dot{B}^{1/2}_{2,\infty}$.
In the current VPB problem \eqref{I1.4}, one must also control the self-consistent electric field $\nabla_x \Delta_x^{-1} \langle f, \sqrt{M} \rangle_v$, which is derived from the density via a Fourier multiplier of order $-1$. As a result, it introduces an additional low-frequency difficulty that does not occur in the Boltzmann equation.
By decomposing $f$ into $f_{L}+f_{H}$ and applying the Duhamel principle, the solution of the Cauchy problem \eqref{I1.4} can be expressed in the mild form:
\begin{align*}
f(t)=e^{t\mathcal B}f_0
+\int_0^t e^{(t-\tau)\mathcal B}\mathcal M(\tau) {\rm d}\tau,    
\end{align*}
 where 
\begin{align*}
\mathcal M
=
\Gamma(f,f)
-(\nabla_x\phi+E)\cdot\nabla_v f
+\frac12 v\cdot(\nabla_x\phi+E)f
+E\cdot v\sqrt M.    
\end{align*}
The source term \(E\cdot v\sqrt M\) is  
estimated through
\begin{align*}
\dot B^{-3/2}_{2,\infty}\longrightarrow \dot B^{1/2}_{2,\infty},    
\end{align*}
which is the mechanism used in \cite{DN-2026} to avoid a non-integrable time convolution.  The terms \(E\cdot vf\) and \(E\cdot\nabla_vf\) are treated in the same spirit, leveraging the high-frequency propagation of \(\langle v\rangle f\) and \(\nabla_v f\).
The only truly new low-frequency term is
\begin{align*}
 G_\phi=-\nabla_x\phi\cdot\nabla_v f+\frac12(v\cdot\nabla_x\phi)f.   
\end{align*}
Since $\mathbf{P}_0 G_{\phi}=0$, only its $\mathbf{P}_1$-projection matters
\begin{align*}
 \mathbf{P}_1 G_{\phi}  
=&\,  (\rho_{f} \nabla_{x}\phi)\cdot v\sqrt{M}+ \frac13(b\cdot \nabla_{x}\phi)(|v|^2-3)\sqrt{M}.
\end{align*}
The product  $\rho_{f}\nabla_{x}\phi$ is written as the divergence of the Maxwell stress tensor. For the second
product we split $b=b_{\parallel}+b_{\perp}$. The continuity equation \eqref{G2.9}$_1$ gives
\begin{align*}
b_{\parallel}=-\partial_t \nabla_{x}\phi, \qquad b_{\parallel}\cdot \nabla_{x}\phi=-\frac12\partial_t|\nabla_{x}\phi|^2. 
\end{align*}
After performing integration by parts in the Duhamel formula, the generator of the VPB semi-group transforms this time derivative into a spatial derivative source.
The transverse part is reduced via the momentum equation (see {\bf  Step 3} in Section 3.2). Consequently, no estimate on the potential $\phi$ itself is required. The microscopic part of $G_{\phi}$ is closed by the microscopic semi-group gain and a zero-order half-weight estimate for $\{\mathbf{I}-\mathbf{P}\}f$. 
At high frequencies, the linear Poisson coupling is kept in the principal energy and cancels with the time derivative of the electric-field energy. The remaining macroscopic dissipation of $(a,b,c)$ is recovered from the Euler-Poisson type moment system \eqref{G2.8}--\eqref{G2.10}, together with the equations for the higher velocity moments, while the prescribed-force terms are estimated by the weighted method of \cite{DN-2026}.
Combining the low-frequency semi-group estimates with the high-frequency Lyapunov inequalities yields a closed a priori bound of the following form:
\begin{align*}
 \mathcal X(t)
\lesssim
\mathcal X(0)+
\|E\|_{\mathcal C(0,T_1;\dot B^{-3/2}_{2,\infty}\cap\dot H^N)}
+
\mathcal X^2(t)
+
\|E\|_{\mathcal C(0,T_1;\dot B^{-3/2}_{2,\infty}\cap\dot H^N)}\mathcal X^2(t),   
\end{align*}
for any $t\in[0,T_1]$, where 
\begin{align*}
  \mathcal{X}(t) =\|f(t)\|_{\CE^{\frac{1}{2},N}}+\|\nabla_{x}\phi(t)\|_{\dot B^{\frac{1}{2}}_{2,\infty}\cap\dot H^N},   
\end{align*}
which yields the global Cauchy estimate \eqref{G1.7} by a standard continuity argument.\\

\item The uniform stability argument follows the second step in \cite{DN-2026}. Let $(f^{(1)},\phi^{(1)})$ and $(f^{(2)},\phi^{(2)})$ be two small solutions driven by the same external force. In the difference equation \eqref{G4.1}, the pure source term $E \cdot v \sqrt{M}$ vanishes. The remaining forcing terms share the same structural form as those in the Cauchy problem, but are linear in the difference and have small coefficients. 
Our estimate is again obtained in two stages. First, high-frequency time-weighted estimates are established for the difference, involving both the time-weighted norm and the velocity-derivative norm. This step is necessary before applying the low-frequency semi-group, because the terms $E \cdot v \widetilde{f}$ and $E \cdot \nabla_v \widetilde{f}$ cannot be controlled using low-frequency analysis alone. Once high-frequency decay is available, it is substituted into the low-frequency Duhamel estimate. The argument parallels that in \cite{DN-2026}, where different ranges of the Besov index are treated separately.
In particular, for the lower range $s\in \big[-\frac{1}{2}+\varepsilon,\frac{1}{2}      \big)$, we need an additional time-weighted low-frequency estimate established in Proposition \ref{P2.5}, which provides the required bound for the VPB semi-group in the time-weighted Besov norm.
The VPB field terms are handled via bilinear versions of the Poisson identities used in the Cauchy problem, yielding decay of the difference in the hybrid Besov norm $\dot{B}_{2,\infty}^s \cap \dot{H}^{N-1}$.

\item

The periodic solution is constructed by Serrin's approach, as in \cite{DL-2015-AMSSB,DN-2026}. Let \((f^*(t),\phi^*(t))\)  be the global solution to the Cauchy problem \eqref{I1.4} with zero initial data. Since $E(t,x)$ is $T$-periodic, for any integers $m\geq k\geq 1$, the two functions \((f^*(t+(m - k)T),\phi^*(t+(m - k)T))\) and \((f^*(t),\phi^*(t))\) solve the same Cauchy problem, but with different initial data. 
Applying the stability estimate \eqref{G1.9}  and taking \(t=kT\) gives (see \eqref{G5.5}):
\begin{align*} 
\|f^*(mT)-f^*(kT)\|_{X_\varepsilon}+\|\nabla_{x}\phi^*(mT)-\nabla_{x}\phi^* (kT)\|_{Y_{\varepsilon}}
\lesssim
\delta_0(1+kT)^{-\frac14+\frac{\varepsilon}{2}},
\end{align*}
where
\begin{align*}
X_\varepsilon
:=
L_v^2\big(\dot B^{1-\varepsilon}_{2,\infty}\cap \dot H^{N-1}\big),\qquad   Y_\varepsilon
:=
 \dot B^{1-\varepsilon}_{2,\infty}\cap \dot H^{N-1}.  
\end{align*}
It follows that $\{f^*(nT)\}_{n\geq1}$ is a Cauchy sequence in \(X_\varepsilon\), and $\{\nabla_{x}\phi^*(nT)\}_{n\geq1}$ is a Cauchy sequence in \(Y_\varepsilon\).
Their limit is chosen as the initial data for a new global-in-time solution.
By the
time-periodicity of $E$ and the uniqueness of the Cauchy problem, this solution satisfies
\begin{align*}
f_{T}(t+T) =f_{T}(t),\qquad \nabla_{x}\phi _{T}(t+T)=\nabla_{x}\phi_{T}(t),  
\end{align*}
which gives a $T$-periodic solution to the problem \eqref{G1.10}. 
When   $E$ is independent of time, the periodic construction reduces to the stationary setting. In the same spirit as the stationary argument in \cite[Theorem 1.4]{DLN-2026}, which yields small stationary solutions driven by general external forces. In particular, the force is not required to be potential, so rotational external fields are also covered.

\end{itemize}

\subsection{Outline of this paper}
The remainder of this paper is structured as follows.
In Section \ref{sec2}, we first introduce some notations and properties of the collision operator $\mathcal{L}$ and the macro-micro decomposition for the Vlasov-Poisson-Boltzmann (VPB) system \eqref{I1.4}. 
We also derive the corresponding macroscopic balance laws and collect several analytic tools in homogeneous Besov spaces. 
Specifically, we formulate the low-frequency and high-frequency estimates for the linearized VPB semi-group.
In Section \ref{sec3}, we prove the global well-posedness of the Cauchy problem \eqref{I1.4} with a small general external force. 
The proof is based on a combination of low-frequency Besov estimates and high-frequency energy estimates.
A crucial aspect is the treatment of the nonlinear Vlasov term $G_\phi(f)=-\nabla_x\phi\cdot\nabla_v f+\frac{1}{2}(v\cdot\nabla_x\phi)f$. We utilize the special Poisson structure and the macroscopic normal form of $\mathbf{P}_1G_\phi(f)$ to complete the low-frequency estimates.
In Section \ref{sec4}, we establish the asymptotic stability of small global solutions of the VPB system \eqref{I1.4}. For two solutions driven by the same external force, we obtain time-decay estimates for their difference in the Besov framework. The argument depends on the linear VPB semi-group, weighted high-frequency estimates, and the cancellation structure of the self-consistent field. Finally, in Section \ref{sec5}, we construct time-periodic solutions for time-periodic external forces. Following Serrin's method (cf. \cite{Mp-1991-Nonlinearity,Serrin-1959-ARMA}), we prove the existence and asymptotic stability of the time-periodic solution to the periodic problem \eqref{G1.10}. In Section \ref{sec.app}, we list some basic estimates that have been used in the previous sections.

\section{Preliminaries}\label{sec2}

\subsection{Notations}
Throughout this paper, \(C\) denotes a generic positive constant, which may vary from line to line. For two non-negative quantities \(A\) and \(B\), we write \(A\lesssim B\) (respectively, \(A\gtrsim B\)) to mean that \(A\leq CB\) (respectively, \(A\geq C^{-1}B\)). Moreover, \(A\backsim B\) means that both \(A\lesssim B\) and \(A\gtrsim B\) hold.
The \(L^2\)-inner products with respect to the velocity and spatial variables are respectively denoted as follows:
\begin{align*}
\langle f,g\rangle=\int_{\mathbb R^3_v} f(v)g(v) {\rm d}v,\qquad (f,g)=\int_{\mathbb R^3_x} f(x)g(x) {\rm d}x.
\end{align*}
Their associated norms are denoted as \(|\cdot|_2\) and \(\|\cdot\|_{L^2}\). For functions that depend on both \(x\) and \(v\), we define
\begin{align*}
\langle f,g\rangle_{x,v}=\int_{\mathbb R^3_x}\!\int_{\mathbb R^3_v} f(x,v)g(x,v){\rm d}v{\rm d}x.
\end{align*}
We denote the corresponding norm by \(\|\cdot\|_{L^2_{x,v}}\), where \(L^2_{x,v}:=L^2(\mathbb R^3_x\times\mathbb R^3_v)\). Unless otherwise specified, \(\|\cdot\|_{L^p}\) represents the \(L^p\)-norm over \(\mathbb R^3_x\). For mixed norms, we define \(L^2_v(L^p):=L^2(\mathbb R^3_v;L^p(\mathbb R^3_x))\), with the norm \(\|\cdot\|_{L^2_v(L^p)}\). The same convention is applied to the spaces \(L^2_v(H^s)\), \(L^2_v(\dot H^s)\), and \(L^2_v(\dot B^s_{p,r})\), where \(s\in\mathbb R\) and \(1\leq p,r\leq\infty\). For $q\geq 1$, we also define
\begin{align*}
\mathcal{Z}_{q}=L_v^2(L^ q),\qquad \|g\|_{\mathcal{Z}_q}^2=\int_{\mathbb R^3}\bigg(\int_{\mathbb R^3}|g(x,v)|^q{\rm d}x\bigg)^{{2}/{q}}{\rm d}v.     
\end{align*}

Given a Banach space \(X\), we write
\begin{align*}
\|(g,h)\|_{X}:=\|g\|_{X}+\|h\|_{X},   
\end{align*}
whenever \(g,h\in X\). For an interval \(I\subset\mathbb R\), the space \(\mathcal C(I;X)\) denotes the set of all \(X\)-valued continuous functions on \(I\).
For \(T > 0\) and \(1\leq p\leq\infty\), 
we abbreviate
\begin{align*}
 L^p_T(X):=L^p(0,T;X),
\qquad
\|\cdot\|_{L^p_T(X)}
:=
\|\cdot\|_{L^p(0,T;X)} .   
\end{align*}
We denote by \(\mathcal S=\mathcal S(\mathbb R^3_x)\) the Schwartz class
and by \(\mathcal S'=\mathcal S'(\mathbb R^3_x)\) the space of tempered distributions.
For a function \(g=g(t,x,v)\), its Fourier transform with respect to the spatial variable
is defined by
\begin{align*}
\widehat g(t,\xi,v)
=\mathcal F g(t,\xi,v)
:=
\int_{\mathbb R^3} e^{-ix\cdot \xi}g(t,x,v)\,{\rm d}x,
\qquad 
x\cdot \xi=\sum_{j=1}^3x_j\xi_j ,    
\end{align*}
where \(\xi\in\mathbb R^3\) and \(i=\sqrt{-1}\). 
The inverse Fourier transform is denoted by \(\mathcal F^{-1}\).
For multi-indices \(\alpha,\beta\in\mathbb N^3\), we set
\begin{align*}
\partial_x^\alpha\partial_v^\beta
:=
\partial_{x_1}^{\alpha_1}\partial_{x_2}^{\alpha_2}
\partial_{x_3}^{\alpha_3}
\partial_{v_1}^{\beta_1}\partial_{v_2}^{\beta_2}
\partial_{v_3}^{\beta_3},   
\end{align*}
with
\begin{align*}
|\alpha|:=\alpha_1+\alpha_2+\alpha_3,
\qquad
|\beta|:=\beta_1+\beta_2+\beta_3 .    
\end{align*}
For the sake of brevity, \(\partial_i\) represents \(\partial_{x_i}\), where \(i = 1, 2, 3\).

Finally, for the hard-sphere Boltzmann collision operator, the collision frequency is given by
\begin{align}\label{G2.1}
 \nu(v)
:=
\int_{\mathbb R^3}\!\int_{\mathbb S^2}
|(v-v_*)\cdot\omega|\,M_*\,{\rm d}\omega{\rm d}v_* .   
\end{align}
It is well known that
\begin{align*}
\nu(v)\backsim \langle v\rangle,
\qquad 
\langle v\rangle:=\sqrt{1+|v|^2}.    
\end{align*}
Accordingly, we introduce the weighted norms
\begin{align*}
|g|_\nu:=|\nu^{1/2}g|_2,
\qquad
\|g\|_\nu:=\|\nu^{1/2}g\|_{L^2_{x,v}} .    
\end{align*}

\subsection{Some properties of \(\mathcal{L}\) and macro-micro decomposition}
For the linearized collision operator $\mathcal{L}$, as in \cite{UY-AA-2006,CIP-1994,Glassey 1996}, we have
\begin{align*}
(\mathcal{L}g )(v)=-\nu(v)g(v)+(\mathcal{K}g)(v), 
\end{align*}
with
\begin{align*}
(\mathcal{K}g)(v)=&\,\int_{\mathbb {R}^3}  \!\int_{\mathbb S^2}\big( 
 -\sqrt{M}g_{*}+\sqrt{M^\prime_{*}}g^\prime+\sqrt{M^\prime}g_{*}^\prime     \big) |(v-v_{*})\cdot\omega|\sqrt{M_{*}}{\rm d}\omega{\rm d} v_{*}  \nonumber\\
 =&\,\int_{\mathbb R^3}\mathcal{K}(v,v_{*})g(v_*){\rm d}v_{*},
\end{align*}
where the collision frequency $\nu(v)$ is defined by  \eqref{G2.1}, and $\mathcal{K}$ represents a self-adjoint compact operator on $L^2(\mathbb R^3_{v})$, which has a real symmetric integral kernel $\mathcal{K}(v,v_{*})$. The null space of $\mathcal{L}$ is five-dimensional, spanned by the following collision invariants:
\begin{align*}
\mathcal{N}=\ker\mathcal{L}=\mathrm{span}\big\{1, v_1,v_2,v_3,|v|^2\big\}\sqrt{M}.
\end{align*}

We shall also use the standard coercivity property of the linearized Boltzmann operator. By the \(H\)-theorem, \(\mathcal L\) is non-positive on \(L^2_v\). Moreover, there exists a constant
\(\lambda_0>0\) such that, for all \(g\in D(\mathcal{L})\),
\begin{align}\label{G2.2}
-\langle \mathcal Lg,g\rangle
\geq
\lambda_0\, |\{\mathbf I-\mathbf P\}g|_\nu^2.
\end{align}
Here, for each fixed \((t,x)\), \(\mathbf P\) denotes the \(L^2_v\)-orthogonal projection
onto the null space \(\mathcal N \), and the domain of \(\mathcal L\) is given by
\begin{align*}
D(\mathcal L)
:=
\big\{
g\in L^2_v:\ { \nu g}\in L^2_v
\big\}.
\end{align*}
Equivalently, \eqref{G2.2} can be written as
\begin{align*}
-\int_{\mathbb R^3_v} g(v)(\mathcal Lg)(v) {\rm d}v
\geq
\lambda_0
\int_{\mathbb R^3_v}
\nu(v)\,|\{\mathbf I-\mathbf P\}g(v)|^2 {\rm d}v .
\end{align*}

Next, we introduce the macro-micro decomposition, which was originally from \cite{GY-iumj-2004}.
For any function \(g=g(t,x,v)\),   \(\mathbf P g\) can be written as
\begin{align}\label{G2.3}
\mathbf P g
=
\left\{
a^g(t,x)+b^g(t,x)\cdot v+c^g(t,x)|v|^2
\right\}\sqrt M ,
\end{align}
where $a=a^g$, \(b=b^g=(b^g_1,b^g_2,b^g_3)\) and $c=c^g$. 
The coefficients are determined by the orthogonality condition
\begin{align*}
\{\mathbf I-\mathbf P\}g\perp \mathcal N
\qquad \text{in }\qquad L^2_v,    
\end{align*}
and are given by
\begin{align*} 
a^g(t,x)
=&\,
\frac12
\int_{\mathbb R^3_v}
(5-|v|^2)g(t,x,v)\sqrt M {\rm d}v,
\\
b^g_i(t,x)
=&\,
\int_{\mathbb R^3_v}
v_i g(t,x,v)\sqrt M {\rm d}v,
\qquad i=1,2,3,
\\
c^g(t,x)
=&\,
\frac16
\int_{\mathbb R^3_v}
(|v|^2-3)g(t,x,v)\sqrt M {\rm d}v .
\end{align*}
Thus every \(g\) admits the decomposition
\begin{align}\label{G2.4}
g=\mathbf P g+ {\{\mathbf I-\mathbf P\}g},
\end{align}
where \(\mathbf P g\in\mathcal N\) is called the macroscopic component, while 
\(\{\mathbf I-\mathbf P\}g\in\mathcal N ^\perp\) is called the microscopic component.

We further split the macroscopic part into
\begin{align}\label{G2.5}
\mathbf P g=\mathbf P_0g+\mathbf P_1g,
\end{align}
where
\begin{gather*}
\mathbf P_0g
=
(a+3c)\sqrt M,
\qquad
a +3c 
=
\int_{\mathbb R^3_v}
g(t,x,v)\sqrt M {\rm d}v,
\\
\mathbf P_1g
=
\left\{
b \cdot v+c (|v|^2-3)
\right\}\sqrt M .
\end{gather*}
Here, \(\mathbf P_0\) and \(\mathbf P_1\) denote the projections corresponding to the density mode and the remaining macroscopic  modes, respectively.

We now derive the macroscopic balance equations associated with the coefficients \((a,b,c)\). The derivation is based on the moment formulation of the kinetic equation along with the macro-micro decomposition \eqref{G2.4}; see also \cite{RS-ARMA-2011}.
Taking velocity moments of  the kinetic equation in \eqref{I1.4} with respect to
\(1\), \(v\), and \(|v|^2\), respectively, we obtain
\begin{equation}\label{G2.6}
\left\{
\begin{aligned}
&\partial_t\int_{\mathbb R^3_v}F{\rm d}v
+\nabla_x\cdot\int_{\mathbb R^3_v}vF{\rm d}v=0,
\\
&\partial_t\int_{\mathbb R^3_v}vF {\rm d}v
+\nabla_x\cdot\int_{\mathbb R^3_v}v\otimes vF {\rm d}v
-\bigl(\nabla_x\phi+E\bigr)\int_{\mathbb R^3_v}F{\rm d}v=0,
\\
&\partial_t\int_{\mathbb R^3_v}|v|^2F {\rm d}v
+\nabla_x\cdot\int_{\mathbb R^3_v}|v|^2vF {\rm d}v
-2\bigl(\nabla_x\phi+E\bigr)\cdot
\int_{\mathbb R^3_v}vF\,{\rm d}v=0.
\end{aligned}
\right.
\end{equation}
Here the collision operator does not contribute to the above identities, since
\(1\), \(v\), and \(|v|^2\) are collision invariants. 
Using the decomposition \eqref{G2.4}, together with the Gaussian moment
identities
\begin{gather*}
\int_{\mathbb R^3_v}M {\rm d}v=1,\qquad
\int_{\mathbb R^3_v}v_iM {\rm d}v=0,\qquad
\int_{\mathbb R^3_v}v_iv_jM {\rm d}v=\delta_{ij},    \\
\int_{\mathbb R^3_v}|v|^2M {\rm d}v=3,\qquad
\int_{\mathbb R^3_v}|v|^2v_iv_jM {\rm d}v=5\delta_{ij},
\qquad
\int_{\mathbb R^3_v}|v|^4M {\rm d}v=15,
\end{gather*}
We deduce from   $F = M+\sqrt{M}f$ and the decompositions \eqref{G2.5}--\eqref{G2.6} that
\begin{equation}\label{G2.7}
\left\{
\begin{aligned}
&\partial_t(a+3c)+\nabla_x\cdot b=0,
\\
&\partial_t b+\nabla_x(a+5c)
+\nabla_x\cdot
\big\langle v\otimes v\sqrt M,(\mathbf I-\mathbf P)f\big\rangle
=(1+a+3c)(\nabla_{x}\phi+E),\\
&\partial_t(3a+15c)+5\nabla_x\cdot b
+\nabla_x\cdot
\big\langle |v|^2v\sqrt M,\{\mathbf I-\mathbf P\}f\big\rangle
=2b\cdot(\nabla_{x}\phi+E).
\end{aligned}
\right.
\end{equation}
Moreover, the Poisson equation in \eqref{I1.4} gives
\begin{align}\label{G2.8}
\Delta_x\phi
=
\int_{\mathbb R^3_v}(F-M) {\rm d}v
=
\big\langle f,\sqrt M\big\rangle
=
a+3c .
\end{align}

For later use, inspired by \cite{D-Physical-2009}, we rewrite the macroscopic balance system in a more convenient form. Set $\rho_f:=a+3c $, 
and define
\begin{align*}
\Theta_{ij}(g):=
\big\langle v_iv_j\sqrt M,g\big\rangle,
\qquad
\Xi_i(g):=
\big\langle |v|^2v_i\sqrt M,g\big\rangle . 
\end{align*}
Then the first two equations in
\eqref{G2.7}, together with \eqref{G2.8}, can be written as
\begin{equation}\label{G2.9}
\left\{
\begin{aligned}
&\partial_t\rho_f+\nabla_x\cdot b=0,
\\
&\partial_t b+\nabla_x(\rho_f+2c)
+\nabla_x\cdot\Theta(\{\mathbf I-\mathbf P\}f)-\nabla_{x}\phi
=
E+\rho_f(\nabla_{x}\phi+E).
\end{aligned}
\right.
\end{equation}
Notice also that the first and third equations in \eqref{G2.7} imply the following
balance law for the temperature coefficient \(c\):
\begin{equation}\label{G2.10}
\partial_t c+\frac13\nabla_x\cdot b+\frac16\nabla_x\cdot\Xi(\{\mathbf I-\mathbf P\}f)
=
\frac13 b\cdot(\nabla_{x}\phi+E).
\end{equation}
 
Similar to Kawashima’s hyperbolic-parabolic dissipation estimates \cite{SK-HMJ-1985}, in the macro-micro analysis of the VPB system, apart from the moment functions \(\Theta\) and \(\Xi\), we also require the following higher-order velocity moments of the microscopic component:
\begin{align}\label{G2.11}
A_{ij}(g):=
\big\langle (v_iv_j-\delta_{ij})\sqrt M,g\big\rangle,
\qquad
B_i(g):=
\frac1{10}\big\langle (|v|^2-5)v_i\sqrt M,g\big\rangle.    
\end{align}
where \(1\le i,j\le3\). 
The perturbation equation may be rewritten in the form
\begin{equation}\label{G2.12}
\partial_t\mathbf P f+v\cdot\nabla_x\mathbf P f-(\nabla_{x}\phi+E)\cdot v\sqrt M
=
-\partial_t \{\mathbf I-\mathbf P\}f+\mathcal R+\mathcal{G},
\end{equation}
where
\begin{equation*}
\left\{
\begin{aligned}
&\mathcal R
=
-v\cdot\nabla_x \{\mathbf I-\mathbf P\}f+\mathcal L\{\mathbf I-\mathbf P\}f,
\\
&\mathcal{G}=\Gamma(f,f)
-(\nabla_{x}\phi+E)\cdot\nabla_v f
+\frac12[v\cdot(\nabla_{x}\phi+E)]f.
\end{aligned}
\right.
\end{equation*}
The term \(-(\nabla_{x}\phi+E)\cdot v\sqrt M\) has no contribution
to the \(A_{ij}\)- and \(B_i\)-moments, since
\begin{align*}
A_{ij}((\nabla_{x}\phi+E)\cdot v\sqrt M)=0,
\qquad
B_i((\nabla_{x}\phi+E)\cdot v\sqrt M)=0.    
\end{align*}

Applying \(A_{ij}\) and \(B_i\) to \eqref{G2.12}, and using the decomposition  $\mathbf P f=(a+b\cdot v+c|v|^2)\sqrt M$,
we obtain
\begin{equation}\label{G2.13}
\left\{
\begin{aligned}
&\partial_t\left[A_{ii}(\{\mathbf I-\mathbf P\}f)+2c\right]+2\partial_i b_i
=
A_{ii}(\mathcal R+\mathcal{G}),
\qquad 1\le i\le3,
\\
&\partial_tA_{ij}(\{\mathbf I-\mathbf P\}f)+\partial_i b_j+\partial_j b_i
=
A_{ij}(\mathcal R+\mathcal G),
\qquad\, 1\le i\ne j\le3,\\
&\partial_tB_i(\{\mathbf I-\mathbf P\}f)+\partial_i c
=
B_i(\mathcal R+\mathcal{G}),
\qquad\qquad\qquad \,\, 1\le i\le3.
\end{aligned}
\right.
\end{equation}
Consequently, for each fixed \(m\in\{1,2,3\}\), it follows from
\eqref{G2.13}$_1$--\eqref{G2.13}$_2$ that
\begin{equation}\label{G2.14}
\begin{aligned}
&-\partial_t\left[
\sum_{j=1}^3\partial_jA_{jm}(\{\mathbf I-\mathbf P\}f)
+\frac12\partial_mA_{mm}(\{\mathbf I-\mathbf P\}f)
\right]
-\Delta_x b_m-\partial_m^2b_m
\\
&\qquad
=
\frac12\sum_{j\ne m}\partial_mA_{jj}(\mathcal R+\mathcal{G})
-\sum_{j=1}^3\partial_jA_{jm}(\mathcal R+\mathcal{G}).
\end{aligned}
\end{equation}
This, which shows an elliptic-type relation for \(b\), along with \eqref{G2.13}$_3$ for \(c\) and \eqref{G2.9}$_1$ for \(\rho_f\), is the basic macroscopic structure employed to recover the spatial dissipation of the fluid part.

\begin{rem}\label{R2.1}
The aforementioned moment system serves as the analogue of the macroscopic moment system employed in the classical VPB macro-micro decomposition. When compared with the standard VPB case in \cite{RS-ARMA-2011}, the only new characteristic is the appearance of the prescribed force \(E\). The self-consistent part \(\nabla_x\phi\) is coupled to the density via $\Delta_x\phi=\rho_f$, and as a result, it plays a role in the recovery of the density dissipation. The external force \(E\), conversely, is not regulated by the Poisson equation and has to be regarded as a given source term. Specifically, the \(E\)-dependent terms emerge in \(\mathcal G\) and on the right-hand side of the momentum equation in \eqref{G2.12}; they do not change the algebraic structure of the higher-order moment equations \eqref{G2.13}--\eqref{G2.14}.
\end{rem}

\subsection{The Littlewood-Paley decomposition and Besov spaces}
We briefly recall the Littlewood-Paley decomposition and the basic facts on homogeneous Besov spaces, which will be used throughout this paper. For standard details, interested readers can refer to \cite[Chapter 2]{BCD-Book-2011}.
Let \(\chi\in C_c^\infty(\mathbb R^3)\) be a radial non-increasing function such that
\begin{align*}
0\leq \chi\leq 1;\qquad
\chi(\xi)\equiv1\quad\text{for}\quad|\xi|\leq \frac34;
\qquad
\operatorname{supp}\chi\subset \bigg\{|\xi|\leq \frac43\bigg\}.    
\end{align*}
By setting $\varphi(\xi):=\chi\big(\frac{\xi}{2}\big)-\chi(\xi)$, we have
\begin{align*}
\operatorname{supp}\varphi
\subset
\bigg\{\xi\in\mathbb R^3:\frac34\leq |\xi|\leq \frac83\bigg\};
\qquad
\sum_{j\in\mathbb Z}\varphi(2^{-j}\xi)=1,
\quad\text{for}\quad \xi\neq 0.    
\end{align*}

For \(j\in\mathbb Z\), the homogeneous dyadic block is defined by
\begin{align*}
\dot\Delta_j g
:=
\mathcal F^{-1}\left(\varphi(2^{-j}\xi)\widehat g(\xi)\right).    
\end{align*}
Obviously, if \(h=\mathcal F^{-1}\varphi\), then it holds 
\begin{align*}
\dot\Delta_j g
=
2^{3j}h(2^j\cdot)*g .    
\end{align*}
We also introduce the low-frequency cut-off operator
\begin{align*}
\dot S_jg:=\sum_{k<j}\dot\Delta_k g.    
\end{align*}
In particular, for a fixed integer \(j_0\), we shall use the decomposition
\begin{align*}
g^L:=\dot S_{j_0}g,
\qquad
g^H:=g-g^L .    
\end{align*}
We use the homogeneous realization $\mathcal S'_h(\mathbb R^3)=\big\{g\in\mathcal S'(\mathbb R^3)\big|\,\lim\limits_{j\to-\infty}\|\dot S_jg\|_{L^\infty}=0\big\}$.
%Let \(\mathcal{P}\) be the space of all polynomials on \(\mathbb{R}^3\), then we can work in the usual homogeneous tempered distribution space \(\mathcal{S}'_h(\mathbb{R}^3) := \mathcal{S}'(\mathbb{R}^3)/\mathcal{P}\).
Moreover, the dyadic blocks satisfy the almost orthogonality property
\begin{align*}
\dot\Delta_j\dot\Delta_k g=0,
\qquad |j-k|\geq 2.    
\end{align*}
 
We now define the homogeneous Besov spaces. 
For \(s\in\mathbb R\) and \(1\leq p,q\leq\infty\), we set
\begin{align*}
\dot B^s_{p,q}
:=
\Big\{
g\in\mathcal S'_h(\mathbb R^3)\,\,\big|\,\,
\|g\|_{\dot B^s_{p,q}}<\infty
\Big\},    
\end{align*}
where
\begin{align*}
\|g\|_{\dot B^s_{p,q}}
:=
\bigg\|
\Big\{
2^{js}\|\dot\Delta_jg\|_{L^p}
\Big\}_{j\in\mathbb Z}
\bigg\|_{\ell^q}.    
\end{align*}
For functions depending on both \(x\) and \(v\), we use the mixed space $L^2_v(\dot B^s_{p,q})
:=
L^2\big(\mathbb R^3_v;\dot B^s_{p,q}(\mathbb R^3_x)\big)$, 
equipped with the norm
\begin{align*}
 \|g\|_{L^2_v(\dot B^s_{p,q})}
:=
\bigg(
\int_{\mathbb R^3_v}
\|g(\cdot,v)\|_{\dot B^s_{p,q}}^2 {\rm d}v
\bigg)^{1/2}.   
\end{align*}

Next, we collect several standard properties of homogeneous Besov spaces.
 
\begin{prop}{\rm(\!\!\cite[Chapter 2]{BCD-Book-2011})}\label{prop2.1}
The following statements hold.

\begin{itemize}
\item For any integer \(m\geq0\), one has
\begin{align*}
\|\nabla^m g\|_{\dot B^s_{2,q}}
\backsim
\|g\|_{\dot B^{s+m}_{2,q}}.    
\end{align*}
\item Let \(s\in\mathbb R\), \(1\leq p_1\leq p_2\leq\infty\), and
\(1\leq q_1\leq q_2\leq\infty\). Then
\begin{align*}
\dot B^s_{p_1,q_1}
\hookrightarrow
\dot B^{s-3(\frac1{p_1}-\frac1{p_2})}_{p_2,q_2}.    
\end{align*}

\item If \(1\leq p\leq q\leq\infty\), then
\begin{align*}
\dot B^0_{p,1}
\hookrightarrow
L^p
\hookrightarrow
\dot B^0_{p,\infty}
\hookrightarrow
\dot B^{\sigma^\prime}_{q,\infty},
\qquad
\sigma^\prime=-3\Big(\frac1p-\frac1q\Big).    
\end{align*}

\item If \(p<\infty\), then
\begin{align*}
\dot B^{\frac3p}_{p,1}
\hookrightarrow
\mathcal C_0(\mathbb R^3),    
\end{align*}
where \(
\mathcal C_0(\mathbb R^3)\) denotes the space of continuous functions vanishing at infinity.

\item Let \(1\leq p,q\leq\infty\), \(s_1<s_2\), and \(0<\theta<1\). Then
\begin{align}\label{G2.15}
\|g\|_{\dot B^{\theta s_1+(1-\theta)s_2}_{p,q}}\lesssim\|g\|_{\dot B^{\theta s_1+(1-\theta)s_2}_{p,1}}
\lesssim
\frac{1}{\theta(1-\theta)(s_2-s_1)}
\|g\|_{\dot B^{s_1}_{p,\infty}}^\theta
\|g\|_{\dot B^{s_2}_{p,\infty}}^{1-\theta}.    
\end{align}

%\item Let \(s\in\mathbb R\), \(1\leq q\leq\infty\), and let \(q'\) be the conjugate
%exponent of \(q\). Then
%\begin{align}\label{duality}
%|\langle g,h\rangle|
%\lesssim
%\|g\|_{\dot B^s_{2,q}}
%\|h\|_{\dot B^{-s}_{2,q'}}.
%\end{align}
%Moreover,
%\[
%\|g\|_{\dot B^s_{2,q}}
%\lesssim
%\sup_{\phi}
%|\langle g,\phi\rangle|,
%\]
%where the supremum is taken over all \(\phi\in\mathcal S(\mathbb R^3)\) satisfying
%\[
%\|\phi\|_{\dot B^{-s}_{2,q'}}\leq 1,
%\qquad
%0\notin \operatorname{supp}\widehat\phi .
%\]

\item For $s>0$ and \(\varepsilon>0\), one has
\begin{align*}
H^{s+\varepsilon}
\hookrightarrow
\dot B^s_{2,1}
\hookrightarrow
\dot H^s .    
\end{align*}

\item For \(\sigma\in\mathbb R\), define
\begin{align*}
\Lambda^\sigma g
:=
(-\Delta)^{\frac\sigma2}g
=
\mathcal F^{-1}\left(|\xi|^\sigma\widehat g(\xi)\right).    
\end{align*}
Then \(\Lambda^\sigma\) is an isomorphism from \(\dot B^s_{p,q}\) onto
\(\dot B^{s-\sigma}_{p,q}\).

\item Let \(1\leq q_1,q_2\leq\infty\), \(s_1,s_2\in\mathbb R\), and assume that
\begin{align*}
s_2<\frac32,
\qquad\text{or}\qquad
s_2=\frac32\quad \text{ and }\quad\ q_2=1.    
\end{align*}
 
Then the space $
\dot B^{s_1}_{2,q_1}\cap \dot B^{s_2}_{2,q_2}$
endowed with the norm
\begin{align*}
\|g\|_{\dot B^{s_1}_{2,q_1}\cap \dot B^{s_2}_{2,q_2}}
:=
\|g\|_{\dot B^{s_1}_{2,q_1}}
+
\|g\|_{\dot B^{s_2}_{2,q_2}}    
\end{align*}
is a Banach space and enjoys the Fatou property. More precisely, if
\(\{g_n\}_{n\geq1}\) is bounded in this space, then there exist \(g\in
\dot B^{s_1}_{2,q_1}\cap \dot B^{s_2}_{2,q_2}\) and a subsequence
\(\{g_{n_k}\}_{k\geq1}\) such that $
g_{n_k}\to g$ in $\mathcal S'(\mathbb R^3)$,
and
\begin{align*}
\|g\|_{\dot B^{s_1}_{2,q_1}\cap \dot B^{s_2}_{2,q_2}}
\lesssim
\liminf_{k\to\infty}
\|g_{n_k}\|_{\dot B^{s_1}_{2,q_1}\cap \dot B^{s_2}_{2,q_2}}.    
\end{align*}
\end{itemize}
\end{prop}

We shall also need the following product estimates in Besov spaces.
\begin{prop}{\rm(\cite[Chapter 2]{BCD-Book-2011})}\label{prop2.2}
The following estimates hold.

\begin{itemize}
\item Let \(s>0\) and \(1\leq q\leq\infty\). Then
\(\dot B^s_{2,q}\cap L^\infty\) is an algebra, and
\begin{align*}
\|g_1g_2\|_{\dot B^s_{2,q}}
\lesssim
\|g_1\|_{L^\infty}\|g_2\|_{\dot B^s_{2,q}}
+
\|g_2\|_{L^\infty}\|g_1\|_{\dot B^s_{2,q}}.    
\end{align*}
In particular,
\begin{align*}
\|g_1g_2\|_{\dot B^{\frac32}_{2,1}}
\lesssim
\|g_1\|_{\dot B^{\frac32}_{2,1}}
\|g_2\|_{\dot B^{\frac32}_{2,1}}.    
\end{align*}

\item Let \(s_1,s_2\in\mathbb R\)   satisfy
\begin{align*}
s_1<\frac32,\qquad
s_2<\frac32,\qquad
s_1+s_2>0.    
\end{align*}
Then it holds that
\begin{align}\label{G2.16}
\|g_1g_2\|_{\dot B^{s_1+s_2-\frac32}_{2,\infty}}
\lesssim
\|g_1\|_{\dot B^{s_1}_{2,\infty}}
\|g_2\|_{\dot B^{s_2}_{2,\infty}} .
\end{align}

\item Let \(s_1,s_2\in\mathbb R\) satisfy
\begin{align*}
 s_1\leq \frac32,\qquad
s_2<\frac32,\qquad
s_1+s_2\geq0.   
\end{align*}
Then
\begin{align}\label{G2.17}
\|g_1g_2\|_{\dot B^{s_1+s_2-\frac32}_{2,\infty}}
\lesssim
\|g_1\|_{\dot B^{s_1}_{2,1}}
\|g_2\|_{\dot B^{s_2}_{2,\infty}} .
\end{align}
\end{itemize}
\end{prop}
\subsection{Linearized VPB semi-group estimates}

In this subsection, we recall and adapt the linearized VPB semi-group estimates needed for the nonlinear analysis. Our approach differs from the classical spectral analysis for the linearized Boltzmann equations
\cite{D-2011-Nonlinearity,Ukai-1974,Ukai-1976,EP-JMPA-1975}. It is based on the energy-spectrum framework for the VPB system developed in \cite{RS-ARMA-2011}, combined with a low-high frequency localization in the spirit of \cite{DLN-2026}.
Compared with the estimates of the pure Boltzmann semi-group, the VPB semi-group includes the field component $\nabla_x\Delta_x^{-1}\mathbf P_0 f$, which corresponds to the self-consistent Poisson field.  

%\textcolor{red}{
%In contrast to the approach for the Boltzmann equation in \cite{DN-2026},  we do not treat the Duhamel terms by duality argument. Instead, we employ a dyadically localized energy method. At each Littlewood-Paley level, the low-frequency part is governed by the diffusive kernel 
%\(e^{-c2^{2j}t}\), whereas the high-frequency part is controlled by the spectral-gap kernel 
%\(e^{-ct}\).}

We now consider the following nonhomogeneous linearized  VPB Cauchy problem in $\mathbb{R}^3$: 
\begin{equation}\label{linear}
\left\{
\begin{aligned}
& \partial_t f+v\cdot\nabla_x f-\mathcal{L}f-\nabla_{x}\phi\cdot v\sqrt M
= h,    \\
&\Delta_{x}\phi= \int_{\mathbb R^3}\sqrt{M}f{\rm d}v, \quad(t,x,v)\in \mathbb R^+\times\mathbb R^3_{x}\times \mathbb R_{v}^3,\\
&f(0,x,v)=f_0(x,v)=\frac{F_0(x,v)-M}{\sqrt{M}}, \quad(x,v)\in  \mathbb R^3_{x}\times \mathbb R_{v}^3, 
\end{aligned}
\right.
\end{equation}
where $h=h(t,x,v)$ and $f_0=f_0(x,v)$ are given.
Hence, the mild solution to the Cauchy problem \eqref{linear} is 
\begin{align}\label{LVPB-mild}
f(t)=e^{t\mathcal B}f_0+\int_0^t e^{(t-s)\mathcal B}h(s)\,{\rm d}s,
\end{align}
where \(\mathcal B\) denotes the linear VPB generator:
\begin{align}\label{DN2.20}
\mathcal{B}f   =-v\cdot\nabla_xf+\mathcal Lf+
 \nabla_x\Delta_x^{-1}{\langle f,\sqrt{M}\rangle}\cdot v\sqrt M. 
\end{align}

The full-frequency time-decay theory for the VPB system was developed in \cite{RS-ARMA-2011} via Fourier energy estimates for the solution semi-group.  The corresponding estimates for the solution $f$ given in \eqref{LVPB-mild} are recalled below.
\begin{prop}[{\!\!\cite[Theorem 2]{RS-ARMA-2011}}]\label{Prop2.3}
Let $1\leq q\leq 2$, and define $\sigma_{q,m}=\frac{3}{2}\big(\frac{1}{q}-\frac{1}{2}\big)+\frac{m}{2}$.
\begin{itemize}
    \item For any $\alpha$, $\alpha^\prime$ with $\alpha^\prime\leq \alpha$, and for any $f_0$ satisfying $\partial^\alpha_{x}f_0\in L_{x,v}^2$ and $\partial^{\alpha^\prime}_{x}f_0\in\mathcal{Z}_q$, one has 
\begin{align}\label{ARMA-full-1}
\|\partial_x^\alpha e^{t\mathcal B}f_0\|_{L^2_{x,v}}
+
\|\partial_x^\alpha\nabla_x\Delta_x^{-1}\mathbf P_0e^{t\mathcal B}f_0\|_{L^2_x}
\leq C
(1+t)^{-\sigma_{q,m-1}}
\big(
\|\partial_x^{\alpha'}f_0\|_{\mathcal Z_q}
+
\|\partial_x^\alpha f_0\|_{L^2_{x,v}}
\big),
\end{align}
and
\begin{align}\label{ARMA-full-2}
&\|\partial_x^\alpha e^{t\mathcal B}\{\mathbf I-\mathbf P_0\}f_0\|_{L^2_{x,v}}
+
\|\partial_x^\alpha\nabla_x\Delta_x^{-1}\mathbf P_0
e^{t\mathcal B}\{\mathbf I-\mathbf P_0\}f_0\|_{L^2_x}
\nonumber\\
&\quad
\leq C
(1+t)^{-\sigma_{q,m}}
\big(
\|\partial_x^{\alpha'}\{\mathbf I-\mathbf P_0\}f_0\|_{\mathcal Z_q}
+
\|\partial_x^\alpha\{\mathbf I-\mathbf P_0\}f_0\|_{L^2_{x,v}}
\big),   
\end{align}
for any $t\geq 0$ with $m=|\alpha-\alpha^\prime|$, where $C$ is a positive constant depending only on $n,m,q$.
 \item  Similarly, for any $\alpha,\alpha^\prime$ with
$\alpha^\prime\leq \alpha$, and for any $h$ which satisfies $\nu(v)^{-1/2}\partial^\alpha_{x}h(t)\in L^2_{x,v}$ and 
$\nu(v)^{-1/2}\partial^\alpha_{x}h(t)\in  \mathcal{Z}_q$, one has
\begin{align*}
&\bigg\|\partial_x^\alpha
\int_0^t e^{(t-s)\mathcal B}\{\mathbf I-\mathbf P\}h(s){\rm d}s
 \bigg\|_{L^2_{x,v}}^2
+
\bigg\|
\partial_x^\alpha\nabla_x\Delta_x^{-1}\mathbf P_0
\int_0^t e^{(t-s)\mathcal B}\{\mathbf I-\mathbf P\}h(s){\rm d}s
\bigg\|_{L^2_x}^2
\\
&\quad
\le
C\int_0^t(1+t-s)^{-2\sigma_{q,m}}
\big(
\|\nu^{-1/2}\partial_x^{\alpha'}\{\mathbf I-\mathbf P\}h(s)\|_{Z_q}^2
+
\|\nu^{-1/2}\partial_x^\alpha\{\mathbf I-\mathbf P\}h(s)\|_{L^2_{x,v}}^2
\big) {\rm d}s,
\end{align*}
for $t\geq0$ with $m=|\alpha-\alpha^\prime|$, where $C$ is a positive constant depending only on $n,m,q$.
 
\end{itemize}
   
\end{prop}

\begin{rem}
In the case of the Boltzmann equation, algebraic time-decay estimates for the
solution semi-group have been obtained by spectral analysis; see
\cite{Ukai-1974}. Similar decay
estimates are also available for related fluid-type systems, such as the Navier-Stokes-Poisson  system; see \cite{LMZ-ARMA-2010}. For our VPB system,
however, the coupling with the Poisson equation produces a different low-frequency structure. The semi-group estimate must therefore include the field
component $
\nabla_x\Delta_x^{-1}\mathbf P_0f$,
which is absent in the pure Boltzmann equation.
This explains the one-half order loss in \eqref{ARMA-full-1}. The loss is caused
by the density mode \(\mathbf P_0f\), since the field multiplier
\(\nabla_x\Delta_x^{-1}\) behaves like \(|\xi|^{-1}\) at low frequencies. When the
\(\mathbf P_0\)-mode is removed, the singular field mode disappears and the
better estimate \eqref{ARMA-full-2} follows. Thus the slower decay in
\eqref{ARMA-full-1} is intrinsic to the VPB coupling, not an artifact of the method.
\end{rem}

Based on Proposition \ref{Prop2.3}, we now derive a frequency-localized version of the VPB semi-group estimates. Different from the treatment of the
Boltzmann equation in \cite{DN-2026}, the Duhamel terms here are not estimated
through a duality argument.
Instead, we work directly at each Littlewood-Paley
level and employ a dyadically localized energy method, which allows us to separate
the low-frequency diffusive decay, governed by the kernel \(e^{-c2^{2j}t}\), from
the high-frequency spectral-gap decay, controlled by the kernel \(e^{-ct}\).
 
\begin{prop}\label{P2.4}
Let $t\geq 0$. Let $s,s_0\in \mathbb{R}$ such that $s_0\leq s$. Then, there exists $j_0\in \mathbb{Z}$ such that, for any   $f_0\in L_v^2(\dot B_{2,\infty}^{s_0})$ with $\nabla_x\phi_0=\nabla_x\Delta_x^{-1}\langle f_0,\sqrt M\rangle\in\dot B_{2,\infty}^{s_0}$, the following estimates hold:
\begin{align}\label{G2.22}
&\|\dot S_{j_0}e^{t\mathcal{B}} f_0\|_{L_v^2(\dot B_{2,\infty}^s)}+\|\nabla_{x}\Delta^{-1}_{x}\mathbf{P}_0\dot S_{j_0}e^{t\mathcal{B}} f_0\|_{L_v^2(\dot B_{2,\infty}^s)} \lesssim (1+t)^{-\frac{s - s_0}{2}} \big(\|f_0\|_{L_v^2(\dot B_{2,\infty}^{s_0})}+\|\nabla_{x}\phi_0\|_{\dot B_{2,\infty}^{s_0}}\big) ,
\end{align}
and 
\begin{align}\label{G2.23}
&\|\dot S_{j_0}e^{t\mathcal{B}} \{\mathbf{I}-\mathbf{P}_0\}f_0\|_{L_v^2(\dot B_{2,\infty}^s)}+\|\nabla_{x}\Delta^{-1}_{x}\mathbf{P}_0\dot S_{j_0}e^{t\mathcal{B}}\{\mathbf{I}-\mathbf{P}_0\} f_0\|_{L_v^2(\dot B_{2,\infty}^s)}\nonumber\\
&\quad\lesssim (1+t)^{-\frac{s - s_0}{2}}  \|\{\mathbf{I}-\mathbf{P}_0\}f_0\|_{L_v^2(\dot B_{2,\infty}^{s_0})}.
\end{align}
For any  $f_0\in L_v^2(\dot B_{2,\infty}^{s})$, we have
\begin{align}\label{G2.24}
&\|({\rm Id} - \dot S_{j_0})e^{t\mathcal{B}} f_0\|_{L_v^2(\dot B_{2,\infty}^s)}+\|\nabla_{x}\Delta^{-1}_{x}\mathbf{P}_0({\rm Id} - \dot S_{j_0})e^{t\mathcal{B}} f_0\|_{L_v^2(\dot B_{2,\infty}^s)} \lesssim e^{-\lambda_1t} \|f_0\|_{L_v^2(\dot B_{2,\infty}^{s})} ,
\end{align}
and
\begin{align}\label{G2.25}
&\|({\rm Id} - \dot S_{j_0})e^{t\mathcal{B}} \{\mathbf{I}-\mathbf{P}_0\} f_0\|_{L_v^2(\dot B_{2,\infty}^s)}+\|\nabla_{x}\Delta^{-1}_{x}\mathbf{P}_0({\rm Id} - \dot S_{j_0})e^{t\mathcal{B}}\{\mathbf{I}-\mathbf{P}_0\} f_0\|_{L_v^2(\dot B_{2,\infty}^s)} \nonumber\\
&\quad\lesssim e^{-\lambda_1t} \|\{\mathbf{I}-\mathbf{P}_0\}f_0\|_{L_v^2(\dot B_{2,\infty}^{s})} ,
\end{align}
where $\lambda_1>0$ is a constant.

In addition, for the inhomogeneous part,   it holds that
\begin{align}\label{G2.26}
 &\bigg\|\int_0^t \dot S_{j_0}e^{(t-\tau)\mathcal{B}}\{\mathbf{I}-\mathbf{P}\}h(\tau){\rm d}\tau\bigg\|_{L_v^2(\dot B_{2,\infty}^s)}+ \bigg\|\nabla_{x}\Delta^{-1}_{x}\mathbf{P}_0\int_0^t \dot S_{j_0}e^{(t-\tau)\mathcal{B}}\{\mathbf{I}-\mathbf{P}\}h(\tau){\rm d}\tau\bigg\|_{L_v^2(\dot B_{2,\infty}^s)}\nonumber\\
 &\quad\lesssim \sup_{0\leq \tau\leq t}\|\nu^{-{1}/{2}} \{\mathbf{I}-\mathbf{P}\}h(\tau)\|_{L_v^2(\dot B_{2,\infty}^{s-1})},
\end{align}
for any $\nu^{-1/2}h(t)\in L_v^2(\dot B_{2,\infty}^{s-1})$,
and
\begin{align}\label{G2.27}
 &\bigg\|\int_0^t  ({\rm Id}-\dot S_{j_0})e^{(t-\tau)\mathcal{B}}\{\mathbf{I}-\mathbf{P}\}h(\tau){\rm d}\tau\bigg\|_{L_v^2(\dot B_{2,\infty}^s)}+ \bigg\|\nabla_{x}\Delta^{-1}_{x}\mathbf{P}_0\int_0^t  ({\rm Id}-\dot S_{j_0})e^{(t-\tau)\mathcal{B}}\{\mathbf{I}-\mathbf{P}\}h(\tau){\rm d}\tau\bigg\|_{L_v^2(\dot B_{2,\infty}^s)}\nonumber\\
 &\quad\lesssim \sup_{0\leq \tau\leq t}\|\nu^{-{1}/{2}} \{\mathbf{I}-\mathbf{P}\}h(\tau)\|_{L_v^2(\dot B_{2,\infty}^{s})},
\end{align}
for any $\nu^{-1/2}h(t)\in L_v^2(\dot B_{2,\infty}^{s})$. Moreover, it holds that
\begin{align}\label{DNG2.28}
 &\bigg\|\int_0^t \dot S_{j_0}e^{(t-\tau)\mathcal{B}}\{\mathbf{I}-\mathbf{P}_0\} h(\tau){\rm d}\tau\bigg\|_{L_v^2(\dot B_{2,\infty}^s)}+ \bigg\|\nabla_{x}\Delta^{-1}_{x}\mathbf{P}_0\int_0^t \dot S_{j_0}e^{(t-\tau)\mathcal{B}}\{\mathbf{I}-\mathbf{P}_0\}  h(\tau){\rm d}\tau\bigg\|_{L_v^2(\dot B_{2,\infty}^s)}\nonumber\\
&\quad \lesssim \sup_{0\leq \tau\leq t}\| \{\mathbf{I}-\mathbf{P}_0\}h(\tau)\|_{L_v^2(\dot B_{2,\infty}^{s-2})},
\end{align}
for any $h(t)\in L_v^2(\dot B_{2,\infty}^{s-2})$,
and
\begin{align}\label{DNG2.29}
 &\bigg\|\int_0^t  ({\rm Id}-\dot S_{j_0})e^{(t-\tau)\mathcal{B}} \{\mathbf{I}-\mathbf{P}_0\}h(\tau){\rm d}\tau\bigg\|_{L_v^2(\dot B_{2,\infty}^s)}+ \bigg\|\nabla_{x}\Delta^{-1}_{x}\mathbf{P}_0\int_0^t  ({\rm Id}-\dot S_{j_0})e^{(t-\tau)\mathcal{B}} \{\mathbf{I}-\mathbf{P}_0\}h(\tau){\rm d}\tau\bigg\|_{L_v^2(\dot B_{2,\infty}^s)}\nonumber\\
 &\quad\lesssim \sup_{0\leq \tau\leq t}\|\{\mathbf{I}-\mathbf{P}_0\} h(\tau)\|_{L_v^2(\dot B_{2,\infty}^{s})},
\end{align}
for any $h(t)\in L_v^2(\dot B_{2,\infty}^{s})$. 
\end{prop}

\begin{proof}
Applying \(\dot\Delta_j\) to \eqref{linear} yields
\begin{equation*} 
\left\{
\begin{aligned}
& \partial_t\dot\Delta_{j}f+v\cdot\nabla_x\dot\Delta_{j} f-\mathcal L\dot\Delta_{j} f-\nabla_x{\dot\Delta_{j}}\phi\cdot v\sqrt M
=\dot\Delta_{j}h,\\
&\Delta_x\dot\Delta_{j}\phi = \langle\dot\Delta_{j} f,\sqrt M \rangle.\\
\end{aligned}
\right.
\end{equation*}
We now choose a dyadic functional $\mathcal{E}_{j}(t)$ satisfying 
\begin{align*}
\mathcal E_j(t)\, {\sim}\,
\|\dot\Delta_{j}f(t)\|_{L^2_{x,v}}^2
+\|\nabla_x\Delta_x^{-1}\mathbf P_0\dot\Delta_{j}f (t)\|_{L^2_{x,v}}^2.    
\end{align*}
As in the energy-spectrum method for the linearized VPB system \cite[Section 3]{RS-ARMA-2011}, one adds the usual compensating macroscopic cross-terms and deduces the following Lyapunov inequality for sources satisfying $\mathbf P h=0$:
\begin{align}\label{G2.28}
\frac{{\rm d}}{{\rm d}t}\mathcal E_j(t)
+\kappa_0\min\{1,2^{2j}\} \mathcal E_j(t)
\lesssim \|\nu^{-{1}/{2}}\{\mathbf{I}-\mathbf{P}\}\dot\Delta _{j}h\|_{L_{x,v}^2}^2,    
\end{align}
for some constant $\kappa_0>0$.
When $\dot\Delta_{j}h = 0$, it implies that $\nu^{-1/2}\{\mathbf{I}-\mathbf{P}\}\dot\Delta_{j}h = 0$. Therefore, from \eqref{G2.28}, we have
\begin{align}\label{G2.29}
\| \dot\Delta_j e^{t\mathcal B}f_0\|_{L^2_{x,v}}
+\|\nabla_x\Delta_x^{-1}\mathbf P_0\dot\Delta_j e^{t\mathcal B}f_0\|_{L^2_{x,v}}\lesssim
e^{-\frac{\kappa_0}2\min\{1,2^{2j}\} t}
\big(
\|\dot\Delta_j f_0\|_{L^2_{x,v}}
+
\|\nabla_x\Delta_x^{-1}\mathbf P_0\dot\Delta_j f_0\|_{L^2_{x,v}}
\big).    
\end{align}

On the one hand, for $j\leq j_0$, the Poisson term satisfies
\begin{align*}
\|\nabla_x\Delta_x^{-1}\mathbf P_0\dot\Delta_j f_0\|_{L^2_{x,v}}\backsim {\|\nabla_{x}\dot \Delta_{j}\phi_0\|_{L^2}},    
\end{align*}
which together with \eqref{G2.29} yields
\begin{align*}
\| \dot\Delta_j e^{t\mathcal B}f_0\|_{L^2_{x,v}}
+\|\nabla_x\Delta_x^{-1}\mathbf P_0\dot\Delta_j e^{t\mathcal B}f_0\|_{L^2_{x,v}}\lesssim&\,  e^{-\frac{\kappa_0}{2}2^{2j}t}\big(\|\dot\Delta_j f_0\|_{L^2_{x,v}}+\|\nabla_{x}\dot\Delta_{j}\phi_0\|_{L^2}\big) \nonumber\\
\lesssim&\, e^{-\frac{\kappa_0}{2}2^{2j}t}2^{-s_0j} \big(\|f_0\|_{L^2_v(\dot B^{s_0}_{2,\infty})}+\|\nabla_{x}\phi_0\|_{ \dot B^{s_0}_{2,\infty}}\big).
\end{align*}
Thus, we further get
\begin{align*}
2^{js}
\big(
\| \dot\Delta_j e^{t\mathcal B}f_0\|_{L^2_{x,v}}
+\|\nabla_x\Delta_x^{-1}\mathbf P_0\dot\Delta_j e^{t\mathcal B}f_0\|_{L^2_{x,v}}
\big) \lesssim&\, 2^{j(s-s_0)}e^{-c2^{2j}t}\big(\|f_0\|_{L^2_v(\dot B^{s_0}_{2,\infty})}+\|\nabla_{x}\phi_0\|_{ \dot B^{s_0}_{2,\infty}}\big)\nonumber\\
\lesssim&\, (1+t)^{-\frac{s-s_0}{2}}\big(\|f_0\|_{L^2_v(\dot B^{s_0}_{2,\infty})}+\|\nabla_{x}\phi_0\|_{ \dot B^{s_0}_{2,\infty}}\big),
\end{align*}
which gives rise to \eqref{G2.22}.

On the other hand, when $j > j_0$, by applying \eqref{G2.29} and using
$\|\nabla_x\Delta_x^{-1}\mathbf P_0\dot\Delta_jf_0\|_{L^2_{x,v}}
\lesssim 2^{-j}\|\dot\Delta_jf_0\|_{L^2_{x,v}}$, one has
\begin{align*}
2^{js}
\big(
\|\dot\Delta_j e^{t\mathcal B}f_0\|_{L^2_{x,v}}
+\|\nabla_x\Delta_x^{-1}\mathbf P_0\dot\Delta_j e^{t\mathcal B}f_0\|_{L^2_{x,v}}
\big)
\lesssim&\,
e^{-\frac{\kappa_0}{2} t}
2^{js}\|\dot\Delta_j f_0\|_{L^2_{x,v}}\nonumber\\
\lesssim&\, 
e^{-\frac{\kappa_0}{2} t} {\|f_0\|_{L^2_v(\dot B^{s}_{2,\infty})}},
\end{align*}
which leads to \eqref{G2.24}.
In particular, if the initial datum is \(\{\mathbf I - \mathbf P_0\}f_0\), then the singular Poisson density mode is absent. Thus, by a similar argument to the proof of \eqref{G2.22} and \eqref{G2.24}, we can directly obtain \eqref{G2.23} and \eqref{G2.25}.

Finally, we prove the remaining equations \eqref{G2.26} and \eqref{G2.27}. For the sake of notational simplicity, we denote
\begin{align*}
\mathcal T(t):=
\int_0^t e^{(t-\tau)\mathcal B}\{\mathbf I-\mathbf P\}h(\tau){\rm d}\tau.    
\end{align*}
It follows from \eqref{G2.28} that
\begin{align}\label{G2.30}
&\|\dot\Delta_{j}\mathcal{T}(t)\|_{L^2_{x,v}}^2
+\|\nabla_x\Delta_x^{-1}\mathbf P_0 \dot\Delta_{j}\mathcal T(t)\|_{L^2_{x,v}}^2\lesssim
\int_0^t e^{-\kappa_1\min\{1,2^{2j}\}(t-\tau)}
\|\nu^{-1/2}\dot\Delta_j\{\mathbf I-\mathbf P\}h(\tau)\|_{L^2_{x,v}}^2
{\rm d}\tau,   
\end{align}
for some constant $\kappa_1>0$. When $j\leq j_0$, taking the square root in \eqref{G2.30}, we get
\begin{align}\label{G2.31}
\|\dot\Delta_{j}\mathcal T(t)\|_{L^2_{x,v}}
+\|\nabla_x\Delta_x^{-1}\mathbf P_0\dot\Delta_{j}\mathcal T(t)\|_{L^2_{x,v}} 
\lesssim&\,
\left(
\int_0^t e^{-\kappa_{1}2^{2j}(t-\tau)}
\|\nu^{-1/2}\dot\Delta_j\{\mathbf I-\mathbf P\}h(\tau)\|_{L^2_{x,v}}^2
{\rm d}\tau
\right)^{1/2}
\nonumber\\
\lesssim&\,
\left(\int_0^t e^{-\kappa_12^{2j}(t-\tau)}\,{\rm d}\tau\right)^{1/2}
\sup_{0\le\tau\le t}
\|\nu^{-1/2}\dot\Delta_j\{\mathbf I-\mathbf P\}h(\tau)\|_{L^2_{x,v}}
\nonumber\\
\lesssim&\,
2^{-j}
\sup_{0\le\tau\le t}
\|\nu^{-1/2}\dot\Delta_j\{\mathbf I-\mathbf P\}h(\tau)\|_{L^2_{x,v}}.    
\end{align}
Multiplying \eqref{G2.31} by \(2^{js}\) and then taking the supremum over all \(j \le j_0\) yields \eqref{G2.26}.
For $j>j_0$, by direct computation, we have
\begin{align}\label{G2.32}
\|\dot\Delta_{j}\mathcal T(t)\|_{L^2_{x,v}}
+\|\nabla_x\Delta_x^{-1}\mathbf P_0\dot\Delta_{j}\mathcal T(t)\|_{L^2_{x,v}}\lesssim&\,
\left(\int_0^t e^{-\kappa_1(t-\tau)}
\|\nu^{-1/2}\dot\Delta_j\{\mathbf I-\mathbf P\}h(\tau)\|_{L^2_{x,v}}^2
{\rm d}\tau\right)^{1/2}\nonumber\\
\lesssim&\,
\sup_{0\le\tau\le t}
\|\nu^{-1/2}\dot\Delta_j\{\mathbf I-\mathbf P\}h(\tau)\|_{L^2_{x,v}}.    
\end{align}
Multiplying \eqref{G2.32} by \(2^{js}\) and taking the supremum over all \(j>j_0\), we obtain \eqref{G2.27}.

To prove \eqref{DNG2.28}--\eqref{DNG2.29}, we set
\begin{align*}
\mathcal U(t):=\int_0^t e^{(t-\tau)\mathcal B}
\{\mathbf I-\mathbf P_0\}h(\tau) {\rm d}\tau,
\qquad \mu_j:=\min\{1,2^{2j}\}.
\end{align*}
Since $\mathbf P_0\{\mathbf I-\mathbf P_0\}h=0$,
the initial Poisson term in \eqref{G2.29} vanishes for each
datum $\{\mathbf I-\mathbf P_0\}h$. From Minkowski's inequality
and \eqref{G2.29}, we have
\begin{align*}
 \|\dot\Delta_j\mathcal U(t)\|_{L^2_{x,v}}
+\|\nabla_x\Delta_x^{-1}\mathbf P_0\dot\Delta_j\mathcal U(t)\|_{L^2_{x,v}}
 \lesssim&\,\int_0^t e^{-c\mu_j(t-\tau)}
\|\dot\Delta_j\{\mathbf I-\mathbf P_0\}h(\tau)\|_{L^2_{x,v}}
 {\rm d}\tau\\
\lesssim&\,\mu_j^{-1}\sup_{0\leq\tau\leq t}
\|\dot\Delta_j\{\mathbf I-\mathbf P_0\}h(\tau)\|_{L^2_{x,v}}.
\end{align*}
Here $\mu_j^{-1}=2^{-2j}$ for $j\leq j_0$ and $\mu_j^{-1}=1$
for $j>j_0$. Multiplying by $2^{js}$ and taking the supremum
over the corresponding dyadic indices gives
\eqref{DNG2.28} and \eqref{DNG2.29}, respectively.

%The proof of \eqref{DNG2.28}--\eqref{DNG2.29} is similar to those of \eqref{G2.26}--\eqref{G2.27}.
Therefore, we complete the proof of Proposition \ref{P2.4}.
\end{proof}

In fact, we can also obtain the time-weighted form of the estimates \eqref{G2.26}--\eqref{DNG2.29}.

\begin{prop}\label{P2.5}
Let $t\geq 0$. Let $s,s_0\in \mathbb{R}$ such that $s_0\leq s$. Then, there exists $j_0\in \mathbb{Z}$ such that
\begin{align}\label{DNnew1}
 &\bigg\|\int_0^t \dot S_{j_0}e^{(t-\tau)\mathcal{B}}\{\mathbf{I}-\mathbf{P}\}h(\tau){\rm d}\tau\bigg\|_{L_v^2(\dot B_{2,\infty}^s)}^2+ \bigg\|\nabla_{x}\Delta^{-1}_{x}\mathbf{P}_0\int_0^t \dot S_{j_0}e^{(t-\tau)\mathcal{B}}\{\mathbf{I}-\mathbf{P}\}h(\tau){\rm d}\tau\bigg\|_{L_v^2(\dot B_{2,\infty}^s)}^2\nonumber\\
 &\quad\lesssim\int_0^t(1+t-\tau)^{-(s-s_0)} \|\nu^{-1/2}\{\mathbf{I}-\mathbf{P}\}h(\tau)\|_{L_v^2(\dot B_{2,\infty}^{s_0})}^2 {\rm d}\tau,
\end{align}
for any $\nu^{-1/2}h(t)\in L_v^2(\dot B_{2,\infty}^{s_0})$,
and
\begin{align}\label{DNnew2}
 &\bigg\|\int_0^t  ({\rm Id}-\dot S_{j_0})e^{(t-\tau)\mathcal{B}}\{\mathbf{I}-\mathbf{P}\}h(\tau){\rm d}\tau\bigg\|_{L_v^2(\dot B_{2,\infty}^s)}^2+ \bigg\|\nabla_{x}\Delta^{-1}_{x}\mathbf{P}_0\int_0^t  ({\rm Id}-\dot S_{j_0})e^{(t-\tau)\mathcal{B}}\{\mathbf{I}-\mathbf{P}\}h(\tau){\rm d}\tau\bigg\|_{L_v^2(\dot B_{2,\infty}^s)}^2\nonumber\\
 &\quad\lesssim \int_0^t e^{-\kappa_1(t-\tau)} \|\nu^{-1/2}\{\mathbf{I}-\mathbf{P}\}h(\tau)\|_{L_v^2(\dot B_{2,\infty}^{s })}^2 {\rm d}\tau,
\end{align}
for any $\nu^{-1/2}h(t)\in L_v^2(\dot B_{2,\infty}^{s})$.  If in addition $s_0\leq s-1$, it holds that
\begin{align}\label{DNnew3}
 &\bigg\|\int_0^t \dot S_{j_0}e^{(t-\tau)\mathcal{B}}\{\mathbf{I}-\mathbf{P}_0\} h(\tau){\rm d}\tau\bigg\|_{L_v^2(\dot B_{2,\infty}^s)}^2+ \bigg\|\nabla_{x}\Delta^{-1}_{x}\mathbf{P}_0\int_0^t \dot S_{j_0}e^{(t-\tau)\mathcal{B}}\{\mathbf{I}-\mathbf{P}_0\}  h(\tau){\rm d}\tau\bigg\|_{L_v^2(\dot B_{2,\infty}^s)}^2\nonumber\\
&\quad\lesssim\int_0^t(1+t-\tau)^{-(s-s_0-1)} \|\{\mathbf{I}-\mathbf{P}_0\}h(\tau)\|_{L_v^2(\dot B_{2,\infty}^{s_0})}^2 {\rm d}\tau,
\end{align}
for any $h(t)\in L_v^2(\dot B_{2,\infty}^{s_0})$,
and
\begin{align}\label{DNnew4}
 &\bigg\|\int_0^t  ({\rm Id}-\dot S_{j_0})e^{(t-\tau)\mathcal{B}} \{\mathbf{I}-\mathbf{P}_0\}h(\tau){\rm d}\tau\bigg\|_{L_v^2(\dot B_{2,\infty}^s)}^2+ \bigg\|\nabla_{x}\Delta^{-1}_{x}\mathbf{P}_0\int_0^t  ({\rm Id}-\dot S_{j_0})e^{(t-\tau)\mathcal{B}} \{\mathbf{I}-\mathbf{P}_0\}h(\tau){\rm d}\tau\bigg\|_{L_v^2(\dot B_{2,\infty}^s)}^2\nonumber\\
 &\quad\lesssim \int_0^t e^{-\kappa_1(t-\tau)} \|\{\mathbf{I}-\mathbf{P}_0\}h(\tau)\|_{L_v^2(\dot B_{2,\infty}^{s })}^2 {\rm d}\tau,
\end{align}
for any $h(t)\in L_v^2(\dot B_{2,\infty}^{s})$. 
\end{prop}
\begin{proof}
For $j\leq j_0$, according to \eqref{G2.30}, we have
\begin{align}\label{J1}
&2^{2js}\|\dot\Delta_{j}\mathcal T(t)\|_{L^2_{x,v}}^2
+2^{2js}\|\nabla_x\Delta_x^{-1}\mathbf P_0\dot\Delta_{j}\mathcal T(t)\|_{L^2_{x,v}} ^2\nonumber\\
&\quad\lesssim
 \int_0^t 2^{2j(s-s_0)}e^{-\kappa_{1}2^{2j}(t-\tau)}
\big(2^{2js_0}\|\nu^{-1/2}\dot\Delta_j\{\mathbf I-\mathbf P\}h(\tau)\|_{L^2_{x,v}}^2\big)
{\rm d}\tau
\nonumber\\   
&\quad\quad\lesssim \int_0^t(1+t-\tau)^{-(s-s_0)} \|\nu^{-1/2}\{\mathbf{I}-\mathbf{P}\}h(\tau)\|_{L_v^2(\dot B_{2,\infty}^{s_0})}^2 {\rm d}\tau,
\end{align}
which yields \eqref{DNnew1}.
For $j > j_0$, we  also infer from \eqref{G2.30} that
\begin{align}\label{J2}
2^{2js}\|\dot\Delta_{j}\mathcal T(t)\|_{L^2_{x,v}}^2
+2^{2js}\|\nabla_x\Delta_x^{-1}\mathbf P_0\dot\Delta_{j}\mathcal T(t)\|_{L^2_{x,v}} ^2 
\lesssim&\,   \int_0^t e^{-\kappa_1(t-\tau)} \big(2^{2js}\|\nu^{-1/2}\dot\Delta_j\{\mathbf I-\mathbf P\}h(\tau)\|_{L^2_{x,v}}^2\big)
{\rm d}\tau \nonumber\\
\lesssim&\,\int_0^t e^{-\kappa_1(t-\tau)} \|\nu^{-1/2}\{\mathbf{I}-\mathbf{P}\}h(\tau)\|_{L_v^2(\dot B_{2,\infty}^{s })}^2 {\rm d}\tau,
\end{align}
which leads to \eqref{DNnew2}.
Thanks to   $\mathbf P_0\{\mathbf I-\mathbf P_0\}h=0$, the initial Poisson term in \eqref{G2.29} vanishes. Thus, form \eqref{G2.29}
and the Cauchy--Schwarz inequality, for $j\leq j_0$, we have
\begin{align*}
& 2^{2js}\|\dot\Delta_j\mathcal U(t)\|_{L^2_{x,v}}^2
+2^{2js}\|\nabla_x\Delta_x^{-1}\mathbf P_0
\dot\Delta_j\mathcal U(t)\|_{L^2_{x,v}}^2\\
\lesssim&\,
2^{2js}\bigg(
\int_0^t e^{-\kappa_1 2^{2j}(t-\tau)}
\|\dot\Delta_j\{\mathbf I-\mathbf P_0\}h(\tau)\|_{L^2_{x,v}}
\,{\rm d}\tau\bigg)^2\\
{\lesssim}&\,
2^{2js}
\bigg(\int_0^t e^{-\kappa_1 2^{2j}(t-\tau)}
\,{\rm d}\tau\bigg)
\bigg(\int_0^t e^{-\kappa_1 2^{2j}(t-\tau)}
\|\dot\Delta_j\{\mathbf I-\mathbf P_0\}h(\tau)\|_{L^2_{x,v}}^2
\,{\rm d}\tau\bigg)\\
\lesssim&\,
\int_0^t 2^{2j(s-1)}e^{-\kappa_1 2^{2j}(t-\tau)}
\|\dot\Delta_j\{\mathbf I-\mathbf P_0\}h(\tau)\|_{L^2_{x,v}}^2
{\rm d}\tau\\
\lesssim&\,
\int_0^t (1+t-\tau)^{-(s-s_0-1)}
\|\{\mathbf I-\mathbf P_0\}h(\tau)\|_{L_v^2(\dot B^{s_0}_{2,\infty})}^2
{\rm d}\tau,
\end{align*}
where the last inequality uses $s-s_0-1\geq0$.
Taking the supremum over $j\leq j_0$ yields \eqref{DNnew3}.
For $j>j_0$, the same argument gives
\begin{align*}
 2^{2js}\|\dot\Delta_j\mathcal U(t)\|_{L^2_{x,v}}^2
+2^{2js}\|\nabla_x\Delta_x^{-1}\mathbf P_0
\dot\Delta_j\mathcal U(t)\|_{L^2_{x,v}}^2 
 \lesssim&\,
\int_0^t e^{-\kappa_1(t-\tau)}
2^{2js}\|\dot\Delta_j\{\mathbf I-\mathbf P_0\}h(\tau)\|_{L^2_{x,v}}^2
 {\rm d}\tau\\
 \lesssim&\,
\int_0^t e^{-\kappa_1(t-\tau)}
\|\{\mathbf I-\mathbf P_0\}h(\tau)\|_{L_v^2(\dot B^s_{2,\infty})}^2
 {\rm d}\tau.
\end{align*}
Taking the supremum over $j>j_0$ yields \eqref{DNnew4}.
%The proof of \eqref{DNnew3}--\eqref{DNnew4} is similar to those of \eqref{J1} and \eqref{J2}. 
Hence, we complete the proof of Proposition \ref{P2.5}.
\end{proof}

\section{Global existence of the Vlasov-Poisson-Boltzmann system}\label{sec3}
We first recall the energy norm introduced in \eqref{G1.5}.
Assume that  $(f(t,x,v),\phi(t,x))$ is a strong solution to the Cauchy problem \eqref{I1.4} on \([0,T_1]\), satisfying  
\begin{align}\label{G3.1}
\sup_{0\leq t\leq T_1} \Big\{\|f(t)\|_{\CE^{\frac{1}{2},N}}+\|\nabla_{x}\phi(t)\|_{\dot B^{\frac{1}{2}}_{2,\infty}\cap\dot H^N} \Big\}\leq \delta,
\end{align}
where $N\geq 3$ is an integer, and $\delta>0$ is a sufficiently small constant that is independent of $T_1$.  By utilizing the Sobolev embedding and the definition of \(\CE^{\frac12,N}\), the above bootstrap assumption \eqref{G3.1} implies
\begin{align*} 
&\|f(t)\|_{L^2_v(L^\infty)}
+\|(\langle v\rangle f(t),\nabla_v f(t))\|_{L^2_v(L^\infty )}+\|\nabla_{x}\phi(t)\|_{L^\infty}
\nonumber\\
&\quad
\lesssim
\|f(t)\|_{L^2_v(\dot B^{\frac12}_{2,\infty}\cap\dot H^N )}
+
\|(\langle v\rangle f(t),\nabla_v f(t))\|_{L^2_v(\dot H^1 \cap\dot H^{N-1} )}+\|\nabla_{x}\phi(t)\|_{\dot B^{\frac12}_{2,\infty}\cap\dot H^N }
\lesssim \delta,
\end{align*}
for any $0\leq t\leq T_1$.

Let \(j_0\in\mathbb Z\) be fixed. We split the solution into its low- and
high-frequency components with respect to the spatial variable:
\begin{align*}
f=f_L+f_H,\qquad
f_L:=\dot S_{j_0}f,\qquad
f_H:=({\rm Id}-\dot S_{j_0})f .    
\end{align*}

\subsection{A priori estimates at low frequencies}
In this subsection, we begin with the low-frequency component and derive an estimate for \(f_L\) in the \(L^2_v(\dot{B}^{\frac{1}{2}}_{2,\infty})\) norm.

\begin{lem}\label{L3.1}
For strong solutions of the problem   \eqref{I1.4}, it holds that
\begin{align}\label{G3.2}
 &\|f_{L}(t)\|_{L_v^2(\dot B_{2,\infty}^{\frac{1 }{2}})}+\|\nabla_{x}\phi_{L}(t)\|_{\dot B_{2,\infty}^{\frac{1 }{2}}} \nonumber\\
&\quad \lesssim  \delta  \sup_{{ 0\leq t\leq T_1}} \Big \{\|f(t)\|_{L_v^2(\dot B_{2,\infty}^{\frac{1}{2}} )} + \|\nabla_{x}\phi(t)\|_{\dot B_{2,\infty}^{\frac{1}{2}}} \Big\} + \sup_{{ 0\leq t\leq T_1}}\|E(t)\|_{ \dot B_{2,\infty}^{-\frac{3}{2}}\cap\dot B_{2,\infty}^{-\frac{1}{2}} }+\|f
_0\|_{L_v^2(\dot B_{2,\infty}^{\frac{1 }{2}})}+\|\nabla_{x}\phi_0\|_{\dot B_{2,\infty}^{\frac{1 }{2}}},
\end{align}
for $0\leq t \leq T_1$.
\end{lem}
\begin{proof}
It follows from \eqref{I1.4} and \eqref{LVPB-mild} that
\begin{align*}
f(t)=e^{t\mathcal B}f_0
+\int_0^t e^{(t-\tau)\mathcal B}\mathcal M(\tau) {\rm d}\tau,    
\end{align*}
where 
\begin{align*}
\mathcal M
=
\Gamma(f,f)
-(\nabla_x\phi+E)\cdot\nabla_v f
+\frac12 v\cdot(\nabla_x\phi+E)f
+E\cdot v\sqrt M.    
\end{align*}
By direct calculation, one has 
\begin{align*}
\mathbf{P}_0\mathcal{M}=0,\qquad \mathbf{P}_1\Gamma(f,f)=0,\qquad \{\mathbf{I}-\mathbf{P}\}E\cdot v\sqrt{M}=0.    
\end{align*}
Therefore, we derive that 
\begin{align}\label{G3.3}
f(t)=&\,e^{t\mathcal B}f_0
+\int_0^t e^{(t-\tau)\mathcal B}  \big[{E}\cdot v\sqrt{M}+ \mathbf{P}_1\big(-(\nabla_{x}\phi+E)\cdot\nabla_{v}f+\frac{1}{2}v\cdot(\nabla_{x}\phi+E)f\big)    \big](\tau) {\rm d}\tau\nonumber\\
&+\int_0^t e^{(t-\tau)\mathcal B} \{\mathbf{I}-\mathbf{P}\}\big[\Gamma(f,f)-(\nabla_{x}\phi+E)\cdot\nabla_{v}f+\frac{1}{2}v\cdot(\nabla_{x}\phi+E)f \big] (\tau){\rm d}\tau.    
\end{align}
By applying \(\dot S_{j_0}\) to \eqref{G3.3} and \(\nabla_x\Delta_x^{-1}\mathbf P_0\dot S_{j_0}\) to \eqref{G3.3}, and then using Proposition \ref{P2.4}, we obtain that
\begin{align}\label{G3.4}
&\|f_L(t)\|_{L^2_v(\dot B^{\frac{1}{2}}_{2,\infty})}
+\|\nabla_x\phi_L(t)\|_{\dot B^{\frac{1}{2}}_{2,\infty}} \nonumber\\
\lesssim&\,   \|\dot S_{j_0}e^{t\mathcal B}f_0\|_{L_v^2(\dot B_{2,\infty}^{\frac{1}{2}})}+ \|\nabla_x\Delta_x^{-1}\mathbf P_0\dot S_{j_0}e^{t\mathcal B}f_0\|_{L_v^2(\dot B_{2,\infty}^{\frac{1}{2}})} +\sup_{0\leq t \leq T_1}\|\nu^{-\frac{1}{2}}\Gamma(f,f)(t)\|_{L_v^2(\dot B_{2,\infty}^{-\frac{1}{2}})}+\sup_{0\leq t \leq T_1}\|E(t)\|_{\dot B_{2,\infty}^{-\frac{3}{2}}}\nonumber\\
&+\sup_{0\leq t \leq T_1} \|\mathbf{P}_1E\cdot(vf,\nabla_{v}f)(t)\|_{L_v^2(\dot B_{2,\infty}^{-\frac{3}{2}})}+ \sup_{0\leq t \leq T_1}\|\nu^{-\frac{1}{2}}\{\mathbf I-\mathbf{P}\}E\cdot(vf,\nabla_{v}f)(t)\|_{L_v^2(\dot B_{2,\infty}^{-\frac{1}{2}})} \nonumber\\
&+\sup_{0\leq t \leq T_1}\|\nu^{-\frac{1}{2}}\{\mathbf I-\mathbf{P}\}\nabla_{x}\phi\cdot(vf,\nabla_{v}f)(t)\|_{L_v^2(\dot B_{2,\infty}^{-\frac{1}{2}})}+\Big\| \int_0^t e_{{L}}^{(t-\tau)\mathcal{B}} \mathbf{P}_1\Big[\nabla_{x}\phi\cdot\Big( \frac{1}{2}vf-\nabla_{v}f\Big)\Big]{\rm d}\tau\Big\|_{L_v^2(\dot B_{2,\infty}^{\frac{1}{2}})}\nonumber\\
&+\Big\| \nabla_{x}\Delta_{x}^{-1}\mathbf{P}_0\int_0^t e_{{L}}^{(t-\tau)\mathcal{B}} \mathbf{P}_1\Big[\nabla_{x}\phi\cdot\Big( \frac{1}{2}vf-\nabla_{v}f\Big)\Big]{\rm d}\tau\Big\|_{L_v^2(\dot B_{2,\infty}^{\frac{1}{2}})},
\end{align}
where we used  the fact that
\begin{align*}
 \|\nabla_x\phi_L(t)\|_{\dot B^{\frac{1}{2}}_{2,\infty}} \backsim \|\nabla_{x}\Delta_{x}^{-1}\mathbf{P}_0 \dot S_{j_0} f(t)\|_{L_v^2(\dot B_{2,\infty}^{\frac{1}{2}})}.   
\end{align*}

Let’s estimate each term on the right-hand side of \eqref{G3.4} one by one. Through direct calculation, one gets  
\begin{align}\label{G3.5}
 \|\dot S_{j_0}e^{t\mathcal B}f_0\|_{L_v^2(\dot B_{2,\infty}^{\frac{1}{2}})}+ \|\nabla_x\Delta_x^{-1}\mathbf P_0\dot S_{j_0}e^{t\mathcal B}f_0\|_{L_v^2(\dot B_{2,\infty}^{\frac{1}{2}})}\lesssim    \|f_0\|_{L^2_v(\dot B^{\frac{1}{2}}_{2,\infty})}
+\|\nabla_x\phi_0 \|_{\dot B^{\frac{1}{2}}_{2,\infty}}. 
\end{align}
From \eqref{G2.16} and \eqref{A.3}, we have
\begin{align}\label{G3.6}
\sup_{0\leq t \leq T_1}\|\nu^{-\frac{1}{2}}\Gamma(f,f)(t)\|_{L_v^2(\dot B_{2,\infty}^{-\frac{1}{2}})}\lesssim&\, \sup_{0\leq t \leq T_1}\||\nu^{\frac{1}{2}}f|_{2}|f|_{2}(t) \|_{\dot B_{2,\infty}^{-\frac{1}{2}}} \nonumber\\
\lesssim&\,  \sup_{0\leq t \leq T_1} \|\nu^{\frac{1}{2}}f(t)\|_{L_v^2(\dot B_{2,\infty}^{\frac{1}{2}})}   \sup_{0\leq t \leq T_1} \| f(t)\|_{L_v^2(\dot B_{2,\infty}^{\frac{1}{2}})} \nonumber\\
\lesssim&\, \sup_{0\leq t \leq T_1}\Big(\|\langle v\rangle^{\frac{1}{2}}\{\mathbf{ I}-\mathbf{P}\}f \|_{L_v^2(\dot B_{2,\infty}^{\frac{1}{2}})}+  \| f \|_{L_v^2(\dot B_{2,\infty}^{\frac{1}{2}})}\Big) \sup_{0\leq t \leq T_1} \| f(t)\|_{L_v^2(\dot B_{2,\infty}^{\frac{1}{2}})} \nonumber\\
\lesssim&\, \delta\sup_{0\leq t \leq T_1}\| f(t)\|_{L_v^2(\dot B_{2,\infty}^{\frac{1}{2}})}.
\end{align}
For the nonlinear terms containing $E $, similar to the approach in \cite{DN-2026}, we derive that
\begin{align}\label{G3.7}
&\sup_{{0\leq t\leq T_1}} \|\mathbf{P}_1E\cdot(vf,\nabla_{v}f)(t)\|_{L_v^2(\dot B_{2,\infty}^{-\frac{3}{2}})}+ \sup_{{{0\leq t\leq T_1}}} \|\nu^{-\frac{1}{2}}\{\mathbf I-\mathbf{P}\}E\cdot(vf,\nabla_{v}f)\|_{L_v^2(\dot B_{2,\infty}^{-\frac{1}{2}})}    \nonumber\\
\lesssim&\, \sup_{0\leq t\leq T_1}\|E(t)\|_{\dot B_{2,\infty}^{-\frac{3}{2}}\cap \dot B_{2,\infty}^{-\frac{1}{2}}} \sup_{{0\leq t\leq T_1}}\|( v f,\nabla_{v} f)(t)\|_{L_v^2(\dot B_{2,1}^{\frac{3}{2}})}\nonumber\\
\lesssim&\, \sup_{0\leq t\leq T_1}\|E(t)\|_{\dot B_{2,\infty}^{-\frac{3}{2}}\cap \dot B_{2,\infty}^{-\frac{1}{2}}}\sup_{{0\leq t\leq T_1}} \|( v f,\nabla_{v} f)(t)\|_{L_v^2(\dot H^1\cap\dot H^{N-1})} \nonumber\\
\lesssim&\, \delta  \sup_{0\leq t\leq T_1}\|E(t)\|_{\dot B_{2,\infty}^{-\frac{3}{2}}\cap \dot B_{2,\infty}^{-\frac{1}{2}}}.
\end{align}
Similarly, it yields
\begin{align}\label{G3.8}
&\sup_{0\leq t \leq T_1}\|\nu^{-\frac{1}{2}}\{\mathbf I-\mathbf{P}\}\nabla_{x}\phi\cdot(vf,\nabla_{v}f)(t)\|_{L_v^2(\dot B_{2,\infty}^{-\frac{1}{2}})}\nonumber\\
\lesssim &\,\sup_{0\leq t \leq T_1} \|\nabla_{x}\phi(t)\|_{\dot B_{2,\infty}^{\frac{1}{2}}} \sup_{0\leq t \leq T_1}\|({ \langle v\rangle^{\frac{1}{2}}f},\nabla_{v}f)(t)\| _{L_v^2(\dot B_{2,\infty}^{\frac{1}{2}})}   \nonumber\\
\lesssim&\,\sup_{0\leq t \leq T_1}\|\nabla_{x}\phi(t)\|_{\dot B_{2,\infty}^{\frac{1}{2}}} \sup_{0\leq t \leq T_1}\Big(\| f(t)\|_{L_v^2(\dot B_{2,\infty}^{\frac{1}{2}})}  + \|\langle v\rangle^{\frac{1}{2}}\{\mathbf{ I}-\mathbf{P}\}f(t)\|_{L_v^2(\dot B_{2,\infty}^{\frac{1}{2}})}+\|\{\mathbf I-\mathbf{P}\}f\|_{H_{x,v}^N}(t)\Big) \nonumber\\
\lesssim&\, \delta\sup_{0\leq t \leq T_1} \|\nabla_{x}\phi(t)\|_{\dot B_{2,\infty}^{\frac{1}{2}}} .
\end{align}
Plugging the estimates \eqref{G3.5}--\eqref{G3.8} into
\eqref{G3.4} gives
\begin{align}\label{G3.9}
&\|f_{L}(t)\|_{L_v^2(\dot B_{2,\infty}^{\frac{1 }{2}})}+\|\nabla_{x}\phi_{{L}}(t)\|_{\dot B_{2,\infty}^{\frac{1 }{2}}} \nonumber\\
 \lesssim&\,  \delta  \sup_{0\leq t\leq
 T_1} \Big \{\|f(t)\|_{L_v^2(\dot B_{2,\infty}^{\frac{1}{2}} )} + \|\nabla_{x}\phi(t)\|_{\dot B_{2,\infty}^{\frac{1}{2}}} \Big\} + \sup_{0\leq t\leq
 T_1}\|E(t)\|_{ \dot B_{2,\infty}^{-\frac{3}{2}}\cap\dot B_{2,\infty}^{-\frac{1}{2}} }\nonumber\\
&+\|f
_0\|_{L_v^2(\dot B_{2,\infty}^{\frac{1 }{2}})}+\|\nabla_{x}\phi_0\|_{\dot B_{2,\infty}^{\frac{1 }{2}}}+\Big\| \int_0^t e^{(t-\tau)\mathcal{B}}_{{L}} \mathbf{P}_1\Big[\nabla_{x}\phi\cdot\Big( \frac{1}{2}vf-\nabla_{v}f\Big)\Big]{\rm d}\tau\Big\|_{L_v^2(\dot B_{2,\infty}^{\frac{1}{2}})}\nonumber\\
& +\Big\| \nabla_{x}\Delta_{x}^{-1}\mathbf{P}_0\int_0^t e_{{L}}^{(t-\tau)\mathcal{B}} \mathbf{P}_1\Big[\nabla_{x}\phi\cdot\Big( \frac{1}{2}vf-\nabla_{v}f\Big)\Big]{\rm d}\tau\Big\|_{L_v^2(\dot B_{2,\infty}^{\frac{1}{2}})}.
\end{align}

Although
\begin{align*}
&\Big\| \int_0^t e_{{L}}^{(t-\tau)\mathcal{B}} \mathbf{P}_1\Big[\nabla_{x}\phi\cdot\Big( \frac{1}{2}vf-\nabla_{v}f\Big)\Big]{\rm d}\tau\Big\|_{L_v^2(\dot B_{2,\infty}^{\frac{1}{2}})}\nonumber\\
&\quad+    \Big\| \nabla_{x}\Delta_{x}^{-1}\mathbf{P}_0\int_0^t e_{{L}}^{(t-\tau)\mathcal{B}} \mathbf{P}_1\Big[\nabla_{x}\phi\cdot\Big( \frac{1}{2}vf-\nabla_{v}f\Big)\Big]{\rm d}\tau\Big\|_{L_v^2(\dot B_{2,\infty}^{\frac{1}{2}})}\nonumber\\
&\quad\quad\lesssim \Big\|  \mathbf{P}_1\Big[\nabla_{x}\phi\cdot\Big( \frac{1}{2}vf - \nabla_{v}f\Big)\Big]\Big\|_{L_v^2(\dot B_{2,\infty}^{-\frac{3}{2}})},
\end{align*}
we are unable to directly  control  $\Big\|  \mathbf{P}_1\Big[\nabla_{x}\phi\cdot\Big( \frac{1}{2}vf - \nabla_{v}f\Big)\Big]\Big\|_{L_v^2(\dot B_{2,\infty}^{-\frac{3}{2}})}$.
To prove Lemma \ref{L3.1}, we need to handle the remaining two terms in \eqref{G3.9}. Since a direct estimate fails to yield a conclusive result, we need to develop a new VPB system structure. We will estimate this term in detail in the next subsection.

\subsection{The new structure of the VPB system}
Below, we proceed to handle the last term in \eqref{G3.9}.
For  the momentum moment, we have
\begin{align*}
\Big\langle \nabla_{x}\phi\cdot\Big( \frac{1}{2}vf - \nabla_{v}f\Big),v_k\sqrt M\Big\rangle
=&\,\int_{\mathbb R^3} v_k\sqrt M\Big(-{\nabla_{x}\phi}_i\partial_{v_i}f+\frac12\nabla\phi_iv_if\Big){\rm d}v\\
=&\,\nabla_{x}\phi_i\int_{\mathbb R^3} f\partial_{v_i}(v_k\sqrt M){\rm d}v
+\frac12\nabla\phi_i\int_{\mathbb R^3} v_kv_if\sqrt M {\rm d}v\\
=&\,\nabla_{x}\phi_i\int_{\mathbb R^3} f\Big(\delta_{ik}\sqrt M-\frac12v_iv_k\sqrt M\Big){\rm d}v
+\frac12\nabla\phi_i\int_{\mathbb R^3} v_iv_kf\sqrt M {\rm d}v\\
=&\,\nabla_{x}\phi_k\rho_{f}.
\end{align*}
For the temperature moment, we also have
\begin{align*}
&\Big\langle \nabla_{x}\phi\cdot\Big( \frac{1}{2}vf - \nabla_{v}f\Big),(|v|^2-3)\sqrt M\Big\rangle\\
=&\,{\nabla_{x} \phi}_i\int f\partial_{v_i}\big((|v|^2-3)\sqrt M\big){\rm d}v
+\frac12{\nabla_{x} \phi}_i\int (|v|^2-3)v_if\sqrt M{\rm d}v\\
=&\,{\nabla_{x} \phi}_i\int f\Big(2v_i-\frac12(|v|^2-3)v_i\Big)\sqrt M {\rm d}v
+\frac12{\nabla_{x} \phi}_i\int (|v|^2-3)v_if\sqrt M {\rm d}v\\
=&\,2{\nabla_{x} \phi}\cdot b.    
\end{align*}
Thus, we compute that 
\begin{align}\label{G3.10}
 \mathbf{P}_1\Big[\nabla_{x}\phi\cdot\Big( \frac{1}{2}vf - \nabla_{v}f\Big)\Big]  
=&\,  (\rho_{f} \nabla_{x}\phi)\cdot v\sqrt{M}+ \frac13(b\cdot \nabla_{x}\phi)(|v|^2-3)\sqrt{M} \nonumber\\
=:&\, I_1+I_2.
\end{align}    

Next, we deal with the terms $I_1$ and $I_2$ in \eqref{G3.10}.

\medskip
\noindent
\textbf{Step 1.  The estimate of \(\rho_{f} \nabla_{x}\phi\).}

Since \(\rho_{f}=\Delta_x\phi=\nabla_x\cdot \nabla_{x}\phi\), it follows that 
\begin{align*}
\rho_{f} \nabla_{x}\phi=\nabla_x\cdot\big(\nabla_{x}\phi\otimes \nabla_{x}\phi-\frac12|\nabla_{x}\phi|^2{\rm Id}\big).    
\end{align*}
which implies
\begin{align*}
\|\rho_{f} \nabla_{x}\phi\|_{\dot B_{2,\infty}^{-\frac{3}{2}}}\lesssim\Big\|\nabla\phi\otimes \nabla_{x}\phi-\frac12|\nabla_{x}\phi|^2{\rm Id}\Big\|_{\dot B^{-\frac{1}{2}}_{2,\infty}}
\lesssim \|\nabla_{x}\phi\|_{\dot B^{\frac{1}{2}}_{2,\infty}}^2
\lesssim \delta \|\nabla_{x}\phi\|_{\dot B_{2,\infty}^{\frac{1}{2}}}.
\end{align*} 
Hence, it holds that
\begin{align}\label{G3.11}
\Big\|\int_0^t e_{{L}}^{(t-\tau)\mathcal{B}}I_1 {\rm d}\tau\|_{L_v^2(\dot B_{2,\infty}^{ \frac{ 1}{2}})}+\Big\|\nabla_{x}\Delta_{x}^{-1}\mathbf{P}_0\int_0^t e_{{L}}^{(t-\tau)\mathcal{B}}I_1 {\rm d}\tau\|_{L_v^2(\dot B_{2,\infty}^{\frac{1}{2}})}\lesssim     \delta \|\nabla_{x}\phi\|_{\dot B_{2,\infty}^{\frac{1}{2}}}.
\end{align}

Now, we concentrate on the term $I_2$. We decompose $b\cdot\nabla_{x}\phi$ into the longitudinal part $b_{\parallel}\cdot \nabla_{x}\phi$ and the transverse part $b_{\perp}\cdot \nabla_{x}\phi$. Here, 
\begin{align*}
 b=b_{\parallel}+b_{\perp},
\qquad
b_{\parallel}=\nabla_x\Delta_x^{-1}\nabla_x\cdot b,
\qquad
\nabla_x\cdot b_{\perp}=0.   
\end{align*}
\medskip
\noindent
\textbf{Step 2.  The estimate of \( b_{\parallel}\cdot \nabla_{x}\phi\).}

We first estimate $b_{\parallel}\cdot \nabla_{x}\phi$.
Using \eqref{G2.9}$_1$ and \(\rho_{f}=\nabla_x\cdot \nabla_{x}\phi\), we infer that
\begin{align*}
b_{\parallel}=-\partial_t \nabla_{x}\phi,   
\end{align*}
which yields
\begin{align*}
b_{\parallel}\cdot \nabla_{x}\phi=-\frac12\partial_t|\nabla_{x}\phi|^2.  
\end{align*}
Let
\begin{align*}
 e_4=(|v|^2-3)\sqrt M, \qquad q_{\phi}=|\nabla_{x} \phi|^2.   
\end{align*}
For the Duhamel term containing \(\partial_tq_\phi e_4\), by integrating by parts in time, we have
\begin{align}\label{G3.12}
\int_0^t e_{L}^{(t-\tau)\mathcal{B}}\partial_\tau q_{\phi}(\tau)e_4{\rm d}\tau={\dot S_{j_0}}(q_{\phi}(t)e_4)-e_{L}^{t\mathcal{B}}q_{\phi}(0)e_4  +\int_0^t e_{L}^{(t-\tau)\mathcal{B}} \mathcal{B}\big[q_{\phi}(\tau)e_4\big]{\rm d}\tau.    
\end{align} 
From \eqref{DN2.20}, one has
\begin{align} \label{G3.13}
 \mathcal{B}\big[q_{\phi}(\tau)e_4\big]
= -v\cdot\nabla_x\big(|\nabla_{x}\phi|^2\big) e_4,
\end{align}
where we utilized the facts that \(e_4\in\ker \mathcal L\) and \(\mathbf{P}_0(e_{4})=0\), and the collision and Poisson parts vanish.
Substituting \eqref{G3.13} into \eqref{G3.12} yields
\begin{align*} 
\int_0^t e_{L}^{(t-\tau)\mathcal{B}}\partial_\tau q_{\phi}(\tau)e_4{\rm d}\tau={\dot S_{j_0}}(q_{\phi}(t)e_4)-e_{L}^{t\mathcal{B}}q_{\phi}(0)e_4  -\int_0^t e_{L}^{(t-\tau){\mathcal{B}}}v\cdot\nabla_x\big(|\nabla_{x}\phi|^2\big) e_4{\rm d}\tau, 
\end{align*} 
which, together with Proposition \ref{P2.4} and the product estimates \eqref{G2.16}--\eqref{G2.17}, gives rise to 
\begin{align}\label{G3.14}
&\Big\|\int_0^t e_{L}^{(t-\tau)\mathcal{B}}\partial_\tau q_{\phi}(\tau)e_4{\rm d}\tau\Big\|_{L_v^2(\dot B_{2,\infty}^{ \frac{1}{2}})}+ \Big\|\nabla_{x}\Delta_{x}^{-1}\mathbf{P}_0\int_0^t e_{L}^{(t-\tau)\mathcal{B}}\partial_\tau q_{\phi}(\tau)e_4{\rm d}\tau \Big\|_{L_v^2(\dot B_{2,\infty}^{\frac{1}{2}})} \nonumber\\
\lesssim&\, \|q_{\phi}(t)e_4\|_{L_v^2(\dot B_{2,\infty}^{\frac{1}{2}})}+\|e_{L}^{t\mathcal{B}}q_{\phi}(0)e_4\|_{L_v^2(\dot B_{2,\infty}^{\frac{1}{2}})}+\big\| v\cdot \nabla_{x}\big(  \nabla_{x}\phi|^2 \big)e_4 \big\|_{L_v^2(\dot B_{2,\infty}^{-\frac{3}{2}})} \nonumber\\
\lesssim&\, \sup_{0\leq t\leq T_1} \|\nabla_{x}\phi(t)\|_{\dot B_{2,\infty}^{\frac{1}{2}}}\|\nabla_{x} \phi (t)\|_{\dot B_{2,1}^{\frac{3}{2}}}+\sup_{0\leq t\leq T_1} \|\nabla_{x}\phi(t)\|_{\dot B_{2,\infty}^{\frac{1}{2}}} ^2 \nonumber\\
\lesssim&\, \delta \sup_{0\leq t\leq T_1} \|\nabla_{x}\phi(t)\|_{\dot B_{2,\infty}^{\frac{1}{2}}}.
\end{align}
\medskip
\noindent
\textbf{Step 3.  The estimate of \( b_{\perp}\cdot \nabla_{x}\phi\).}

We now continue to estimate the remaining part $b_{\perp} \cdot \nabla_{x} \phi$. 
Applying the longitudinal projection \(P_{\parallel}:=\nabla_x\Delta_x^{-1}\nabla_x\cdot\) to the momentum equation   \eqref{G2.9}$_2$, we obtain  
\begin{align}\label{G3.15}
\nabla_{x} \phi=\partial_tb_{\parallel}+\nabla_x(\rho_{f} + 2c)+P_{\parallel}\nabla_x\cdot\Theta(\{\mathbf I-\mathbf P\}f) - P_{\parallel}N_b,
\end{align}
where \(N_{b}:=E+\rho_{f}(\nabla_{x}\phi + E)\).
By using \eqref{G3.15}, one has
\begin{align}\label{G3.16}
b_{\perp}\cdot \nabla\phi
=&\,\ b_{\perp}\cdot\partial_tb_{\parallel}
+b_{\perp}\cdot\nabla_x(\rho_{f}+2c)
+b_{\perp}\cdot P_{\parallel}\nabla_x\cdot\Theta(\{\mathbf I-\mathbf P\}f)
-b_{\perp}\cdot P_{\parallel}N_b\nonumber\\
=&\, \partial_t(b_{\perp}\cdot b_{\parallel})-(\partial_tb_{\perp})\cdot b_{\parallel} +b_{\perp}\cdot\nabla_x(\rho_{f}+2c)
+b_{\perp}\cdot P_{\parallel}\nabla_x\cdot\Theta(\{\mathbf I-\mathbf P\}f)
-b_{\perp}\cdot P_{\parallel}N_b.
\end{align}
On the other hand, 
projecting \eqref{G2.9}$_2$ by \(P_{\perp}={\rm Id}-P_{\parallel}\) gives
\begin{align}\label{G3.17}
 \partial_tb_{\perp}+P_{\perp}\nabla_x\cdot\Theta(\{\mathbf I-\mathbf P\}f)=P_{\perp}N_b. 
\end{align} 
 From \eqref{G3.16}--\eqref{G3.17}, one gets
\begin{align} \label{G3.18}
b_{\perp}\cdot \nabla\phi
=&\ \partial_t(b_{\perp}\cdot b_{\parallel})
+b_{\perp}\cdot\nabla_x(\rho_{f}+2c)+b_{\perp}\cdot P_{\parallel}\nabla_x\cdot\Theta(\{\mathbf I-\mathbf P\}f)\nonumber\\
&+b_{\parallel}\cdot P_{\perp}\nabla_x\cdot\Theta(\{\mathbf I-\mathbf P\}f)-b_{\perp}\cdot P_{\parallel}N_b
-b_{\parallel}\cdot P_{\perp}N_b.
\end{align}

Next, we estimate each item on the right-hand side of \eqref{G3.18} one by one. First,
since 
\begin{align*}
\|b_{\perp}\cdot b_{\parallel}\|_{\dot B_{2,\infty}^{-\frac{1}{2}}\cap \dot B_{2,\infty}^{\frac{1}{2}}} \lesssim  \|f\|_{L_v^2(\dot B_{2,\infty}^{\frac{1}{2}})}^2+\|f\|_{L_v^2(\dot B_{2,\infty}^{\frac{1}{2}})} \|f\|_{L_v^2(\dot B_{2,1}^{\frac{3}{2}})} \lesssim \delta \|f\|_{L_v^2(\dot B_{2,\infty}^{\frac{1}{2}})},
\end{align*}
we  use the same integration-by-parts argument in time as in \eqref{G3.12} to obtain  
\begin{align}\label{G3.19}
&\Big\|\int_0^t e_{L}^{(t-\tau)\mathcal{B}}\partial_\tau  (b_{\perp}\cdot b_{\parallel})(\tau)e_4{\rm d}\tau\Big\|_{L_v^2(\dot B_{2,\infty}^{ \frac{1}{2}})}+ \Big\|\nabla_{x}\Delta_{x}^{-1}\mathbf{P}_0\int_0^t e_{L}^{(t-\tau)\mathcal{B}}\partial_\tau  (b_{\perp}\cdot b_{\parallel})(\tau)e_4{\rm d}\tau \Big\|_{L_v^2(\dot B_{2,\infty}^{\frac{1}{2}})} \nonumber\\
&\quad\lesssim  \delta \sup_{0\leq t\leq T_1} \|f(t)\|_{L_v^2(\dot B_{2,\infty}^{\frac{1}{2}})}.   
\end{align}

Second, by leveraging \(\nabla_x\cdot b_{\perp}=0\), one gets
\begin{align*}
b_{\perp}\cdot\nabla_x(\rho_{f}+2c)=\nabla_x\cdot[(\rho_{f}+2c)b_{\perp}],
\end{align*}
which implies
\begin{align}\label{G3.20}
&\Big\|\int_0^t e_{L}^{(t-\tau)\mathcal{B}} \nabla_x\cdot[(\rho_{f}+2c)b_{\perp}](\tau)e_4{\rm d}\tau\Big\|_{L_v^2(\dot B_{2,\infty}^{ \frac{1}{2}})}+ \Big\|\nabla_{x}\Delta_{x}^{-1}\mathbf{P}_0\int_0^t e_{L}^{(t-\tau)\mathcal{B}}\nabla_x\cdot[(\rho_{f}+2c)b_{\perp}]e_4{\rm d}\tau \Big\|_{L_v^2(\dot B_{2,\infty}^{\frac{1}{2}})} \nonumber\\   
&\quad\lesssim \|(\rho_{f}+2c)b_{\perp}\|_{\dot B_{2,\infty}^{-\frac{1}{2}}}\lesssim \Big(\|\nabla_{x}\phi\|_{L_v^2(\dot B^{\frac{1}{2}}_{2,\infty})}+\|f\|_{L_v^2(\dot B_{2,\infty}^{\frac{1}{2}})}\Big) \|f\|_{L_v^2(\dot B_{2,\infty}^{\frac{1}{2}})}\lesssim \delta \|f\|_{L_v^2(\dot B_{2,\infty}^{\frac{1}{2}})}.
\end{align}

Third, thanks to \(P_{\parallel}\nabla_x\cdot\Theta(\{\mathbf I-\mathbf P\}f)=\nabla_x\psi_{\Theta(\{\mathbf I-\mathbf P\}f)}\) for a scalar \(\psi_{\Theta(\{\mathbf I-\mathbf P\}f)}\), one reaches
\begin{align*}
 b_{\perp}\cdot P_{\parallel}\nabla_x\cdot\Theta(\{\mathbf I-\mathbf P\}f)
=b_{\perp}\cdot\nabla_x\psi_{\Theta(\{\mathbf I-\mathbf P\}f)}
=\nabla_x\cdot(\psi_{\Theta(\{\mathbf I-\mathbf P\}f)} b_{\perp}).   
\end{align*}
Since the Riesz transforms are bounded on Besov spaces, it holds that
\begin{align*}
\|\psi_{\Theta(\{\mathbf I-\mathbf P\}f)}\|_{\dot B^{\frac{1}{2}}_{2,\infty}}
{\lesssim}\| {\Theta(\{\mathbf I-\mathbf P\}f)}\|_{\dot B^{\frac{1}{2}}_{2,\infty}}
\lesssim \|f(t)\|_{\CE^{\frac{1}{2},N}},  
\end{align*}
which leads to
\begin{align*}
\|\psi_{\Theta(\{\mathbf I-\mathbf P\}f)} b_{\perp}\|_{\dot B^{-\frac{1}{2}}_{2,\infty}}
\lesssim  \delta  \|f\|_{L_v^2(\dot B_{2,\infty}^{\frac{1}{2}})}.
\end{align*}
Thus, we arrive at
\begin{align}\label{G3.21}
&\Big\|\int_0^t e_{L}^{(t-\tau)\mathcal{B}} b_{\perp}\cdot P_{\parallel}\nabla_x\cdot\Theta(\{\mathbf I-\mathbf P\}f)(\tau)e_4{\rm d}\tau\Big\|_{L_v^2(\dot B_{2,\infty}^{ \frac{1}{2}})}\nonumber\\
&\quad+ \Big\|\nabla_{x}\Delta_{x}^{-1}\mathbf{P}_0\int_0^t e_{L}^{(t-\tau)\mathcal{B}} b_{\perp}\cdot P_{\parallel}\nabla_x\cdot\Theta(\{\mathbf I-\mathbf P\}f)(\tau)e_4{\rm d}\tau \Big\|_{L_v^2(\dot B_{2,\infty}^{\frac{1}{2}})} \lesssim \delta  \|f\|_{L_v^2(\dot B_{2,\infty}^{\frac{1}{2}})}.
\end{align}

Fourth, for the term \(b_{\parallel}\cdot P_{\perp}\nabla_x\cdot\Theta(\{\mathbf I-\mathbf P\}f) \), we find that
\begin{align*}
\|P_{\perp}\nabla_x\cdot\Theta (\{\mathbf I-\mathbf P\}f) \|_{\dot B^{-\frac{1}{2}}_{2,1}}
\lesssim \|\{\mathbf I-\mathbf P\}f\|_{ L^2_v(L^2_{x}\cap\dot H_{x}^N)} \lesssim  \|f(t)\|_{\CE^{\frac{1}{2},N}},    
\end{align*}
and
\begin{align*}
\|b_{\parallel}\|_{\dot B^{\frac{1}{2}}_{2,\infty}}
\lesssim  \|f\|_{L_v^2(\dot B_{2,\infty}^{\frac{1}{2}})}.    
\end{align*}
Hence, the endpoint product estimate \eqref{G2.17} gives
\begin{align*}
 \|b_{\parallel}\cdot P_{\perp}\nabla_x\cdot\Theta (\{\mathbf I-\mathbf P\}f)\|_{\dot B^{-\frac{3}{2}}_{2,\infty}}
\lesssim  \de \|f\|_{L_v^2(\dot B_{2,\infty}^{\frac{1}{2}})},   
\end{align*}
which results in
\begin{align}\label{G3.22}
&\Big\|\int_0^t e_{L}^{(t-\tau)\mathcal{B}} b_{\parallel}\cdot P_{\perp}\nabla_x\cdot\Theta (\{\mathbf I-\mathbf P\}f)(\tau)e_4{\rm d}\tau\Big\|_{L_v^2(\dot B_{2,\infty}^{ \frac{1}{2}})}\nonumber\\
&\quad+ \Big\|\nabla_{x}\Delta_{x}^{-1}\mathbf{P}_0\int_0^t e_{L}^{(t-\tau)\mathcal{B}} b_{\parallel}\cdot P_{\perp}\nabla_x\cdot\Theta (\{\mathbf I-\mathbf P\}f)(\tau)e_4{\rm d}\tau \Big\|_{L_v^2(\dot B_{2,\infty}^{\frac{1}{2}})} \lesssim \delta  \|f\|_{L_v^2(\dot B_{2,\infty}^{\frac{1}{2}})}.
\end{align}

Fifth, for the last two terms in \eqref{G3.18}, since $P_{\perp}$ and $P_{\parallel}$ are zeroth-order operators, by using  {\eqref{G2.17} and \eqref{DNG2.28}}, we have
\begin{align}\label{G3.23}
&\Big\|\int_0^t e_{L}^{(t-\tau)\mathcal{B}} \big(b_{\perp}\cdot P_{\parallel}N_b
+b_{\parallel}\cdot P_{\perp}N_b\big)(\tau)e_4{\rm d}\tau\Big\|_{L_v^2(\dot B_{2,\infty}^{ \frac{1}{2}})}\nonumber\\
&\quad+ \Big\|\nabla_{x}\Delta_{x}^{-1}\mathbf{P}_0\int_0^t e_{L}^{(t-\tau)\mathcal{B}} \big(b_{\perp}\cdot P_{\parallel}N_b
+b_{\parallel}\cdot P_{\perp}N_b\big)(\tau)e_4{\rm d}\tau \Big\|_{L_v^2(\dot B_{2,\infty}^{\frac{1}{2}})}  \nonumber\\
&\quad\quad {\lesssim
\sup_{0\leq t\leq T_1}
\Big\|
\big(b_\perp\cdot P_\parallel N_b
+b_\parallel\cdot P_\perp N_b\big)(t)
\Big\|_{\dot B^{-\frac{3}{2}}_{2,\infty}}}\nonumber\\
&\quad\quad {\lesssim
\sup_{0\leq t\leq T_1} \|b(t)\|_{\dot B_{2,1}^{\frac{3}{2}}} \|N_{b}(t)\|_{\dot B_{2,\infty}^{-\frac{3}{2}}}         }
 \nonumber\\
&\quad\quad\lesssim  \sup_{0\leq t\leq T_1}\|b(t)\|_{\dot B^{\frac{3}{2}}_{2,1}} \bigg(\|E(t)\|_{\dot B_{2,\infty}^{-\frac{3}{2}}}\Big(1+  \|{\rho_f(t)}\|_{\dot B^{\frac{3}{2}}_{2,1}}\Big)+\|\rho_{f}\nabla_{x}\phi\|_{\dot B_{2,\infty}^{-\frac{3}{2}}}  \bigg) \nonumber\\
 &\quad\quad\lesssim \delta   \sup_{0\leq t\leq
 T_1} \Big \{\|f(t)\|_{L_v^2(\dot B_{2,\infty}^{\frac{1}{2}} )}+ \|\nabla_{x}\phi\|_{\dot B_{2,\infty}^{\frac{1}{2}}} \Big\}.
\end{align}
Consequently, collecting all the estimates    \eqref{G3.19}--\eqref{G3.23}, we further obtain
\begin{align}\label{G3.24}
 &\Big\|\int_0^t e_{L}^{(t-\tau)\mathcal{B}} b_{\perp}\cdot \nabla_{x}\phi(\tau)e_4{\rm d}\tau\Big\|_{L_v^2(\dot B_{2,\infty}^{ \frac{1}{2}})}+ \Big\|\nabla_{x}\Delta_{x}^{-1}\mathbf{P}_0\int_0^t e_{L}^{(t-\tau)\mathcal{B}}  b_{\perp}\cdot \nabla_{x}\phi(\tau)e_4{\rm d}\tau \Big\|_{L_v^2(\dot B_{2,\infty}^{\frac{1}{2}})}  \nonumber\\  
 &\quad\lesssim \delta   \sup_{0\leq t\leq
 T_1} \Big \{\|f(t)\|_{L_v^2(\dot B_{2,\infty}^{\frac{1}{2}} )}+ \|\nabla_{x}\phi\|_{\dot B_{2,\infty}^{\frac{1}{2}}} \Big\}.
\end{align}

According to the estimates  \eqref{G3.11}, \eqref{G3.14} and \eqref{G3.24},  one derives
\begin{align*}
&\Big\| \int_0^t e_{{L}}^{(t-\tau)\mathcal{B}} \mathbf{P}_1\Big[\nabla_{x}\phi\cdot\Big( \frac{1}{2}vf-\nabla_{v}f\Big)\Big]{\rm d}\tau\Big\|_{L_v^2(\dot B_{2,\infty}^{\frac{1}{2}})}\nonumber\\
&\quad+    \Big\| \nabla_{x}\Delta_{x}^{-1}\mathbf{P}_0\int_0^t e_{{L}}^{(t-\tau)\mathcal{B}} \mathbf{P}_1\Big[\nabla_{x}\phi\cdot\Big( \frac{1}{2}vf-\nabla_{v}f\Big)\Big]{\rm d}\tau\Big\|_{L_v^2(\dot B_{2,\infty}^{\frac{1}{2}})}\nonumber\\
&\quad\quad\lesssim  \delta   \sup_{0\leq t\leq
 T_1} \Big \{\|f(t)\|_{L_v^2(\dot B_{2,\infty}^{\frac{1}{2}} )}+ \|\nabla_{x}\phi\|_{\dot B_{2,\infty}^{\frac{1}{2}}} \Big\},
\end{align*}
which together with \eqref{G3.9} gives rise to \eqref{G3.2}.
Therefore, we complete the proof of Lemma \ref{L3.1}.
\end{proof}

\subsection{A priori estimates at high frequencies}

\begin{lem}\label{L3.2}
For strong solutions of the problem \eqref{I1.4}, there exists a positive constant  $\eta_1 >0$ such that  
\begin{align}\label{G3.25}
&\frac{{\rm d}}{{\rm d}t}\Big(\|  f(t)\|_{L_v^2(\dot H^1\cap \dot H^N)}^2+\|\nabla_{x}\phi(t)\|_{\dot H^1\cap\dot H^N}^2\Big)+\eta_1\sum_{1\leq k\leq N}\|\nabla_{x}^k\{\mathbf{I}-\mathbf{P}\}f(t)\|_{\nu}^2    \nonumber\\
&\quad\leq C{{\sum_{k=1}^N}} \sum_{|\beta|\leq k} \big\|  \big|\nabla_{x}^{|\beta|}f \big|_2 \big|  \nabla_{x}^{k-|\beta|}f \big|_{\nu} \big \|_{L^2}^2+{C\sup_{{0\leq t\leq T_1}}\|E(t)\|_{     H^N}^2}+ \kappa \sum_{2\leq k\leq N}\|\nabla^k_{x}\mathbf{P}f(t)\|_{L_{x,v}^2}^2\nonumber\\
&\quad\quad+ C\delta \|(vf,\nabla_v f)\|_{L_v^2(\dot H^1\cap \dot H^{N-1})}^2,
\end{align}
for $0\leq t \leq T_1$. Here and in the sequel, $0 < \kappa< 1$ is a constant that can be sufficiently small.
\end{lem}

\begin{proof}
Let \(1\leq |\alpha|\leq N\). Applying \(\partial_x^\alpha\) to \eqref{I1.4}$_1$--\eqref{I1.4}$_2$ and then taking the \(L^2_{x,v}\) inner product with \(\partial_x^\alpha f\) yields
\begin{align}\label{G3.26}
&\frac12\frac{{\rm d}}{{\rm d}t}
\Big(\|\partial_x^\alpha f\|_{L^2_{x,v}}^2+
\|\partial_x^\alpha \nabla_{x}\phi\|_{L^2_x}^2\Big)
+\lambda_0\|\partial_x^\alpha\{\mathbf{ I}-\mathbf P\}f\|_\nu^2\nonumber\\
\leq&\, |\langle\partial^\alpha_{x}\Gamma(f,f) ,\partial^\alpha_{x} f\rangle_{{x,v}} |+ |\langle\partial_{x}^\alpha E\cdot v\sqrt{M},\partial^\alpha_{x} f\rangle_{{x,v}} |+|\langle\partial^ \alpha_{x} (E \cdot vf) ,\partial^\alpha_{x} f\rangle_{{x,v}} |\nonumber\\
 & +|\langle\partial^ \alpha_{x} (E \cdot \nabla_vf) ,\partial^\alpha_{x} f\rangle_{{x,v}} | +|\langle\partial^ \alpha_{x} (\nabla_{x}\phi \cdot vf) ,\partial^\alpha_{x} f\rangle_{{x,v}} |+|\langle\partial_{x}^ \alpha (\nabla_{x}\phi \cdot \nabla_vf) ,\partial_{x}^\alpha f\rangle_{{x,v}} |\nonumber\\
 \equiv:& \sum_{k=1}^6J_{k},
\end{align}
where we used \eqref{G2.2} and \eqref{G2.9}$_1$.
The terms $J_1$, $J_2$, $J_3$, and $J_4$ are the same as those in \cite{DN-2026}. Thus, we have
\begin{align}\label{G3.27}
 {\sum_{i=1}^4 J_{i}}\leq&\,      \kappa   \|\partial^\alpha_{x} \{\mathbf{I}-\mathbf{P}\}f\|_{\nu}^2+\kappa { \sum_{\ell=2}^N  \|\nabla^{\ell}_{x} f\|_{L_{x,v}^2}^2}+C\sum_{|\beta|\leq k}\big\|  |\nabla^{|\beta|}_{x} f|_{2} |\nabla^{k-|\beta|}_{x}f|_{\nu} \big \|_{L^2}^2+C\sup_{0\leq t\leq T_1}\|E(t)\|_{   H^{N}}^2.
\end{align}
For $k = 1, 2, \dots, N - 1$, by applying \eqref{G2.16}, Lemma \ref{LA.6}, integration by parts and Young's inequality, one arrives at
\begin{align} \label{G3.28}
J_5+J_6\leq&\,      \kappa \sum_{|\beta|=1}\|\nabla^{|\alpha|+|\beta|}_{x} f\|_{L_{x,v}^2}^2+C\sup_{0\leq t\leq T_1}\|(\nabla_{x}\phi\cdot v f,\nabla_{x}\phi\cdot\nabla_vf)(t)\|_{L_v^2(\dot H^{k-1})}^2  \nonumber\\
\leq&\,    \kappa  \sum_{|\beta|=1}\|\nabla^{|\alpha|+|\beta|}_{x} f\|_{L_{x,v}^2}^2+C\sup_{0\leq t\leq T_1}\|(\nabla_{x}\phi\cdot v f,\nabla_{x}\phi\cdot\nabla_vf)(t)\|_{L_{x,v}^2 }^2\nonumber\\
&+C\sup_{0\leq t\leq T_1}\|(\nabla_{x}\phi\cdot v f,\nabla_{x}\phi\cdot\nabla_vf)(t)\|_{L_v^2(\dot H^1\cap \dot H^{N-2}) }^2 \nonumber\\
\leq&\, \kappa  \sum_{|\beta|=1}\|\nabla^{|\alpha|+|\beta|}_{x} f\|_{L_{x,v}^2}^2+C\sup_{0\leq t\leq T_1} \|\nabla_{x}\phi\|_{\dot B^{\frac{1}{2}}_{2,\infty}}^2 \|(vf,\nabla_v f)\|_{L_v^2(\dot H^1)}^2\nonumber\\
&+C\sup_{0\leq t\leq
 T_1}\|\nabla_{x} \phi\|_{{\dot H^1\cap\dot H^{N}}}^2\|(vf,\nabla_v f)\|_{L_v^2(\dot H^1\cap \dot H^{N-2})}^2 \nonumber\\
 \leq&\, \kappa \sum_{|\beta|=1}\|\nabla^{|\alpha|+|\beta|}_{x} f\|_{L_{x,v}^2}^2+C\delta \|(vf,\nabla_v f)\|_{L_v^2(\dot H^1\cap \dot H^{N-2})}^2.
\end{align}
For $k = N$, applying the macro-micro decomposition \eqref{G2.4} and integration by parts yields
\begin{align}\label{G3.29}
J_5+J_6\leq&\, C \|\nabla_{x}\phi f\|_{L_v^2(\dot H^N)}\Big( \|\nabla
^N_{x}\{\mathbf{I}-\mathbf{P}\}f\|_{\nu}+\|\mathbf{P}f\|_{L_v^2(\dot H^N)}  \Big)\nonumber\\
&+ C\|\nabla^N_{x}(\nabla_{x}\phi \cdot\nabla_{v}f)-\nabla_{x}\phi \nabla^N_{x}\nabla_vf\|_{L_{x,v}^2} \|\nabla^N_{x} f\|_{L_{x,v}^2} \nonumber\\
\leq&\, \kappa    \Big( \|\nabla^{N}_{x} f\|_{L_{x,v}^2}^2+\|\nabla^N_{x}\{\mathbf{I}-\mathbf{P}\}f\|_{\nu}^2\Big)+C\sup_{0\leq t\leq T_1} \|\nabla_{x}\phi\|_{\dot B^{\frac{1}{2}}_{2,\infty}}^2 \|\nabla_v f\|_{L_v^2(\dot H^1)}^2\nonumber\\
&+ C{\sup_{0\leq t\leq
 T_1}\|\nabla_{x}\phi\|_{\dot H^1\cap\dot H^{N}}^2\| \nabla_v f\|_{L_v^2(\dot H^1\cap \dot H^{N-1})}^2} \nonumber\\
 \leq&\,  \kappa    \Big( \|\nabla^{N}_{x} f\|_{L_{x,v}^2}^2+\|\nabla^N_{x}\{\mathbf{I}-\mathbf{P}\}f\|_{\nu}^2\Big)+C\delta \| \nabla_v f \|_{L_v^2(\dot H^1\cap \dot H^{N-1})}^2,
\end{align}
where we utilized 
\begin{align*}
\int_{\mathbb{R}^3_{x}} \big\langle \nabla_{x}\phi\cdot\nabla^N_{x}\nabla_{v}f,\nabla^N_{x}f\big\rangle {\rm d}x=-\frac{1}{2}\int_{\mathbb R^3_{x}}\!\int_{\mathbb R^3_{v}}\nabla_{v}\nabla_{x}\phi |\nabla^N_{x} f|^2{\rm d}x{\rm d}v=0.
\end{align*}
Substituting the estimates \eqref{G3.27}--\eqref{G3.29} into   \eqref{G3.26} yields \eqref{G3.25}.
\end{proof}

To deal with the first nonlinear term on the right-hand side of \eqref{G3.25}, we apply the argument of \cite[Lemma 3.3]{DLN-2026-arXiv} to the current setting, which leads to the following estimate, and the proof is omitted here.
\begin{lem} \label{L3.3}
For strong solutions of the problem \eqref{I1.4}, it holds that
\begin{align} \label{G3.30}
\sum_{|\beta|\leq k} \big\|  \big|\nabla_{x}^{|\beta|}f \big|_2 \big|  \nabla^{k-|\beta|}_{x}f \big|_{\nu} \big \|_{L^2}^2 \lesssim \delta \Big(\sum_{2\leq\ell\leq N} \|\nabla_{x}^{\ell} f\|_{L_{x,v}^2}^2+\sum_{1\leq \ell\leq N}\|\nabla_{x}^\ell
  \{\mathbf{I}-\mathbf{P}\}f\|_{\nu}^2 \Big),   
\end{align}
for any $0\leq t \leq  T_1$ and $1\leq k\leq N$.
\end{lem}

Combining \eqref{G3.25} and \eqref{G3.30}, we directly have
the  result as follows.
\begin{cor}\label{cor3.4}
For strong solutions of the problem \eqref{I1.4}, there exists a positive constant  $\eta_2 >0$ such that  
\begin{align}\label{G3.31}
&\frac{{\rm d}}{{\rm d}t}\Big(\|  f(t)\|_{L_v^2(\dot H^1\cap \dot H^N)}^2+\|\nabla_{x}\phi(t)\|_{\dot H^1\cap\dot H^N}^2\Big)+\eta_2\sum_{1\leq k\leq N}\|\nabla_{x}^k\{\mathbf{I}-\mathbf{P}\}f(t)\|_{\nu}^2    \nonumber\\
&\quad \leq{C\sup_{{0\leq t\leq T_1}}\|E(t)\|_{     H^N}^2}+ \kappa \sum_{2\leq k\leq N}\|\nabla^k_{x}\mathbf{P}f(t)\|_{L_{x,v}^2}^2 + C\delta \|(vf,\nabla_v f)\|_{L_v^2(\dot H^1\cap \dot H^{N-1})}^2,
\end{align}
for $0\leq t \leq T_1$. Here and in the sequel, $0 < \kappa< 1$ is a constant that can be sufficiently small.
\end{cor}

To control the dissipation of the hydrodynamic part $\mathbf{P}f$ on the right-hand side of \eqref{G3.31}, for \(1\le k\le N - 1\), we define the following macroscopic interaction functional:
\begin{align}\label{macro}
\mathfrak E_k(t)
:=&\,
\sum_{|\alpha|=k}\sum_{i,j=1}^3
\int_{\mathbb R^3_x}
\partial^\alpha_{x}(\partial_i b_j+\partial_j b_i) 
\partial^\alpha_{x} {A_{ij}(\{\mathbf{I}-\mathbf{P}\}f)} {\rm d}x
\nonumber\\
&+
\sum_{|\alpha|=k}\sum_{i=1}^3
\int_{\mathbb R^3_x}
\partial^\alpha_{x}\partial_i c
\partial^\alpha_{{x}} B_i(\{\mathbf{I}-\mathbf{P}\}f){\rm d}x-\frac{1}{4}
\sum_{|\alpha|=k}
\int_{\mathbb R^3_x}
\partial^\alpha_{x}\rho_{f}
\partial^\alpha_{x}\nabla_x\cdot b{\rm d}x.
\end{align}
By using Young's inequality, we have
\begin{align*}
 \sum_{k=1}^{N-1}|\mathfrak E_{k}(t) |\lesssim  \|\mathbf{P}f\|_{L_v^2(\dot H^1\cap\dot H^N)}^2+\|\{\mathbf{I}-\mathbf{P}\}f\|_{L_v^2(\dot H^1\cap\dot H^N)}^2\lesssim \|f\|_{L_v^2(\dot H^1\cap\dot H^N)}^2. 
\end{align*}
Since $\|\nabla_{x}^k\mathbf{P}f\|_{L_{x,v}^2}\sim \|\nabla^k_{x}(a,b,c)\|_{L^2}$ for any $k \geq 0$, we then provide the estimate of $a$, $b$ and $c$.
\begin{lem}\label{L3.5}
For strong solutions of the problem \eqref{I1.4}, there exists a positive constant $\eta_3>0$    such that 
\begin{align}\label{G3.33}
&\frac{{\rm d}}{{\rm d}t}\sum_{k=1}^{N-1}\mathfrak{E}_{ k}(t)+\eta_3  \Big( \|(a,b,c)(t)\|^2_{\dot H^2\cap\dot H^N}+\|\rho_{f}(t)\|^2_{\dot H^1\cap\dot H^{N-1}}\Big) \nonumber\\
&\quad\leq   C\sup_{{0\leq t\leq T_1}}\|E(t)\|_{  H^N}^2+C\| \{\mathbf{I}-\mathbf{P}\}f(t) \|_{L_v^2(\dot H^1\cap\dot H^{N})}^2+C\delta\|(a,b,c)(t)\|_{\dot H^1}^2,
\end{align}
for    any $0\leq t\leq T_1$,  where $\mathfrak{E}_{k}(t)$ is defined by \eqref{macro}. 
\end{lem}

\begin{proof}
We first provide the estimate of $b$. By applying \(\partial^\alpha_{x}\) to \eqref{G2.13}$_1$--\eqref{G2.13}$_2$, with \(|\alpha| = k\), one arrives at
\begin{align}\label{G3.34}
\sum_{i,j=1}^3\|\partial^\alpha_{x} (\partial_i b_j+\partial_j b_i)\|^2_{L^2}
=&\,
-\frac{{\rm d}}{{\rm d}t}
\sum_{i,j=1}^3
\int_{\mathbb R^3_{x}}
\partial^\alpha_{x} (\partial_i b_j+\partial_j b_i)
\partial^\alpha_{x}{A_{ij}(\{\mathbf{I}-\mathbf{P}\}f) }{\rm d}x
\nonumber\\
&
+
\sum_{i,j=1}^3
\int_{\mathbb R^3_{x}}
\partial^\alpha_{x}\partial_t (\partial_i b_j+\partial_j b_i)
\partial^\alpha_{x} A_{ij}(\{\mathbf{I}-\mathbf{P}\}f) {\rm d}x
\nonumber\\
&
+2
\sum_{i,j=1}^3
\int_{\mathbb R^3_{x}}
\partial^\alpha_{x}\partial_t (\partial_i b_j+\partial_j b_i)
\partial^\alpha_{x} c\delta_{ij}  {\rm d}x
\nonumber\\
&
+
\sum_{i,j=1}^3
\int_{{\mathbb R^3_{x}}}
\partial^\alpha_{x} (\partial_i b_j+\partial_j b_i)
\partial^\alpha_{x}A_{ij}(\mathcal R+\mathcal G){\rm d}x .
\end{align}
Taking $|\alpha|=k$ with  $1\leq k\leq N-1$, one notices that
\begin{align}\label{G3.35}
 \sum_{i,j=1}^3\|\partial^\alpha_{x} (\partial_i b_j+\partial_j b_i)\|^2_{L^2} =
2\|\partial^\alpha_{x}\nabla_x b\|^2_{L^2}
+
2\|\partial^\alpha_{x}\nabla_x\cdot b\|^2_{L^2}.
\end{align}
For the  second term on the right hand side of \eqref{G3.34}, it follows from \eqref{G2.9}$_2$ that
\begin{align}\label{G3.36}
& \sum_{i,j=1}^3
\int_{\mathbb R^3_{x}}
\partial^\alpha_{x}\partial_t (\partial_i b_j+\partial_j b_i)
\partial^\alpha_{x} A_{ij}(\{\mathbf{I}-\mathbf{P}\}f) {\rm d}x  \nonumber\\
=&\, -2\sum_{i,j=1}^3\int_{\mathbb R^3_{x}}
\partial^\alpha_{x}\partial_t  b_i
\partial_{j}\partial^\alpha_{x} A_{ij}(\{\mathbf{I}-\mathbf{P}\}f){\rm d}x  \nonumber\\
=&\,2\sum_{i,j=1}^3 \int_{\mathbb R^3_{x}} \partial^\alpha_{x}\big[ \partial_{i}\rho_{f}+2\partial_{i}c+\sum_{m=1}^3\partial_{m} \Theta_{im}(\{\mathbf{I}-\mathbf{P}\}f)-\partial_{i}\phi-E_{i}-\rho_{f}(\partial_{i}\phi+E_{i})  \big]  \partial_{j}\partial^\alpha_{x} A_{ij}(\{\mathbf{I}-\mathbf{P}\}f) {\rm d}x  \nonumber\\
\leq&\, \kappa \|\nabla_x\partial^\alpha_{x}(\rho_{f},c)\|^2_{L^2}+(\kappa+\delta) \|   \rho_{f}\|_{\dot H^1\cap\dot H^{N}}^2+C\sup_{0\leq t\leq T_1}\|E\|_{H^N}^2
+C\| \{\mathbf{I}-\mathbf{P}\}f \|_{L_v^2(\dot H^1\cap\dot H^{N})}^2.
\end{align}
For the third term on the right hand side of \eqref{G3.34}, we deduce from \eqref{G2.10}  that
\begin{align}\label{G3.37}
& 2
\sum_{i,j=1}^3
\int_{\mathbb R^3_{x}}
\partial^\alpha_{x}\partial_t (\partial_i b_j+\partial_j b_i)
\partial^\alpha_{x} c\delta_{ij}  {\rm d}x  \nonumber\\
=&\, -4\int_{\mathbb R^3_{x}}
\partial^\alpha_{x}\nabla_x\cdot b
\partial^\alpha_{x}\partial_t c{\rm d}x\nonumber\\
=&\, \frac{2}{3} \int_{\mathbb R^3_{x}}
\partial^\alpha_{x}\nabla_x\cdot b \partial^\alpha_{x}\big[2\nabla_{x}\cdot b+\nabla_{x}  \cdot\Xi(\{\mathbf I-\mathbf P\}f)-
2 b\cdot(\nabla_{x}\phi+E)\big] {\rm d}x \nonumber\\
\leq&\, \Big(\frac{4}{3}+\kappa\Big) \|\partial^\alpha_{x}\nabla_{x}\cdot b\|_{L^2}^2+C\| \{\mathbf{I}-\mathbf{P}\}f \|_{L_v^2(\dot H^1\cap\dot H^{N})}^2+C\sup_{0\leq t\leq T_1}\|E\|_{H^N}^2+C\delta \Big( \|b\|_{\dot H^2\cap \dot H^N}^2 + \|\rho_{f}\|_{\dot H^1\cap \dot H^{N-1}}^2\Big). 
\end{align}
For the last term on the right  hand side of \eqref{G3.34}, by using Young’s inequality and the fact that $A_{ij}$ can absorb any velocity derivative and any velocity weight, we also have
\begin{align}\label{G3.38}
 &\sum_{i,j=1}^3
\int_{{\mathbb R^3_{x}}}
\partial^\alpha_{x} (\partial_i b_j+\partial_j b_i)
\partial^\alpha_{x}A_{ij}(\mathcal R+\mathcal G){\rm d}x \nonumber\\
\leq&\,\kappa \|\nabla_x\partial^\alpha b\|^2_{L^2}+C\| \{\mathbf{I}-\mathbf{P}\}f \|_{L_v^2(\dot H^1\cap\dot H^{N})}^2+C\sup_{0\leq t\leq T_1}\|E\|_{H^N}^2+C\delta \|(a,b,c)\|_{\dot H^1\cap \dot H^N}^2.
\end{align}
Putting all the estimates \eqref{G3.35}--\eqref{G3.38} into \eqref{G3.34} and summing over $1\leq k\leq N-1$  gives
\begin{align}\label{G3.39}
&\frac{{\rm d}}{{\rm d}t}{{\sum_{k=1}^{N-1}}}
\sum_{|\alpha|=k}\sum_{i,j=1}^3
\int_{\mathbb R^3_{x}}
\partial^\alpha_{x} (\partial_i b_j+\partial_j b_i)
\partial^\alpha_{x} {A_{ij}(\{\mathbf{I}-\mathbf{P}\}f)}  {\rm d}x+  (1+\eta_4) \| b\|_{\dot H^2\cap \dot H^N}^2 \nonumber\\
\lesssim&\,   \sup_{0\leq t\leq T_1}\|E(t)\|_{  H^N}^2+ \| \{\mathbf{I}-\mathbf{P}\}f \|_{L_v^2(\dot H^1\cap\dot H^{N})}^2+\delta\|(a,b,c)\|_{\dot H^1\cap\dot H^N}^2+(\kappa+\delta) \|\rho_{f}\|_{\dot H^1\cap \dot H^{N-1}}^2,
\end{align}
for some constant $\eta_{4}>0$.

Next, we give the estimate of \(c\). According to \eqref{G2.13}$_3$, we obtain
\begin{align}\label{G3.40}
\|\nabla_x\partial^\alpha c\|_{L^2 }^2
=&\,
-\frac{{\rm d}}{{\rm d}t}
\sum_{i=1}^3
\int
\partial^\alpha_{x}\partial_i c
\partial_{x}^\alpha B_i(\{\mathbf{I}-\mathbf P\}f) {\rm d}x
\nonumber\\
&\quad
+
\sum_{i=1}^3
\int_{\mathbb R^3_{x}}
\partial^\alpha_{x}\partial_i\partial_t c
\partial^\alpha_{x} B_i(\{\mathbf{I}-\mathbf P\}f){\rm d}x
+
\sum_{i=1}^3
\int_{{\mathbb R_{x}^3}}
\partial^\alpha_{x}\partial_i c
\partial^\alpha_{x}B_i (\mathcal R+\mathcal G) {\rm d}x.
\end{align}
Similar to \eqref{G3.37}, through integration by parts and utilizing the equation \eqref{G2.10}, one gets
\begin{align}\label{G3.41}
&\sum_{i=1}^3
\int_{\mathbb R^3_{x}}
\partial^\alpha_{x}\partial_i\partial_t c
\partial^\alpha_{x} B_i(\{\mathbf{I}-\mathbf P\}f){\rm d}x\nonumber\\
\leq &\,      C\| \{\mathbf{I}-\mathbf{P}\}f \|_{L_v^2(\dot H^1\cap\dot H^{N})}^2+C\sup_{0\leq t\leq T_1}\|E\|_{H^N}^2+C\delta \Big( \|b\|_{\dot H^2\cap \dot H^N}^2 + \|\rho_{f}\|_{\dot H^1\cap \dot H^{N-1}}^2\Big).
\end{align}
Through a direct calculation as in \eqref{G3.38}, we conclude that
\begin{align}\label{G3.42}
&\sum_{i=1}^3
\int_{{\mathbb R_{x}^3}}
\partial^\alpha_{x}\partial_i c
\partial^\alpha_{x}B_i (\mathcal R+\mathcal G) {\rm d}x\nonumber\\
\leq&\,\kappa \|\nabla_x\partial^\alpha_{x} c\|^2_{L^2} +C\| \{\mathbf{I}-\mathbf{P}\}f \|_{L_v^2(\dot H^1\cap\dot H^{N})}^2+C\sup_{0\leq t\leq T_1}\|E\|_{H^N}^2+C\delta \|(a,b,c)\|_{\dot H^1\cap \dot H^N}^2   . 
\end{align}
Plugging \eqref{G3.41} and \eqref{G3.42} into \eqref{G3.40} and summing over $1\leq k\leq N-1$ yields
\begin{align}\label{G3.43}
&\frac{{\rm d}}{{\rm d}t}
{\sum_{k=1}^{N-1}}\sum_{|\alpha|=k}\sum_{i=1}^3
\int
\partial^\alpha_{x}\partial_i c
\partial_{x}^\alpha B_i(\{\mathbf{I}-\mathbf P\}f) {\rm d}x +\Big(\frac{1}{2}+\eta_{5}\Big) \|c\|_{\dot H^2\cap\dot H^N}^2 \nonumber\\
\lesssim&\,   \sup_{0\leq t\leq T_1}\|E(t)\|_{  H^N}^2+ \| \{\mathbf{I}-\mathbf{P}\}f \|_{L_v^2(\dot H^1\cap\dot H^{N})}^2+\delta\|(a,b,c)\|_{\dot H^1\cap\dot H^N}^2+\delta \|\rho_{f}\|_{\dot H^1\cap \dot H^{N-1}}^2, 
\end{align}
for some constant $\eta_{5}>0$.

Finally, we provide the remaining estimate of $\rho_{f}$.
By means of \eqref{G2.9}$_1$, one has
\begin{align} \label{G3.44}
-\frac{{\rm d}}{{\rm d}t} \int_{\mathbb R^3_x}
\partial^\alpha_{x}\rho_f 
\partial^\alpha_{x}\nabla_x\cdot b {\rm d}x
=&\,
\|\partial^\alpha_{x}\nabla_x\cdot b\|_{L^2}^2
-
\int_{\mathbb R^3_x}
\partial^\alpha_{x}\rho_f 
\partial^\alpha_{x}\nabla_x\cdot\partial_t b {\rm d}x.
\end{align}
Making use of  \eqref{G2.9}$_2$, we derive
\begin{align}\label{G3.45}
&-\int_{\mathbb R^3_x}
\partial^\alpha_{x}\rho_f 
\partial^\alpha_{x}\nabla_x\cdot\partial_t b {\rm d}x\nonumber\\
= &\, -\|\nabla_x\partial^\alpha_{x}\rho_{f}\|^2_{L^2}
-\|\partial^\alpha_{x}\rho_{f}\|^2_{L^2}
-2\int_{\mathbb R^3_{x}}
\nabla_x\partial^\alpha_{x}\rho_{f}\cdot
\nabla_x\partial^\alpha_{x} c {\rm d}x-\int_{\mathbb R^3_{x}}
\nabla_x\partial^\alpha_{x}\rho_{f}\cdot
\partial^\alpha_{x}\nabla_x\cdot\Theta(\{\mathbf{I}-\mathbf{P}\}f){\rm d}x
\nonumber\\
&
+
\int_{\mathbb R^3_{x}}
\nabla_x\partial^\alpha_{x}\rho_{f}\cdot
\partial^\alpha_{x} \big[ E+\rho_{f}(\nabla_{x}\phi +E) \big]{\rm d}x \nonumber\\
\leq&\, -\|\nabla_x\partial^\alpha_{x}\rho_{f}\|^2_{L^2}
-\|\partial^\alpha_{x}\rho_{f}\|^2_{L^2}+\Big(\frac{1}{2}+\kappa \Big)\|\nabla_x\partial^\alpha_{x}\rho_{f}\|^2_{L^2}+ C\| \{\mathbf{I}-\mathbf{P}\}f \|_{L_v^2(\dot H^1\cap\dot H^{N})}^2+C\delta \|\rho_{f}\|_{\dot H^1\cap \dot H^{N-1}}^2\nonumber\\
&+C\sup_{0\leq t\leq T_1}\|E(t)\|_{  H^N}^2+2 \|\nabla_{x}\partial^\alpha_{x}c\|_{L^2}^2.
\end{align}
By substituting \eqref{G3.45} into \eqref{G3.44}, and summing over $1\leq k\leq N-1$, we find that
\begin{align}\label{G3.46}
 &-\frac{1}{4}\frac{{\rm d}}{{\rm d}t}{\sum_{k=1}^{N-1}}\sum_{|\alpha|=k}\int_{\mathbb R^3_x}
\partial^\alpha_{x}\rho_f 
\partial^\alpha_{x}\nabla_x\cdot b {\rm d}x
+\eta_{6} \sum_{|\alpha|=k}\big(\| \nabla_x\partial^\alpha_{x}\rho_{f}\|^2_{L^2}
+\|\partial^\alpha_{x}\rho_{f}\|^2_{L^2}\big)\nonumber\\  
\leq&\, \frac{1}{2} \| c\|_{\dot H^2\cap \dot H^N}^2+\frac{1}{4}\|b\|_{\dot H^2\cap\dot H^N}^2+ C\| \{\mathbf{I}-\mathbf{P}\}f \|_{L_v^2(\dot H^1\cap\dot H^{N})}^2+C\delta \|\rho_{f}\|_{\dot H^1\cap \dot H^{N-1}}^2+C\sup_{0\leq t\leq T_1}\|E(t)\|_{  H^N}^2
\end{align}
for some constant $\eta_{6}>0$.
By combining \eqref{G3.39}, \eqref{G3.43}, and \eqref{G3.46}, and leveraging the fact that 
\begin{align*}
\|\nabla_x^{k}a\|_{L^2}\lesssim\|\nabla_x^{k}\rho_{f}\|_{L^2}+\|\nabla_x^{k}c\|_{L^2}    
\end{align*}
for any $k\geq 0$, we obtain \eqref{G3.33}, thereby completing the proof of Lemma \ref{L3.5}.
\end{proof}

\subsection{A priori estimates for the microscopic part.}
To absorb the norm of $\langle v\rangle f$ and $\nabla_{v}f$, we shall study the microscopic part of $f$, since 
\begin{align*}
\|\langle v\rangle \mathbf{P}f\|_{L_v^2(\dot H^1\cap\dot H^{N - 1})}+\|\nabla_v \mathbf{P}f\|_{L_v^2(\dot H^1\cap\dot H^{N - 1})}\lesssim \|f\|_{L_v^2(\dot H^1\cap\dot H^{N - 1})}.
\end{align*}
Below, we estimate $\langle v\rangle\{\mathbf{I}-\mathbf{P}\}f$ and $\nabla_{v}\{\mathbf{I}-\mathbf{P}\}f$.

\begin{lem}\label{L3.6}
For strong solutions of the problem   \eqref{I1.4}, there exists a positive constant $\eta_{7}>0$ such that
\begin{align} \label{G3.47}
&\frac{{\rm d}}{{\rm d}t}\|\nu\{\mathbf{I}-\mathbf{P}\}f(t) \|_{L_v^2(\dot H^1\cap\dot H^{N-1})}^2+ \eta_{7} \|\nu^{\frac{3}{2}}\{\mathbf{I}-\mathbf{P}\}f(t)\|_{L_v^2(\dot H^1\cap\dot H^{N-1})}^2 \nonumber\\
&\quad\lesssim  \|f(t)\|_{L_v^2(\dot H^1\cap\dot H^N)}^2+\delta \|\nu^{\frac{1}{2}}\nabla_v\{\mathbf{I}-\mathbf{P}\}f(t)\|_{L_v^2(\dot H^1\cap \dot H^{N-1})}^2+\sup_{0\leq t\leq T_1}\|E(t)\|_{H^N}^2\nonumber\\
&\quad\quad+\|\nu^{\frac{1}{2}}\{\mathbf{I}-\mathbf{P}\}f\|_{L_v^2(\dot H^1\cap \dot H^{N-1})}^2,
\end{align}
for any $0\leq t \leq T_1$.
\end{lem} 
\begin{proof}
Applying the operator $\{\mathbf{I}-\mathbf P\}$ to \eqref{I1.4}$_1$ gives
\begin{align}\label{G3.48}
&\partial_t \{\mathbf{I}-\mathbf{P}\}f+v\cdot\nabla_{x} \{\mathbf{I}-\mathbf{P}\}f+(E+\nabla_{x}\phi)\cdot\nabla_v \{\mathbf{I}-\mathbf{P}\}f-\frac{1}{2}(E+\nabla_{x}\phi)\cdot v\{\mathbf{I}-\mathbf{P}\}f
{-\mathcal{L}}\{\mathbf{I}-\mathbf{P}\}f\nonumber\\
=&\,\Gamma(f,f)-v\cdot \nabla_{x} \mathbf{P}f+\mathbf{P}(v\cdot\nabla_{x} f)-(E+\nabla_{x}\phi)\cdot\nabla_v\mathbf{P}f+\frac{1}{2}(E+\nabla_{x}\phi)\cdot v\mathbf{P}f\nonumber\\
&+ \mathbf{P}\Big[ \frac{1}{2}(E+\nabla_{x}\phi)\cdot v f-(E+\nabla_{x}\phi)\cdot\nabla_v f\Big],
\end{align}
where we used
\begin{align*}
\{\mathbf{I}-\mathbf{P}\} (E+\nabla_{x}\phi)\cdot v\sqrt{M}=0,\quad \{\mathbf{I}-\mathbf{P}\}\mathcal{L}f=\mathcal{L} \{\mathbf{I}-\mathbf{P}\}f. 
\end{align*}
Let \(1\le |\alpha|\le N-1\). Applying \(\partial_x^\alpha\) to
\eqref{G3.48}, multiplying by \(\nu^2\partial_x^\alpha \{\mathbf{I}-\mathbf{P}\}f\), and integrating over \(\mathbb R^3_x\times\mathbb R^3_v\), we get
\begin{align}\label{G3.49}
&\frac{1}{2}\frac{{\rm d}}{{\rm d}t}\|\nu\{\mathbf{I}-\mathbf{P}\}\partial^\alpha_{x} f\|_{L_{x,v}^2}^2+ \langle -\nu^2 \mathcal{L}\{\mathbf{I}-\mathbf{P}\}\partial^\alpha_{x} f,  \{\mathbf{I}-\mathbf{P}\} \partial^\alpha_{x} f \rangle_{x,v} \nonumber\\
=&\,\langle \nu\partial^\alpha_{x}\Gamma(f,f), \nu \{\mathbf{I}-\mathbf{P}\}\partial^\alpha_{x} f \rangle_{x,v}+\big\langle \partial^\alpha_{x}\big[\mathbf{P}(v\cdot\nabla_{x} f)-v\cdot\nabla_{x} \mathbf{P}f\big],\nu^2 \{\mathbf{I}-\mathbf{P}\} \partial^\alpha_{x} f\rangle_{x,v}\nonumber\\
&{-}  \Big\langle  \partial^\alpha_{x}\mathbf{P}\Big[ \frac{1}{2}E\cdot vf-E\cdot\nabla_v f  \Big], \nu^2\{\mathbf{I}-\mathbf{P}\}\partial^\alpha_{x} f  \Big\rangle_{x,v}\nonumber\\
&+\Big\langle  \partial^\alpha_{x} \Big[ \frac{1}{2}E\cdot v\mathbf{P}f-E\cdot\nabla_v \mathbf{P}f  \Big], \nu^2\{\mathbf{I}-\mathbf{P}\}\partial^\alpha_{x} f  \Big\rangle_{x,v}\nonumber\\
&+\Big\langle  \partial^\alpha_{x} \Big[ \frac{1}{2}E\cdot v 
\{\mathbf{I}- \mathbf{P}\}f-E\cdot\nabla_v\{ \mathbf{I}- \mathbf{P} \}f  \Big], \nu^2\{\mathbf{I}-\mathbf{P}\}\partial^\alpha_{x} f  \Big\rangle_{x,v}\nonumber\\
&+  \Big\langle  \partial^\alpha_{x}\mathbf{P}\Big[ \frac{1}{2}\nabla_{x}\phi\cdot vf-\nabla_{x}\phi\cdot\nabla_v f  \Big], \nu^2\{\mathbf{I}-\mathbf{P}\}\partial^\alpha_{x} f  \Big\rangle_{x,v}\nonumber\\
&+\Big\langle  \partial^\alpha_{x} \Big[ \frac{1}{2}\nabla_{x}\phi\cdot v\mathbf{P}f-\nabla_{x}\phi\cdot\nabla_v \mathbf{P}f  \Big], \nu^2\{\mathbf{I}-\mathbf{P}\}\partial^\alpha_{x} f  \Big\rangle_{x,v}\nonumber\\
&+\Big\langle  \partial^\alpha_{x} \Big[ \frac{1}{2}\nabla_{x}\phi\cdot v 
\{\mathbf{I}- \mathbf{P}\}f-\nabla_{x}\phi\cdot\nabla_v\{ \mathbf{I}- \mathbf{P} \}f  \Big], \nu^2\{\mathbf{I}-\mathbf{P}\}\partial^\alpha_{x} f  \Big\rangle_{x,v}\nonumber\\
=&\,\sum_{j=1}^8 K_{j}.
\end{align}
For the left hand side of \eqref{G3.49}, with the help of Lemma \ref{LA.1}, one has
\begin{align} \label{G3.50}
 \langle- \nu^2 \mathcal{L}\partial^{\alpha}_{x}\{\mathbf{I}-\mathbf{P}\}f
,\partial^{\alpha}_{x}\{\mathbf{I}-\mathbf{P}\}f \rangle_{x,v}\geq \frac{1}{2} \|\nu\partial^\alpha_{x}\{\mathbf{I}-\mathbf{P}\}f\|_{\nu}^2-C\|\partial^\alpha_{x} \{\mathbf{I}-\mathbf{P}\}f\|_{\nu}^2.
\end{align}
Similar to the argument in \cite[Lemma 3.6]{DN-2026}, for the terms $K_1,\dots,K_5$, we can conclude that
\begin{align}\label{G3.51}
\sum_{j=1}^5K_{j}\leq&\, C \delta \|\nu^{\frac{1}{2}}\nabla_v\{\mathbf{I}-\mathbf{P}\}f\|_{L_v^2(\dot H^1\cap \dot H^{N-1})}^2+ C  \sup_{0\leq t\leq T_1} \|E(t)\|_{H^N}^2\nonumber\\ 
&+\Big(C\delta+\frac{1}{8}\Big)\|\nu^\frac{3}{2}  \{\mathbf{I}-\mathbf{P}\}f\|_{L_v^2(\dot H^1\cap\dot H^{N-1})}^2 +C\|f\|_{L_v^2(\dot H^2\cap\dot H^N)}^2.
\end{align}

For the terms $K_6$ and $K_7$, noting that $v^k \sqrt{M} \lesssim 1$ for any integer $k \geq 0$, we infer that  
\begin{align}\label{G3.52}
K_6+K_7\leq&\,   C\sup_{0\leq t\leq T_1}\|\nabla_{x}\phi f\|_{L_v^2(\dot H^1\cap\dot H^{N-1})}^2 +\frac{1}{8}\|\nu \partial^{\alpha}_{x}\{\mathbf{I}-\mathbf{P}\}f\|_{\nu}^2  \nonumber\\
\leq&\,C\sup_{0\leq t\leq T_1}\|\nabla_{x}\phi\|_{\dot H^1\cap \dot H^{N-1}}^2 \|f\|_{L_v^2(\dot H^1\cap \dot H^{N-1})}^2+\frac{1}{8}\|\nu \partial^{\alpha}_{x}\{\mathbf{I}-\mathbf{P}\}f\|_{\nu}^2 \nonumber\\
\leq &\, C  \delta \|f\|_{L_v^2(\dot H^1\cap\dot H^N)}^2+\frac{1}{8}\|\nu^\frac{3}{2}  \{\mathbf{I}-\mathbf{P}\}f\|_{L_v^2(\dot H^1\cap\dot H^{N-1})}^2 . 
\end{align}
Similarly, for the remaining term $K_8$, we also have  
\begin{align}\label{G3.53}
 K_8\leq&\,   C\|\nabla_{x}\phi\cdot {\nu}^{\frac{3}{2}}\{\mathbf{I}-\mathbf{P}\}f\|_{L_v^2(\dot H^1\cap \dot H^{N-1})}^2+ C\|\nabla_{x}\phi\cdot {\nu}^{\frac{1}{2}} \nabla_v\{\mathbf{I}-\mathbf{P}\}f\|_{L_v^2(\dot H^1\cap \dot H^{N-1})}^2\nonumber\\
& +\frac{1}{8}\|\nu \partial^{\alpha}\{\mathbf{I}-\mathbf{P}\}f\|_{\nu}^2 \nonumber\\
 \leq&\, C\|\nabla_{x}\phi\|_{ \dot H^1\cap 
{\dot H^{N-1}}}^2 \Big( \|\nu^{\frac{3}{2}}\{\mathbf{I}-\mathbf{P}\}f\|_{L_v^2(\dot H^1\cap {\dot H^{N-1}})}^2 +\|\nu^{\frac{1}{2}}\nabla_v\{\mathbf{I}-\mathbf{P}\}f\|_{L_v^2(\dot H^1\cap \dot H^{N-1})}^2\Big) \nonumber\\
 &+\frac{1}{8}\|\nu \partial^{\alpha}\{\mathbf{I}-\mathbf{P}\}f\|_{\nu}^2\nonumber\\
 \leq&\,  \Big( C\delta+\frac{1}{8}\Big) \|\nu^\frac{3}{2}\{\mathbf{I}-\mathbf{P}\}f\|_{L_v^2(\dot H^1\cap\dot H^{N-1})}^2+\delta \|\nu^{\frac{1}{2}}\nabla_v\{\mathbf{I}-\mathbf{P}\}f\|_{L_v^2(\dot H^1\cap \dot H^{N-1})}^2.
\end{align}
By inserting the estimates \eqref{G3.50}--\eqref{G3.53} into \eqref{G3.49}, we end up with \eqref{G3.47}.
\end{proof}

\begin{lem}\label{L3.7}
For strong solutions of the problem   \eqref{I1.4}, there exists a positive constant $\eta_{8}>0$ such that
\begin{align} \label{G3.54}
&\frac{{\rm d}}{{\rm d}t} C_{\beta} \sum_{\substack{ 1\leq |\beta|\leq N \\|\alpha|+|\beta| \leq N}}\|\partial^\alpha_{x}\partial^\beta_{v}\{\mathbf{I}-\mathbf{P}\}f(t) \|_{L_{x,v}^2 }^2+ \eta_{8}\sum_{\substack{ 1\leq |\beta|\leq N \\|\alpha|+|\beta| \leq N}} \| \partial^{\alpha}_{x}\partial^\beta_{v}\{\mathbf{I}-\mathbf{P}\}f(t)\|_{\nu}^2 \nonumber\\
\lesssim&\, \|\nu^{\frac{1}{2}}\{\mathbf{I}-\mathbf{P}\}f(t)\|_{L_v^2(   H^{N-1})}^2+ \|f(t)\|_{L_v^2(\dot H^{\frac{3}{4}}\cap\dot H^N)}^2,
\end{align}
for any $0\leq t \leq T_1$.
\end{lem} 

\begin{proof}
For fixed \(1\leq k\leq {N}\), let \(\alpha,\beta\) be multi-indices such that $
|\beta|=k$ and $ |\alpha|+|\beta|\leq N$.
Taking the \(L^2_{x,v}\)-inner product of \eqref{G3.48} with
\(\partial^\alpha_{\beta}\{\mathbf I-\mathbf P\}f\), we have
\begin{align}\label{G3.55}
&\frac{1}{2}\frac{{\rm d}}{{\rm d}t} \|\partial^\alpha_{x}\partial^\beta_{v}\{\mathbf{I}-\mathbf{P}\}f\|_{L_{x,v}^2}^2+  \langle -\partial^\alpha_{x} \partial^\beta_{v} \mathcal{L}\{\mathbf{I}-\mathbf{P}\}f, \partial^\alpha_{x} \partial^\beta_{v}  \{\mathbf{I}-\mathbf{P}\}f   \rangle _{x,v} \nonumber\\
=&\, \langle  \partial^\alpha_{x} \partial^\beta_{v} \Gamma(f,f),  \partial^\alpha_{x} \partial^\beta_{v} \{\mathbf{I-\mathbf{P}}\}f \rangle_{x,v}+\big\langle \partial^\alpha_{x}\partial^\beta_{v}\big[\mathbf{P}(v\cdot\nabla f)-v\cdot\nabla \mathbf{P}f\big],\partial^\alpha_{x}\partial^\beta_{v} \{\mathbf{I}-\mathbf{P}\} f\rangle_{x,v} \nonumber\\
 &+  \Big\langle  \partial^\alpha_{x}\partial^\beta_{v}\Big[\mathbf{P}\Big( \frac{1}{2}E\cdot vf-E\cdot\nabla_v f  \Big)\Big], \partial^\alpha_{x}\partial^\beta_{v}\{\mathbf{I}-\mathbf{P}\} f  \Big\rangle_{x,v}\nonumber\\
&+\Big\langle \partial^\alpha_{x}\partial^\beta_{v}\Big[ \frac{1}{2}E\cdot v\mathbf{P}f-E\cdot\nabla_v \mathbf{P}f  \Big], \partial^\alpha_{x}\partial^\beta_{v}\{\mathbf{I}-\mathbf{P}\}  f  \Big\rangle_{x,v}\nonumber\\
&+\Big\langle  \partial^\alpha_{x}\partial^\beta_{v} \Big[ \frac{1}{2}E\cdot v 
\{\mathbf{I}- \mathbf{P}\}f-E\cdot\nabla_v\{ \mathbf{I}- \mathbf{P} \}f  \Big],  \partial^\alpha_{x}\partial^\beta_{v}\{\mathbf{I}-\mathbf{P}\} f  \Big\rangle_{x,v} \nonumber\\
 &+  \Big\langle  \partial^\alpha_{x}\partial^\beta_{v}\Big[\mathbf{P}\Big( \frac{1}{2}\nabla_{x}\phi\cdot vf-\nabla_{x}\phi\cdot\nabla_v f  \Big)\Big], \partial^\alpha_{x}\partial^\beta_{v}\{\mathbf{I}-\mathbf{P}\} f  \Big\rangle_{x,v}\nonumber\\
&+\Big\langle \partial^\alpha_{x}\partial^\beta_{v}\Big[ \frac{1}{2}\nabla_{x}\phi\cdot v\mathbf{P}f-\nabla_{x}\phi\cdot\nabla_v \mathbf{P}f  \Big], \partial^\alpha_{x}\partial^\beta_{v}\{\mathbf{I}-\mathbf{P}\}  f  \Big\rangle_{x,v}\nonumber\\
&+\Big\langle  \partial^\alpha_{x}\partial^\beta_{v} \Big[ \frac{1}{2}\nabla_{x}\phi\cdot v 
\{\mathbf{I}- \mathbf{P}\}f-\nabla_{x}\phi\cdot\nabla_v\{ \mathbf{I}- \mathbf{P} \}f  \Big],  \partial^\alpha_{x}\partial^\beta_{v}\{\mathbf{I}-\mathbf{P}\} f  \Big\rangle_{x,v} \nonumber\\
\equiv:&\, \sum_{i=1}^{8}L_{i}.
\end{align}
According to Lemma \ref{LA.1}, one has
\begin{align}\label{G3.56}
 \langle  -\partial^\alpha_{x}\partial^\beta_{v}\mathcal{L} \{\mathbf{I}-\mathbf{P}\}f
,\partial^\alpha_{x}\partial^\beta_{v}\{\mathbf{I}-\mathbf{P}\}f \rangle_{x,v}\geq \frac{1}{2} \| \partial^\alpha_{x}\partial^\beta_{v}\{\mathbf{I}-\mathbf{P}\}f\|_{\nu}^2-C\|\partial^\alpha_{x} \{\mathbf{I}-\mathbf{P}\}f\|_{\nu}^2.
\end{align}
On the one hand, the estimates of $L_1,\dots,L_5$ are the same as those in \cite[Lemma 3.7]{DN-2026}, and it holds that
\begin{align}\label{G3.57}
\sum_{i=1}^5L_{i}\leq  C\| f\|_{L_v^2(\dot H^\frac{3}{4}\cap\dot H^N)}^2  +\Big( C\delta+\frac{1}{9}\Big) \sum_{\substack{1\leq|\beta |\leq N\\|\alpha|+|\beta |\leq N}}\| \partial^{\alpha}_{x}\partial^\beta_{v}\{\mathbf{I}-\mathbf{P}\}f\|_{\nu}^2.   
\end{align}

On the other hand, by applying \eqref{G2.16}, Sobolev's and Young's inequalities, we get
\begin{align}\label{G3.58}
L_{6}+L_{7} \leq &\,   C\|\nabla_{x}\phi f\|_{L_v^2(H^{N-1})}^2+  \frac{1}{9}  \| \partial^{\alpha}_{x}\partial^\beta_{v}\{\mathbf{I}-\mathbf{P}\}f\|_{\nu}^2\nonumber\\
\leq&\, C \|\nabla_{x}\phi\|_{\dot B^\frac{1}{2}_{2,\infty}}^2 \|f\|_{L_v^2(\dot H^1)}^2+C\|\nabla_{x}\phi\|_{\dot H^1\cap\dot H^{N-1}}^2 \|f\|_{L_v^2(\dot H^1\cap\dot H^{N-1})}^2+\frac{1}{9}  \| \partial^{\alpha}_{x}\partial^\beta_{v}\{\mathbf{I}-\mathbf{P}\}f\|_{\nu}^2\nonumber\\
\leq&\, C\|f\|_{L_v^2(\dot H^1\cap \dot H^N)}^2+\Big(C\delta+\frac{1}{9} \Big)\| \partial^{\alpha}_{x}\partial^\beta_{v}\{\mathbf{I}-\mathbf{P}\}f\|_{\nu}^2.    
\end{align}
For the remaining term $L_{8}$, by making use of Lemmas \ref{LA.5}--\ref{LA.7}, we derive that 
\begin{align} \label{G3.59}
L_{8} \lesssim&\, \sum_{|\alpha^\prime|<|\alpha|}\int_{\mathbb{R}^3\times{\mathbb{R}^3}}|\partial^{\alpha-\alpha^{\prime}}\nabla_{x}\phi||\nabla_v\partial_x^{\alpha^{\prime}}\partial^\beta_v\{\mathbf{I}-\mathbf{P}\}f|_2|\partial_x^{\alpha}\partial^\beta_v\{\mathbf{I}-\mathbf{P}\}f|_2\mathrm{d}x\mathrm{d}v\nonumber\\
&+\sum_{|\alpha^\prime|<|\alpha|}\int_{\mathbb{R}^3\times{\mathbb{R}^3}} |\partial^{\alpha-\alpha^{\prime}}\nabla_{x}\phi||\partial_x^{\alpha^{\prime}}\partial^\beta_v (v\{\mathbf{I}-\mathbf{P}\}f)|_2|\partial_x^{\alpha}\partial^\beta_v\{\mathbf{I}-\mathbf{P}\}f|_2\mathrm{d}x  {\mathrm{d}v} \nonumber\\
\lesssim&\,\|\nabla_{x}\phi\|_{\dot B_{2,\infty}^{\frac{1}{2}}\cap\dot H^N}\sum_{\substack{1\leq|\beta^{\prime}|\leq N\\|\alpha^{\prime}|+|\beta^{\prime}|\leq N}}\|\partial^{\alpha^{\prime}}_x\partial^{\beta^{\prime}}_v\{\mathbf{I}-\mathbf{P}\}f\|_{\nu}^2,\nonumber\\
\lesssim&\,\delta\sum_{\substack{1\leq|\beta^{\prime}|\leq N\\|\alpha^{\prime}|+|\beta^{\prime}|\leq N}}\|\partial^{\alpha^{\prime}}_x\partial^{\beta^{\prime}}_v\{\mathbf{I}-\mathbf{P}\}f\|_{\nu}^2.   
\end{align}
By substituting all the estimates \eqref{G3.56}–\eqref{G3.59} into \eqref{G3.55}, taking the sum over   $\beta$ with $|\beta|=k$ and   $\alpha$ with $|\alpha|+|\beta|\leq N$, we consequently obtain the desired inequality \eqref{G3.54}. 
\end{proof}

Finally, to absorb the  estimate of $\|\nu^{\frac{1}{2}}\{\mathbf{I}-\mathbf{P}\}f\|_{L_{x,v}^2}$ on the right-hand side of \eqref{G3.54}, we present the following lemma.

 \begin{lem}\label{L3.8}
For strong solutions of the problem   \eqref{I1.4}, there exists a positive constant $\eta_{9}>0$ such that
\begin{align}  \label{G3.60}
&\frac{{\rm d}}{{\rm d}t}\| \{\mathbf{I}-\mathbf{P}\}f(t) \|_{L_{x,v}^2 }^2+ \eta_{9} \| \{\mathbf{I}-\mathbf{P}\}f(t)\|_{\nu}^2  \lesssim \|f(t)\|_{L_v^2(\dot H^{\frac{3}{4}}\cap\dot H^1)}^2,
\end{align}
for any $0\leq t \leq T_1$.
\end{lem} 
\begin{proof}
Taking the $L^2_{x,v}$-inner product of \eqref{G3.48} with $\{\mathbf I - \mathbf P\}f$, we obtain  
\begin{align}\label{G3.61}
&\frac{1}{2}\frac{{\rm d}}{{\rm d}t} \|\{\mathbf{I}-\mathbf{P}\}f\|_{L_{x,v}^2}^2+ \langle-\mathcal{L}\{\mathbf{I}-\mathbf{P}\}f, \{\mathbf{I}-\mathbf{P}\}f\rangle _{x,v} \nonumber\\
=&\,\langle \Gamma(f,f),\{\mathbf{I}-\mathbf{P}\}f \rangle_{x,v}+ \langle  \mathbf{P}(v\cdot\nabla_{x} f)-v\cdot\nabla_{x} \mathbf{P}f ,  \{\mathbf{I}-\mathbf{P}\}   f\rangle_{x,v} \nonumber\\
&+  \Big\langle   \mathbf{P}\Big[ \frac{1}{2}E\cdot vf-E\cdot\nabla_v f  \Big],  \{\mathbf{I}-\mathbf{P}\}  f  \Big\rangle_{x,v} +\Big\langle     \frac{1}{2}E\cdot v\mathbf{P}f-E\cdot\nabla_v \mathbf{P}f  ,  \{\mathbf{I}-\mathbf{P}\}  f  \Big\rangle_{x,v}\nonumber\\
&+\Big\langle    \frac{1}{2}E\cdot v 
\{\mathbf{I}- \mathbf{P}\}f    ,  \{\mathbf{I}-\mathbf{P}\}  f  \Big\rangle_{x,v}+  \Big\langle   \mathbf{P}\Big[ \frac{1}{2}\nabla_{x}\phi\cdot vf-\nabla_{x}\phi\cdot\nabla_v f  \Big],  \{\mathbf{I}-\mathbf{P}\}  f  \Big\rangle_{x,v}\nonumber\\
&+\Big\langle     \frac{1}{2}\nabla_{x}\phi\cdot v\mathbf{P}f-\nabla_{x}\phi\cdot\nabla_v \mathbf{P}f  ,  \{\mathbf{I}-\mathbf{P}\}  f  \Big\rangle_{x,v}+\Big\langle    \frac{1}{2}\nabla_{x}\phi\cdot v 
\{\mathbf{I}- \mathbf{P}\}f    ,  \{\mathbf{I}-\mathbf{P}\}  f  \Big\rangle_{x,v}  \nonumber\\
\equiv:&\, \sum_{j=1}^{8}\mathcal{I}_{j}.
\end{align}
Owing to \eqref{G2.2}, one gets
\begin{align} \label{G3.62}
 \langle-\mathcal{L}\{\mathbf{I}-\mathbf{P}\}f, \{\mathbf{I}-\mathbf{P}\}f\rangle _{x,v}    \geq \lambda_0 \|\{\mathbf{I}-\mathbf{P}\}f\|_{\nu}^2.
\end{align}
The terms $\mathcal{I}_1,\dots, \mathcal{I}_5$ are the same as those in \cite[Lemma 3.8]{DN-2026}. We state the results as follows and omit the details for brevity. It holds that
\begin{align}\label{G3.63}
\sum_{j=1}^5\mathcal{I}_{j}\lesssim    \|f\|_{L_v^2( \dot H^{\frac{3}{4}}\cap\dot H^1)}^2+\Big(\frac{\lambda_0}{3}+C\delta\Big)\|\{\mathbf{I}-\mathbf{P}\}f\|_{\nu}^2.  
\end{align}
Since $v^k\sqrt{M}\lesssim 1$ holds for any $k\geq 0$,  leveraging inequality \eqref{G2.16}, we have
\begin{align}\label{G3.64}
\mathcal{I}_{6}+\mathcal{I}_{7} \leq  C\|\nabla_{x}\phi f\|_{L_{x,v}^2}\|\{\mathbf{I}-\mathbf{P}\}f\|_{\nu}
\leq&\,C   {\|\nabla_{x}\phi\|_{\dot B_{2,\infty}^{\frac{1}{2}}}}\|f\|_{L_{v}^2(\dot H^1)}\|\{\mathbf{I}-\mathbf{P}\}f\|_{\nu}\nonumber\\
\leq&\, C \|f\|_{L_v^2( \dot H^1)}^2+\frac{\lambda_0}{3}\|\{\mathbf{I}-\mathbf{P}\}f\|_{\nu}^2.
\end{align}
Similarly, by direct calculation, one has  
\begin{align}\label{G3.65}
\mathcal {I}_{8}\lesssim    \|\nabla_{x}\phi\|_{L^\infty} \|\{\mathbf{I}-\mathbf{P}\}f\|_{\nu}^2 \lesssim \delta \|\{\mathbf{I}-\mathbf{P}\}f\|_{\nu}^2.  
\end{align}
Plugging the estimates \eqref{G3.62}--\eqref{G3.65} into \eqref{G3.61} yields \eqref{G3.60}. Hence, we complete the proof of Lemma \ref{L3.8}.
\end{proof}

With the aid of Lemmas \ref{L3.6}--\ref{L3.8}, by combining \eqref{G3.47}, \eqref{G3.54} and \eqref{G3.60}, we further obtain the estimate of the microscopic part $\{\mathbf{I}-\mathbf{P}\}f$ as follows.

\begin{cor}\label{cor3.9}
For strong solutions of the problem   \eqref{I1.4}, there exists a positive constant $\eta_{10}>0$ such that
\begin{align} \label{G3.66}
&\frac{{\rm d}}{{\rm d}t}\bigg(\|\nu\{\mathbf{I}-\mathbf{P}\}f(t) \|_{L_v^2(\dot H^1\cap\dot H^{N-1})}^2 +\|\{\mathbf{I}-\mathbf{P}\}f(t)\|_{L_{x,v}^2}^2+C_{\beta} \sum_{\substack{ 1\leq |\beta|\leq N \\|\alpha|+|\beta| \leq N}}\|\partial^\alpha_{x}\partial^\beta_{v}\{\mathbf{I}-\mathbf{P}\}f(t) \|_{L_{x,v}^2 }^2\bigg)
 \nonumber\\
&+ \eta_{10} \bigg(\|\nu^{\frac{3}{2}}\{\mathbf{I}-\mathbf{P}\}f(t)\|_{L_v^2(\dot H^1\cap\dot H^{N-1})}^2+ \|\{\mathbf{I}-\mathbf{P}\}f(t)\|_{\nu}^2+\sum_{\substack{ 1\leq |\beta|\leq N \\|\alpha|+|\beta| \leq N}} \| \partial^{\alpha}_{x}\partial^\beta_{v}\{\mathbf{I}-\mathbf{P}\}f(t)\|_{\nu}^2\bigg)\nonumber\\    
&\quad\lesssim \|f(t)\|_{L_v^2(\dot H^\frac{3}{4}\cap\dot H^N)}^2+\|\nu^\frac{1}{2}\{\mathbf{I}-\mathbf{P}\}f(t)\|_{L_v^2(\dot H^1\cap\dot H^{N-1})}^2+\sup_{0\leq t\leq T_1}\|E(t)\|_{H^N}^2,
\end{align}
for any $0\leq t \leq T_1$.
\end{cor}

\subsection{Proof of global existence}
Based on the Lemmas \ref{L3.1} and \ref{L3.5}, Corollaries \ref{cor3.4} and \ref{cor3.9}, we are now in a position to give the proof of Theorem \ref{Th1}.
\begin{proof}[Proof of Theorem \ref{Th1}]
Similar to the argument in \cite{DN-2026}, we next define the following temporal energy functional $\mathcal{E}^H(t)$ and the corresponding dissipation rate $\mathcal{D}^H(t)$:

\begin{align}\label{G3.67}
\mathcal{E}^H(t):=&\, \|f(t)\|_{L_{v}^2(\dot H^1\cap \dot H^N)}^2+\|\nabla_{x}\phi(t)\|_{\dot H^1\cap\dot H^N}^2+\vartheta_1\sum_{k=1}^{N-1}\mathfrak{E}_{ k}(t)+\vartheta_2\|\nu\{\mathbf{I}-\mathbf{P}\}f(t) \|_{L_v^2(\dot H^1\cap\dot H^{N-1})}^2\nonumber\\
& + \vartheta_2\|\{\mathbf{I}-\mathbf{P}\}f(t)\|_{L_{x,v}^2}^2+ \vartheta_2C_{\beta}\sum_{\substack{ 1\leq |\beta|\leq N \\|\alpha|+|\beta| \leq N}}\|\partial^\alpha_{x}\partial^\beta_{v}\{\mathbf{I}-\mathbf{P}\}f(t) \|_{L_{x,v}^2 }^2,\\\label{G3.68}
\mathcal{D}^H(t):=&\,\|(a,b,c)(t)\|^2_{\dot H^2\cap\dot H^N}+\|\rho_{f}(t)\|^2_{\dot H^1\cap\dot H^{N-1}}+\|\nu^{\frac{3}{2}}\{\mathbf{I}-\mathbf{P}\}f(t)\|_{L_v^2(\dot H^1\cap\dot H^{N-1})}^2 \nonumber\\
&+\sum_{
\ell\leq N} \|\nabla^\ell\{\mathbf{I}-\mathbf{P}\}f(t)\|_{\nu}^2+\sum_{\substack{ 1\leq |\beta|\leq N \\|\alpha|+|\beta| \leq N}} \| \partial^{\alpha}_{x}\partial^\beta_{v}\{\mathbf{I}-\mathbf{P}\}f(t)\|_{\nu}^2,
\end{align}
where $0 < \vartheta_2 \ll \vartheta_1 \ll 1$ are sufficiently small constants.
By adding \eqref{G3.31}, $\vartheta_1\times$ \eqref{G3.33}
and $\vartheta_2\times$ \eqref{G3.66}, we obtain
\begin{align}\label{G3.69}
\frac{{\rm d}}{{\rm d}t} \mathcal{E}^H(t)+  \eta_{11} \Big(\mathcal{D}^H(t) -\vartheta_2 \|f\|_{L_v^2(\dot H^{\frac{3}{4}}\cap\dot H^2)}^2  \Big)\lesssim \sup_{0\leq t\leq T_1} \|E(t)\|_{H^N}^2,
\end{align}
for some constant $\eta_{11}>0$.
Adding the term $\eta_{11}\Big(\|f_{L}(t)\|_{L_v^2(\dot B_{2,\infty}^{\frac{1}{2}})}^2+\|\nabla_{x}\phi_{L}\|_{\dot B_{2,\infty}^{\frac{1}{2}}}^2\Big)$ to both sides of \eqref{G3.69} gives
\begin{align}\label{G3.70}
&\frac{{\rm d}}{{\rm d}t} \mathcal{E}^H(t)+  \eta_{12} \Big(\mathcal{D}^H(t) + \|f(t)\|_{L_v^2(\dot B_{2,\infty}^{\frac{1}{2}}\cap\dot H^N)}^2+\|\nabla_{x}\phi(t)\|_{\dot B_{2,\infty}^{\frac{1}{2}}\cap\dot H^N}^2  \Big)\nonumber\\
&\quad\lesssim \sup_{0\leq t\leq T_1} \|E(t)\|_{H^N}^2+ \|f_{L}(t)\|_{L_v^2(\dot B_{2,\infty}^{\frac{1}{2}}) }^2+\|\nabla_{x}\phi_{L}(t)\|_{\dot B_{2,\infty}^{\frac{1}{2}} } ^2,
\end{align}
for some constant $\eta_{12}>0$,
where we used the smallness of $\vartheta_2$.
 
From definitions \eqref{G3.67}–\eqref{G3.68}, it is easy to observe that
\begin{align*}
\mathcal{E}^H(t)\backsim&\, \|f(t)\|_{L_{v}^2(\dot H^1\cap \dot H^N)}^2+\|\nabla_{x}\phi(t)\|_{\dot H^1\cap\dot H^N}^2 +\|\langle v\rangle f(t) \|_{L_v^2(\dot H^1\cap\dot H^{N-1})}^2 +\|\{\mathbf{I}-\mathbf{P}\}f(t)\|_{H_{x,v}^N}^2, 
\end{align*}
and
\begin{align}\label{G3.71}
\mathcal{E}^H(t)\lesssim \mathcal{D}^H(t) + \|f(t)\|_{L_v^2(\dot B_{2,\infty}^{\frac{1}{2}}\cap\dot H^N)}^2+\|\nabla_{x}\phi(t)\|_{\dot B_{2,\infty}^{\frac{1}{2}}\cap\dot H^N}^2.
\end{align}
Combining \eqref{G3.70} and \eqref{G3.71}, we arrive at
\begin{align*}
&\frac{{\rm d}}{{\rm d}t} \mathcal{E}^H(t)+  \eta_{13}   \mathcal{E}^H(t)\lesssim \sup_{0\leq t\leq T_1} \|E(t)\|_{H^N}^2+ \|f_{L}(t)\|_{L_v^2(\dot B_{2,\infty}^{\frac{1}{2}} )}^2+\|\nabla_{x}\phi_{L}(t)\|_{\dot B_{2,\infty}^{\frac{1}{2}} }^2,
\end{align*}
for some constant $\eta_{13}>0$, which together with \eqref{G3.2} gives
\begin{align*}
&\frac{{\rm d}}{{\rm d}t} \mathcal{E}^H(t)+  \eta_{14}   \mathcal{E}^H(t)\lesssim \sup_{0\leq t\leq T_1} \|E(t)\|_{\dot B_{2,\infty}^{-\frac{3}{2}}\cap\dot H^N}^2+ \|f_{0}\|_{L_v^2(\dot B_{2,\infty}^{\frac{1}{2}} )}^2+\|\nabla_{x}\phi_{0}\|_{\dot B_{2,\infty}^{\frac{1}{2}} }^2,
\end{align*}
for some constant $\eta_{14}>0$.
By utilizing Gr\"{o}nwall's inequality, we have
\begin{align}\label{G3.72}
\mathcal{E}^H(t)\lesssim&\, e^{-\eta_{14}t}\mathcal{E}^H(0)+\bigg\{\sup_{0\leq t\leq T_1}\|E(t)\|_{\dot B_{2,\infty}^{-\frac{3}{2}}\cap\dot H^N}^2+\|f_{0}\|_{L_v^2(\dot B_{2,\infty}^{\frac{1}{2}} )}^2+\|\nabla_{x}\phi_{0}\|_{\dot B_{2,\infty}^{\frac{1}{2}} }^2\bigg\}\int _0^{\infty} e^{-\eta_{14}\tau} {\rm d}\tau \nonumber\\
\lesssim&\,   \sup_{0\leq t\leq T_1} \|E(t)\|_{\dot B_{2,\infty}^{-\frac{3}{2}}\cap\dot H^N}^2 +\mathcal{E}^H(0)+ \|f_{0}\|_{L_v^2(\dot B_{2,\infty}^{\frac{1}{2}} )}^2+\|\nabla_{x}\phi_{0}\|_{\dot B_{2,\infty}^{\frac{1}{2}} }^2,
\end{align}
for all $0\leq t \leq T_1$. Therefore, from \eqref{G3.2}, we get
\begin{align} \label{G3.73}
  &\|f(t)\|_{L_{v}^2(\dot B_{2,\infty}^{\frac{1}{2}}\cap \dot H^N)} +\|\nabla_{x}\phi(t)\|_{\dot B_{2,\infty}^{\frac{1}{2}}\cap\dot H^N} +\|\langle v\rangle f(t) \|_{L_v^2(\dot H^1\cap\dot H^{N-1})}  +\|\{\mathbf{I}-\mathbf{P}\}f(t)\|_{H_{x,v}^N} \nonumber\\
  &\quad\lesssim  \sup_{0\leq t\leq T_1}\|E(t)\|_{{\dot B_{2,\infty}^{-\frac{3}{2}}\cap \dot H^{N}}}+\|f_0\|_{L_v^2(\dot B^{\frac{1}{2}}_{2,\infty}\cap \dot H^N)}+\|\nabla_{x}\phi_0\|_{\dot B_{2,\infty}^{\frac{1}{2}}\cap\dot H^N}\nonumber\\
& \quad\quad+ \|\langle v\rangle f_0\|_{L_v^2(\dot H^1\cap \dot H^{N-1})} +{\|\{\mathbf{I}-\mathbf{P}\}f_0\|_{H_{x,v}^N}},
\end{align}
for all $0\leq t\leq T_1$.

The {\it local-in-time} existence of strong solutions to the Cauchy problem \eqref{I1.4}
can be obtained by an argument similar to that in \cite[Section 3]{Gy-CPAM-2002};
therefore, we omit the details here. 
Combining the {\it global a priori estimates} established above with a standard continuity
argument, we extend the local solution globally in time. Consequently, \eqref{G3.73}
holds for all \(t\geq 0\). Finally, by applying the maximum principle as in \cite{GY-iumj-2004}, we obtain  
\begin{align*}
F(t,x,v)=M+\sqrt{M}f(t,x,v)\geq 0 .
\end{align*}
Hence, we complete the proof of Theorem \ref{Th1}.
\end{proof}

\section{Uniform stability of Vlasov-Poisson-Boltzmann system}\label{sec4}
In this section, we study the uniform stability of the Cauchy problem \eqref{I1.4}. First, we present the following error equations between $\big(f^{(1)},\phi^{(1)}\big)$ and $\big(f^{(2)},\phi^{(2)}\big)$:
\begin{equation}\label{G4.1}
\left\{
\begin{aligned}
& \partial_{t} \widetilde{f}+v\cdot\nabla_{x}\widetilde{f}-\mathcal{L}\widetilde{f}-\nabla_{x}\widetilde \phi\cdot v\sqrt{M}=\Gamma(f^{(1)}+f^{(2)},\widetilde{f})-(E+\nabla_{x}\phi^{(1)})\cdot\nabla_{v}\widetilde{f}+\frac{1}{2} (E+\nabla_{x}\phi^{(1)})\cdot v\widetilde{f} \\
&\qquad\qquad\qquad\qquad\qquad\qquad\qquad\qquad-\nabla_{x}\widetilde \phi \cdot \nabla_{v}  f^{(2)}+\frac{1}{2} \nabla_{x}\widetilde\phi\cdot v f^{(2)} ,\\
& \Delta_{x}\widetilde\phi= \int_{\mathbb R^3}\sqrt{M}\widetilde f{\rm d}v=\mathbf{P}_0\widetilde f, \\
&\widetilde{f}(t,x,v)|_{t=0}={\widetilde{f}_0(x,v)=f_0^{(1)}(x,v)-f_0^{(2)}(x,v)},
\end{aligned}
\right.
\end{equation}
where we utilized the bilinearity and symmetry of the operator $\Gamma$, and the differences between $\widetilde f(t,x,v)$ and $\widetilde \phi$ are denoted by
\begin{align*}
 \widetilde f(t,x,v)=f^{(1)}(t,x,v)-f^{(2)}(t,x,v),\qquad    \widetilde \phi=\phi^{(1)}(t,x)-\phi^{(2)}(t,x).
\end{align*}
 
Under the assumption \eqref{G1.8} in Theorem \ref{Th2}, we now establish the
time-decay estimates for system \eqref{G4.1}. The main difficulty comes from the
nonlinear terms in \eqref{G4.1}, such as
$(E+\nabla_x\phi)\cdot\nabla_v\widetilde f$ and 
$\frac12(E+\nabla_x\phi)\cdot v\widetilde f$, 
which involve either the velocity derivative \(\nabla_v\) or the velocity
weight \(v\), and hence do not directly yield decay estimates at low spatial
frequencies. 
To handle this difficulty, we follow the strategy in \cite[Section 4]{DN-2026}, where the high-frequency decay is used to compensate for the lack of direct low-frequency decay caused by the velocity derivative and velocity weight.
 
\subsection{Energy estimates at high frequencies}
We begin by establishing the estimates for \(\widetilde f(t,x,v)\) in
\(L_v^2(\dot H^1\cap \dot H^{N-1})\), together with the corresponding estimates
for \(\nabla_x\widetilde\phi(t,x)\) in \(\dot H^1\cap \dot H^{N-1}\), where
\(N\geq 4\).

\begin{lem}\label{L4.1}
For strong solutions of system \eqref{G4.1}, there exists a positive constant  $\widetilde \eta_{1} >0$ such that 
\begin{align} \label{G4.2}
&\frac{{\rm d}}{{\rm d}t}\Big(\| \widetilde  f(t)\|_{L_v^2(\dot H^1\cap \dot H^{N-1})}^2+\|\nabla_{x}\widetilde \phi(t)\|_{\dot H^1\cap\dot H^{N-1}}^2\Big)+\widetilde \eta_1\sum_{1\leq k\leq N-1}\|\nabla^k\{\mathbf{I}-\mathbf{P}\}\widetilde f(t)\|_{\nu}^2  \nonumber\\
\leq&\,  (\kappa+C\delta_0)  \|  (\widetilde a,\widetilde b,\widetilde c)(t)\|_{ \dot H^1\cap\dot H^{N-1}}^2 +C\delta_0\sum_{\substack{ 1\leq |\beta|\leq N -1\\|\alpha|+|\beta| \leq N-1}} \| \partial^{\alpha}_{x}\partial^\beta_{v}\{\mathbf{I}-\mathbf{P}\}\widetilde f(t)\|_{\nu}^2{+C\delta_0\|\nabla_x\widetilde\phi(t)\|_{\dot H^1}^2},
\end{align}
for any $t\geq 0$.   Here, \(\delta_0\) is as in \eqref{G1.6}, and \(0<\kappa<1\) is chosen
sufficiently small according to \eqref{G3.25}.
\end{lem}
\begin{proof}
Let \(1\leq |\alpha|\leq N-1\). Applying \(\partial_x^\alpha\) to
\eqref{G4.1}$_1$--\eqref{G4.1}$_2$, then taking the
\(L^2_{x,v}\)-inner product with \(\partial_x^\alpha\widetilde f\), we obtain
\begin{align}\label{G4.3}
&\frac{{\rm d}}{{\rm d}t}\Big(\|\partial^\alpha \widetilde f\|_{L_{x,v}^2}^2+{\|\partial^\alpha_{x}\nabla_{x}\widetilde\phi\|_{L^2}^2}\Big)+\widetilde\eta_2 \|\partial^\alpha \{\mathbf{I}-\mathbf{P}\}\widetilde f\|_{\nu}^2\nonumber\\
\leq&\, |\langle\partial^\alpha\Gamma(f^{(1)}+f^{(2)},\widetilde f) ,\partial^\alpha \widetilde f\rangle_{{x,v}} | +|\langle\partial^ \alpha(E\cdot v\widetilde f) ,\partial^\alpha \widetilde f\rangle_{{x,v}} |+|\langle\partial^ \alpha(E\cdot \nabla_v\widetilde f) ,\partial^\alpha \widetilde f\rangle_{{x,v}} |\nonumber\\
&+|\langle\partial^ \alpha(\nabla_{x}{\phi^{(1)}}\cdot v\widetilde f) ,\partial^\alpha \widetilde f\rangle_{{x,v}} |+|\langle\partial^ \alpha(\nabla_{x}{\phi^{(1)}}\cdot \nabla_v\widetilde f) ,\partial^\alpha \widetilde f\rangle_{{x,v}} |\nonumber\\
&+{|\langle\partial_x^\alpha(\nabla_x\widetilde\phi\cdot vf^{(2)}),\partial_x^\alpha\widetilde f\rangle_{x,v}|}+{|\langle\partial_x^\alpha(\nabla_x\widetilde\phi\cdot\nabla_vf^{(2)}),\partial_x^\alpha\widetilde f\rangle_{x,v}|}\nonumber\\
 \equiv:&\,{\sum_{j=1}^7\widetilde I_j},
\end{align}
for some constant $\widetilde\eta_2>0$. 

The terms \(\widetilde I_1,\ldots,\widetilde I_3\) can be treated  in the same way as in \cite[Lemma 4.1]{DN-2026}. Therefore, we have
\begin{align}\label{G4.4}
\sum_{j=1}^3\widetilde I_{j}\lesssim&\, (\kappa+C\delta_0)\|\partial^\alpha \{\mathbf{I}-\mathbf{P}\}\widetilde f\|_{\nu}^2+(\kappa+C\delta_0) \bigg(\sum_{1\leq\ell\leq N-1}\|\nabla^\ell\{\mathbf{I}-\mathbf{P}\}\widetilde f\|_{\nu}^2+\|(\widetilde a,\widetilde b,\widetilde c)\|_{\dot H^1\cap {{\dot H^{N-1}}}}^2\bigg) \nonumber\\
&+C\delta_0\sum_{\substack{ 1\leq |\beta|\leq {N-1} \\|\alpha|+|\beta| \leq {N-1}}} \| \partial^{\alpha}_{x}\partial^\beta_{v}\{\mathbf{I}-\mathbf{P}\}\widetilde f\|_{\nu}^2.
\end{align}
Similar to \eqref{G3.27}--\eqref{G3.29}, we find that
\begin{align}\label{G4.5}
\widetilde I_{4}+\widetilde I_5\leq &\, \kappa \sum_{2\leq\ell\leq N-1}\|{\nabla^\ell}\mathbf{P}\widetilde f\|_{L_{x,v}^2}^2+ C \|(  \nabla_{x}\phi^{{(1)}}\cdot v\widetilde f, \nabla_{x}\phi^{{(1)}}\cdot \nabla_{v}\widetilde f)\|_{L_v^2(H^{N-3})} ^2\nonumber\\
&+C \|\nabla_{x}\phi^{{(1)}} \widetilde f\|_{L_v^2(\dot H^{N-1})} \Big( \|\nabla
^{N-1}\{\mathbf{I}-\mathbf{P}\}\widetilde f\|_{\nu}+\|\mathbf{P}\widetilde f\|_{L_v^2(\dot H^{N-1})}  \Big)\nonumber\\
&+C\|\nabla^{N-1}(\nabla_{x}\phi^{{(1)}}\cdot\nabla_{v}\widetilde f)-\nabla_{x}\phi^{{(1)}}\nabla^{N-1}\nabla_v\widetilde f\|_{L_{x,v}^2} \|\nabla^{N-1}\widetilde f\|_{L_{x,v}^2} \nonumber\\
\leq&\, \kappa\|(\widetilde a,\widetilde b,\widetilde c)\|_{\dot H^1\cap \dot H^{N-1}}^2+(C\delta_0+\kappa)\bigg(\sum_{1\leq\ell\leq {N-1}}\|\nabla^\ell\{\mathbf{I}-\mathbf{P}\}\widetilde f\|_{\nu}^2+\|(\widetilde a,\widetilde b,\widetilde c)\|_{\dot H^1\cap \dot H^{N-1}}^2\bigg)\nonumber\\
&+C\delta_0\sum_{\substack{ 1\leq |\beta|\leq {N-1} \\|\alpha|+|\beta| \leq {N-1}}} \| \partial^{\alpha}_{x}\partial^\beta_{v}\{\mathbf{I}-\mathbf{P}\}\widetilde f\|_{\nu}^2,     
\end{align}
and
\begin{align}\label{DNYG4.6}
\widetilde I_6+\widetilde I_7
\lesssim &\,
 \big\|\partial_x^\alpha
\big(\nabla_x\widetilde\phi\cdot vf^{(2)},
     \nabla_x\widetilde\phi\cdot\nabla_v f^{(2)}\big)
\big\|_{L^2_{x,v}}
\|\partial_x^\alpha\widetilde f\|_{L^2_{x,v}}\nonumber\\
\lesssim&\,\big(
\|\langle v\rangle f^{(2)}\|_
{L_v^2(\dot H^1\cap\dot H^{N-1})}
+\|\{\mathbf I-\mathbf P\}f^{(2)}\|_{H^N_{x,v}}
\big)
\|\nabla_x\widetilde\phi\|_
{\dot H^1\cap\dot H^{N-1}}
\|\widetilde f\|_{L_v^2(\dot H^1\cap\dot H^{N-1})}
\nonumber\\
\lesssim&\,
 \delta_0
\big(
\|\nabla_x\widetilde\phi\|_{\dot H^1}
+\|\widetilde f\|_{L_v^2(\dot H^1\cap\dot H^{N-1})}
\big)
\|\widetilde f\|_{L_v^2(\dot H^1\cap\dot H^{N-1})}
\nonumber\\
\lesssim&\,
C\delta_0\big(
\|\nabla_x\widetilde\phi\|_{\dot H^1}^2
+\|\widetilde f\|_{L_v^2(\dot H^1\cap\dot H^{N-1})}^2
\big).
\end{align}
Combining \eqref{G4.3} with the estimates \eqref{G4.4}--\eqref{DNYG4.6}, and then summing over \(1\leq |\alpha|\leq N-1\), yields \eqref{G4.2}.
\end{proof}

We next turn to the estimates for the macroscopic coefficients \((\widetilde a,\widetilde b,\widetilde c)\). 
By combining \eqref{G2.9} with \eqref{G2.10}, we derive the following macroscopic balance equations for \((\widetilde a,\widetilde b,\widetilde c)\):
 \begin{equation} \label{G4.6}
\left\{
\begin{aligned}
&\partial_t\widetilde \rho_f+\nabla_x\cdot \widetilde b=0,
\\
&\partial_t \widetilde b+\nabla_x(\widetilde\rho_f+2\widetilde c)
+\nabla_x\cdot\Theta(\{\mathbf I-\mathbf P\}\widetilde f)-\nabla_{x}\widetilde\phi
=\widetilde \rho_f(\nabla_{x}\phi^{(1)}+E)+\rho_{f}^{(2)}\nabla_{x}\widetilde\phi,\\
&\partial_t \widetilde c+\frac13\nabla_x\cdot \widetilde b+\frac16\nabla_x\cdot\Xi(\{\mathbf I-\mathbf P\}\widetilde f)
=
\frac13 \widetilde b\cdot(\nabla_{x}\phi^{(1)}+E)+\frac{1}{3} b^{(2)}\cdot\nabla_{x}\widetilde \phi.
\end{aligned}
\right.
\end{equation}
Besides, the difference of the self-consistent potentials satisfies the following Poisson equation:
\begin{align}\label{G4.7}
\Delta_x\widetilde\phi
=
\widetilde\rho_f
=
\widetilde a+3\widetilde c,
\end{align}
which is equivalent to  
\begin{align*}
\nabla_x\widetilde\phi
=
\nabla_x\Delta_x^{-1}\widetilde\rho_f.     
\end{align*}
Thus, equations \eqref{G4.6}--\eqref{G4.7} form an Euler-Poisson type macroscopic system for \((\widetilde\rho_f,\widetilde b,\widetilde c,\widetilde\phi)\).

In addition, we have the error equation of the microscopic component $\{\mathbf I-\mathbf P\}f$ as follows:
\begin{equation}\label{G4.8}
\left\{
\begin{aligned}
&\partial_t\big[A_{ii}(\{\mathbf I-\mathbf P\}\widetilde f)+2\widetilde c\big]+2\partial_i \widetilde b_i
=
A_{ii}(\widetilde{\mathcal  R}+\widetilde{\mathcal{G}}),
\quad \,\,\,\,1\le i\le3,
\\
&\partial_tA_{ij}(\{\mathbf I-\mathbf P\}\widetilde f)+\partial_i \widetilde b_j+\partial_j \widetilde b_i
=
A_{ij}(\widetilde{\mathcal R}+\widetilde {\mathcal G}),
\qquad\, 1\le i\ne j\le3,\\
&\partial_t B_i(\{\mathbf I-\mathbf P\}\widetilde f)+\partial_i \widetilde c
=
B_i(\widetilde{\mathcal R}+\widetilde{\mathcal{G}}),
\qquad\qquad\qquad \,\, 1\le i\le3,
\end{aligned}
\right.
\end{equation}
where $\widetilde{\mathcal R}$ and $\widetilde{\mathcal{G}}$ are given by
\begin{equation*}
\left\{
\begin{aligned}
&\widetilde{\mathcal R}
=
-v\cdot\nabla_x \{\mathbf I-\mathbf P\}\widetilde f+\mathcal L\{\mathbf I-\mathbf P\}\widetilde f,
\\
&\widetilde{\mathcal{G}}=\Gamma(f^{(1)}+f^{(2)},\widetilde{f})
-(\nabla_{x}\phi^{(1)}+E)\cdot\nabla_v \widetilde  f
+\frac12[v\cdot(\nabla_{x}\phi^{(1)}+E)]\widetilde f\nonumber\\
&\qquad-\nabla_{x}\widetilde\phi\cdot\nabla_{v}f^{(2)}+\frac{1}{2}v\cdot\nabla_{x}\widetilde\phi f^{(2)}.
\end{aligned}
\right.
\end{equation*}
 
Analogously to \eqref{macro}, we define the following time-dependent energy functional \(\widetilde{\mathfrak E}_{k}(t)\) associated with
\(\widetilde f(t,x,v)\):
 \begin{align}\label{macro2}
\widetilde{\mathfrak E}_k(t)
:=&\,
\sum_{|\alpha|=k}\sum_{i,j=1}^3
\int_{\mathbb R^3_x}
\partial^\alpha_{x}(\partial_i \widetilde b_j+\partial_j \widetilde b_i) 
\partial^\alpha_{x}{A_{ij}(\{\mathbf{I}-\mathbf{P}\}\widetilde f)  }{\rm d}x
\nonumber\\
&+
\sum_{|\alpha|=k}\sum_{i=1}^3
\int_{\mathbb R^3_x}
\partial^\alpha_{x}\partial_i \widetilde c
\partial^\alpha B_i(\{\mathbf{I}-\mathbf{P}\}\widetilde f){\rm d}x-\frac{1}{4}
\sum_{|\alpha|=k}
\int_{\mathbb R^3_x}
\partial^\alpha_{x}\widetilde\rho_{f}
\partial^\alpha_{x}\nabla_x\cdot \widetilde b{\rm d}x,
\end{align}
for any $k=1,2,\dots N-2$. Through direct calculation, we also have
\begin{align*}
 \sum_{k=1}^{N-2}|\widetilde{\mathfrak E}_{k}(t) |\lesssim  \|(\widetilde a,\widetilde b,\widetilde c)\|_{ \dot H^1\cap\dot H^{N-1}}^2+\|\{\mathbf{I}-\mathbf{P}\}\widetilde f\|_{L_v^2(\dot H^1\cap\dot H^{N-1})}^2\lesssim \|\widetilde f\|_{L_v^2(\dot H^1\cap\dot H^{N-1})}^2. 
\end{align*}

Then, we present the estimate of $\|(\widetilde a,\widetilde b,\widetilde c)\|_{\dot H^2\cap\dot H^{N - 1}}$.

\begin{lem}\label{L4.2}
For strong solutions of system \eqref{G4.1}, there exists a positive constant $\widetilde \eta_3>0$    such that 
\begin{align}\label{G4.10} 
&\frac{{\rm d}}{{\rm d}t}\sum_{1\leq k\leq N-2}\widetilde{\mathfrak{E}}_{ k}(t)+\widetilde\eta_3 \big( \|(\widetilde a,\widetilde b,\widetilde c)(t)\|_{\dot H^2\cap\dot H^{N-1}}^2+\|\widetilde \rho_{f}(t)\|_{\dot H^1\cap{\dot H^{N-2}}}^2 \big) \nonumber\\
&\quad\lesssim \|  \{\mathbf{I}-\mathbf{P}\}\widetilde f(t)\|_{L_{v}^2(\dot H^1\cap \dot H^{N-1})}^2+\delta_0\big(\|(\widetilde a,\widetilde b,\widetilde c)(t)\|_{\dot H^1}^2+\|\nabla_{x}\widetilde\phi(t)\|_{\dot H^1}^2\big),
\end{align}
for    any $t\geq 0$,  where $\widetilde {\mathfrak{E}}_{k}(t)$ is defined by \eqref{macro2}.  Here, $\delta_0$ is defined in \eqref{G1.6}.
\end{lem}
\begin{proof}
The proof of Lemma \ref{L4.2} follows a similar line of argument as that in Lemma \ref{L3.5}.
We first capture the dissipation of $\widetilde b$. Based on the equations \eqref{G4.8}$_1$--\eqref{G4.8}$_2$, it follows that
\begin{align} \label{G4.11}
& \frac{{\rm d}}{{\rm d}t}
\sum_{i,j=1}^3
\int_{\mathbb R^3_{x}}
\partial^\alpha_{x} (\partial_i \widetilde b_j+\partial_j \widetilde b_i)
\partial^\alpha_{x}\big(A_{ij}(\{\mathbf{I}-\mathbf{P}\}\widetilde f)+2\widetilde c\delta_{ij}\big){\rm d}x+2\|\partial^\alpha_{x}\nabla_x \widetilde b\|^2_{L^2}
+
2\|\partial^\alpha_{x}\nabla_x\cdot \widetilde b\|^2_{L^2} \nonumber\\
=&\,
\sum_{i,j=1}^3
\int_{\mathbb R^3_{x}}
\partial^\alpha_{x}\partial_t (\partial_i \widetilde b_j+\partial_j \widetilde b_i)
\partial^\alpha_{x} A_{ij}(\{\mathbf{I}-\mathbf{P}\}\widetilde f) {\rm d}x
+2
\sum_{i,j=1}^3
\int_{\mathbb R^3_{x}}
\partial^\alpha_{x}\partial_t (\partial_i \widetilde b_j+\partial_j \widetilde b_i)
\partial^\alpha_{x} \widetilde c\delta_{ij}  {\rm d}x
\nonumber\\
&
+
\sum_{i,j=1}^3
\int_{{\mathbb R^3_{x}}}
\partial^\alpha_{x} (\partial_i \widetilde b_j+\partial_j \widetilde b_i)
\partial^\alpha_{x}A_{ij}(\widetilde{\mathcal R}+\widetilde{\mathcal G}){\rm d}x\nonumber\\
\equiv:&\,\sum_{i=1}^3\widetilde J_{i},
\end{align}
for $|\alpha| = k$, where $1\leq k\leq N - 2$.
From \eqref{G4.6}$_2$, one has
\begin{align}\label{G4.12}
\widetilde J_1
=&\,2\sum_{i,j=1}^3 \int_{\mathbb R^3_{x}} \partial^\alpha_{x}\big[ \partial_{i} \widetilde\rho_{f}+2\partial_{i}\widetilde c+\sum_{m=1}^3\partial_{m} \Theta_{im}(\{\mathbf{I}-\mathbf{P}\}\widetilde f)-\partial_{i}\widetilde\phi \big]  \partial_{j}\partial^\alpha_{x} A_{ij}(\{\mathbf{I}-\mathbf{P}\} \widetilde f) {\rm d}x  \nonumber\\
&-2\sum_{i,j=1}^3 \int_{\mathbb R^3_{x}} \partial^\alpha_{x} \big[  \widetilde\rho_{f}(\partial_{i}\phi^{(1)}+E_{i})+\rho_{f}^{(2)}{\partial_{i}}\widetilde \phi\big]\partial_{j}\partial^\alpha_{x} A_{ij}(\{\mathbf{I}-\mathbf{P}\}\widetilde f)  {\rm d}x\nonumber\\
\leq&\, \kappa \|\nabla_x\partial^\alpha_{x}(\widetilde  \rho_{f},\widetilde c)\|^2_{L^2}+(\kappa+\delta_0) \|   \widetilde\rho_{f}\|_{\dot H^1\cap\dot H^{N-1}}^2+C\delta_0\|\nabla_{x}\widetilde \phi\|_{\dot H^1}^2+C\| \{\mathbf{I}-\mathbf{P}\}\widetilde f \|_{L_v^2(\dot H^1\cap\dot H^{N-1})}^2.    
\end{align}
Owing to \eqref{G4.6}$_3$, we get
\begin{align}\label{G4.13}
\widetilde J_{2}
=&\, -4\int_{\mathbb R^3_{x}}
\partial^\alpha_{x}\nabla_x\cdot \widetilde b
\partial^\alpha_{x}\partial_t \widetilde c{\rm d}x\nonumber\\
=&\, \frac{2}{3} \int_{\mathbb R^3_{x}}
\partial^\alpha_{x}\nabla_x\cdot \widetilde b \partial^\alpha_{x}\big[2\nabla_{x}\cdot \widetilde b+\nabla_{x}  \cdot\Xi(\{\mathbf I-\mathbf P\}\widetilde f)-
2 \widetilde b\cdot(\nabla_{x}\phi^{(1)}+E)-2b^{(2)}\cdot\nabla_{x}\widetilde\phi\big] {\rm d}x \nonumber\\
\leq&\, \Big(\frac{4}{3}+\kappa\Big) \|\partial^\alpha_{x}\nabla_{x}\cdot \widetilde b\|_{L^2}^2+C\| \{\mathbf{I}-\mathbf{P}\}\widetilde f \|_{L_v^2(\dot H^1\cap\dot H^{N-1})}^2+C\delta_0 \|\nabla_{x}\widetilde\phi\|_{\dot H^1}^2  \nonumber\\
&+C\delta_0  \|\widetilde b\|_{\dot H^1\cap \dot H^{N-1}}^2 +C\delta_0 \|\widetilde\rho_{f}\|_{\dot H^1\cap \dot H^{N-2}}^2 . 
\end{align}
For the remaining term $\widetilde J_3$, similar to    \eqref{G3.38}, we have
\begin{align} \label{G4.14}
 \widetilde J_3
\leq&\,\kappa \|\nabla_x\partial^\alpha \widetilde b\|^2_{L^2}+C\| \{\mathbf{I}-\mathbf{P}\}\widetilde f \|_{L_v^2(\dot H^1\cap\dot H^{N-1})}^2 +C\delta_0 \|(\widetilde a,\widetilde b,\widetilde c)\|_{\dot H^1\cap \dot H^{N-1}}^2\nonumber\\
&+C\delta_0\|\nabla\widetilde\phi\|_{\dot H^1}^2+C\delta_0\|\widetilde\rho_{f}\|_{\dot H^1\cap \dot H^{N-2}}^2.
\end{align}
Summing over $1\leq |\alpha|\leq N-2$ and
substituting the estimates \eqref{G4.12}--\eqref{G4.14} into \eqref{G4.11} yields
\begin{align} \label{G4.15}
&\frac{{\rm d}}{{\rm d}t}{\sum_{1\leq |\alpha|\leq N-2}}
 \sum_{i,j=1}^3
\int_{\mathbb R^3_{x}}
\partial^\alpha_{x} (\partial_i \widetilde b_j+\partial_j \widetilde b_i)
\partial^\alpha_{x}\big(A_{ij}(\{\mathbf{I}-\mathbf{P}\}\widetilde f)+2\widetilde c\delta_{ij}\big){\rm d}x+  (1+\widetilde\eta_4) \| \widetilde b\|_{\dot H^2\cap \dot H^{N-1}}^2 \nonumber\\
\lesssim&\,   \| \{\mathbf{I}-\mathbf{P}\} \widetilde f \|_{L_v^2(\dot H^1\cap\dot H^{N-1})}^2+\delta_0\|(\widetilde a,\widetilde b,\widetilde c)\|_{\dot H^1\cap\dot H^{N-1}}^2+(\kappa+\delta_0) \|\widetilde\rho_{f}\|_{\dot H^1\cap \dot H^{N-2}}^2+\delta_0\|\nabla_{x}\widetilde\phi\|_{\dot H^1}^2,
\end{align}
for some constant $\widetilde\eta_{4}>0$.

The next step is to obtain the dissipation of $\widetilde c$. To this end, we conclude from  \eqref{G4.8}$_3$ that
\begin{align} \label{G4.16}
&\frac{{\rm d}}{{\rm d}t}
\sum_{i=1}^3
\int
\partial^\alpha_{x}\partial_i \widetilde c
\partial_{x}^\alpha B_i(\{\mathbf{I}-\mathbf P\}\widetilde f) {\rm d}x+\|\nabla_x\partial^\alpha \widetilde c\|_{L^2 }^2 \nonumber\\
= &\,
\sum_{i=1}^3
\int_{\mathbb R^3_{x}}
\partial^\alpha_{x}\partial_i\partial_t \widetilde c
\partial^\alpha_{x} B_i(\{\mathbf{I}-\mathbf P\}\widetilde f){\rm d}x
+
\sum_{i=1}^3
\int_{{\mathbb R_{x}^3}}
\partial^\alpha_{x}\partial_i \widetilde c
\partial^\alpha_{x}B_i (\widetilde{\mathcal R}+\widetilde{\mathcal G}) {\rm d}x\nonumber\\
\equiv:&\,\widetilde{J}_4+\widetilde J_5.
\end{align}
Thanks to \eqref{G4.6}$_3$, we infer that  
\begin{align} \label{G4.17}
\widetilde J_4
\leq &\,   \kappa \|\partial^\alpha_{x}\nabla_{x}\widetilde c\|_{L^2}^2+   C\| \{\mathbf{I}-\mathbf{P}\}\widetilde f \|_{L_v^2(\dot H^1\cap\dot H^{N-1})}^2 +C{(\delta_0+\kappa) }  \|\widetilde b\|_{\dot H^1\cap \dot H^{N-1}}^2\nonumber\\
&+C\delta_0\|\nabla_{x}\widetilde\phi\|_{\dot H^1}^2 + C\delta_0\|\widetilde \rho_{f}\|_{\dot H^1\cap \dot H^{N-2}}^2 .
\end{align}
Similar to \eqref{G3.42}, we compute that
\begin{align} \label{G4.18}
\widetilde J_5\leq&\,\kappa \|\nabla_x\partial^\alpha_{x} \widetilde c\|^2_{L^2} +C\| \{\mathbf{I}-\mathbf{P}\}\widetilde f \|_{L_v^2(\dot H^1\cap\dot H^{N-1})}^2 +C\delta_0 \|(\widetilde a,\widetilde b,\widetilde c)\|_{\dot H^1\cap \dot H^{N-1}}^2 \nonumber\\
&+C\delta_0\|\nabla\widetilde\phi\|_{\dot H^1}^2+C\delta_0\|\widetilde\rho_{f}\|_{\dot H^1\cap \dot H^{N-2}}^2.
\end{align}
Summing over $1\leq |\alpha|\leq N-2$ and putting the estimates \eqref{G4.17}--\eqref{G4.18} into \eqref{G4.16} yields
\begin{align}  \label{G4.19}
&\frac{{\rm d}}{{\rm d}t}
{\sum_{1\leq |\alpha|\leq N-2}}\sum_{i=1}^3
\int
\partial^\alpha_{x}\partial_i \widetilde c
\partial_{x}^\alpha B_i(\{\mathbf{I}-\mathbf P\}\widetilde f) {\rm d}x +\Big(\frac{1}{2}+\widetilde\eta_{5}\Big) \|\widetilde c\|_{\dot H^2\cap\dot H^{N-1}}^2 \nonumber\\
\lesssim&\, \| \{\mathbf{I}-\mathbf{P}\}\widetilde f \|_{L_v^2(\dot H^1\cap\dot H^{N-1})}^2+{(\delta_0+\kappa)} \|(\widetilde a,\widetilde b,\widetilde c)\|_{\dot H^1\cap \dot H^{N-1}}^2 + \delta_0\|\widetilde\rho_{f}\|_{\dot H^1\cap \dot H^{N-2}}^2+\delta_0\|\nabla_{x}\widetilde\phi\|_{\dot H^1}^2,  
\end{align}
for some constant $\widetilde\eta_{5}>0$.

For the dissipation of $\widetilde \rho_{f}$, according to \eqref{G4.6}$_1$ and \eqref{G4.6}$_2$, we infer that 
\begin{align*} 
&-\frac{{\rm d}}{{\rm d}t} \int_{\mathbb R^3_x}
\partial^\alpha_{x}\widetilde\rho_f 
\partial^\alpha_{x}\nabla_x\cdot \widetilde b {\rm d}x+\|\nabla_x\partial^\alpha_{x}\widetilde\rho_{f}\|^2_{L^2}
+\|\partial^\alpha_{x}\widetilde\rho_{f}\|^2_{L^2}\nonumber\\
=&\,
\|\partial^\alpha_{x}\nabla_x\cdot \widetilde b\|_{L^2}^2
-2\int_{\mathbb R^3_{x}}
\nabla_x\partial^\alpha_{x}\widetilde\rho_{f}\cdot
\nabla_x\partial^\alpha_{x} \widetilde c {\rm d}x-\int_{\mathbb R^3_{x}}
\nabla_x\partial^\alpha_{x}\widetilde\rho_{f}\cdot
\partial^\alpha_{x}\nabla_x\cdot\Theta(\{\mathbf{I}-\mathbf{P}\}\widetilde f){\rm d}x
\nonumber\\
&
+
\int_{\mathbb R^3_{x}}
\nabla_x\partial^\alpha_{x}\widetilde\rho_{f}\cdot
\partial^\alpha_{x} \big[ \widetilde\rho_{f}(\nabla_{x}\phi^{(1)}+E)+{\rho_{f}^{(2)}}\nabla_{x}\widetilde\phi \big]{\rm d}x \nonumber\\
\leq&\, \|\widetilde b\|_{\dot H^2\cap\dot H^{N-1}}^2+\Big(\frac{1}{2}+\kappa \Big)\|\nabla_x\partial^\alpha_{x}\widetilde \rho_{f}\|^2_{L^2}+ C\| \{\mathbf{I}-\mathbf{P}\}\widetilde f \|_{L_v^2(\dot H^1\cap\dot H^{N-1})}^2+C\delta_0 \|\widetilde\rho_{f}\|_{\dot H^1\cap \dot H^{N-2}}^2\nonumber\\
&+2\|\widetilde c\|_{\dot H^2\cap\dot H^{N-1}}^2+C\delta_0\|\nabla_{x}\widetilde \phi\|_{\dot H^1}^2,
\end{align*}
which implies that
\begin{align}\label{G4.20}
 &-\frac{1}{4}\frac{{\rm d}}{{\rm d}t}\sum_{{1\leq
 |\alpha|\leq N-2}}\int_{\mathbb R^3_x}
\partial^\alpha_{x}\widetilde \rho_f 
\partial^\alpha_{x}\nabla_x\cdot \widetilde b {\rm d}x
+\widetilde\eta_{6} \sum_{{1\leq
 |\alpha|\leq N-2}}\big(\| \nabla_x\partial^\alpha_{x}\widetilde \rho_{f}\|^2_{L^2}
+\|\partial^\alpha_{x}\widetilde \rho_{f}\|^2_{L^2}\big)\nonumber\\  
\leq&\, \frac{1}{2} \|\widetilde c\|_{\dot H^2\cap \dot H^{N-1}}^2+\frac{1}{4}\|\widetilde b\|_{\dot H^2\cap\dot H^{N-1}}^2+ C\| \{\mathbf{I}-\mathbf{P}\}\widetilde f \|_{L_v^2(\dot H^1\cap\dot H^{N-1})}^2+C\delta_0 \|\widetilde\rho_{f}\|_{\dot H^1\cap \dot H^{N-2}}^2 +C\delta_0\|\nabla_{x}\widetilde\phi\|_{\dot H^1}^2,    
\end{align}
for some constant $\widetilde\eta_6>0$.
Collecting the estimates \eqref{G4.15}, \eqref{G4.19}, and \eqref{G4.20}, we justify \eqref{G4.10} and complete the proof of Lemma \ref{L4.2}.
\end{proof}

To close the estimate \eqref{G4.2}, it is necessary to control the mixed space-velocity derivatives of the microscopic component \(\{\mathbf I-\mathbf P\}\widetilde f\). Therefore, we establish the corresponding estimates. The proof is based on the same energy method employed in Lemmas \ref{L3.6}--\ref{L3.8}. Since no new difficulties are involved, we omit the details and record the resulting estimates in Lemmas \ref{L4.3}--\ref{L4.5} (see also \cite{DN-2026}).

\begin{lem} \label{L4.3}
For strong solutions of system  \eqref{G4.1}, there exists a positive constant $\widetilde\eta_{7}>0$ such that
\begin{align}  \label{G4.21}
&\frac{{\rm d}}{{\rm d}t}\|\nu\{\mathbf{I}-\mathbf{P}\}\widetilde f(t) \|_{L_v^2(\dot H^1\cap\dot H^{N-2})}^2+ \widetilde  \eta_{7} \|\nu^{\frac{3}{2}}\{\mathbf{I}-\mathbf{P}\}\widetilde f(t)\|_{L_v^2(\dot H^1\cap\dot H^{N-2})}^2 \nonumber\\
&\quad\lesssim \|\nu^{\frac{1}{2}}\{\mathbf{I}-\mathbf{P}\}\widetilde f(t)\|_{L_v^2(\dot H^1\cap\dot H^{N-2})}^2+ \|\widetilde f(t)\|_{L_v^2(\dot H^1\cap\dot H^{N-1})}^2\nonumber\\
 & \quad\quad+\delta_0 \|\nu^{\frac{1}{2}}\nabla_v\{\mathbf{I}-\mathbf{P}\}\widetilde f(t)\|_{L_v^2(\dot H^1\cap \dot H^{N-2})}^2{+\delta_0\|\nabla_x\widetilde\phi(t)\|_{\dot H^1}^2},
\end{align}
for any $t\geq 0$.   Here, $\delta_0$ is defined in \eqref{G1.6}.
\end{lem} 

\begin{lem}\label{L4.4}
For strong solutions of system  \eqref{G4.1}, there exists a positive constant $\widetilde \eta_{8}>0$ such that
\begin{align}  \label{G4.22}
&\frac{{\rm d}}{{\rm d}t}C_{\beta} \sum_{\substack{ 1\leq |\beta|\leq N-1 \\|\alpha|+|\beta| \leq N-1}}\|\partial^\alpha_{x}\partial^\beta_{v}\{\mathbf{I}-\mathbf{P}\}\widetilde f(t) \|_{L_{x,v}^2 }^2+ \widetilde\eta_{8}\sum_{\substack{ 1\leq |\beta|\leq N -1\\|\alpha|+|\beta| \leq N-1}} \| \partial^{\alpha}_{x}\partial^\beta_{v}\{\mathbf{I}-\mathbf{P}\}\widetilde f(t)\|_{\nu}^2 \nonumber\\
\lesssim&\, \|\nu^{\frac{1}{2}}\{\mathbf{I}-\mathbf{P}\}\widetilde f(t)\|_{L_v^2(   H^{N-2})}^2+ \|\widetilde f(t)\|_{L_v^2(\dot H^{1}\cap\dot H^{N-1})}^2{+\delta_0\|\nabla_x\widetilde\phi(t)\|_{\dot H^1}^2},
\end{align}
for any $t\geq 0$.
\end{lem}

  \begin{lem} \label{L4.5}
For strong solutions of system   \eqref{G4.1}, there exists a positive constant $\widetilde \eta_{9}>0$ such that
\begin{align} \label{G4.23}
&\frac{{\rm d}}{{\rm d}t}\| \{\mathbf{I}-\mathbf{P}\}\widetilde f(t) \|_{L_{x,v}^2 }^2+ \widetilde\eta_{9} \| \{\mathbf{I}-\mathbf{P}\}\widetilde f(t)\|_{\nu}^2  \lesssim \|\widetilde f\|_{L_v^2(\dot H^1)}^2{+\delta_0\|\nabla_x\widetilde\phi(t)\|_{\dot H^1}^2},,
\end{align}
for any  $t\geq 0$.
\end{lem} 

With the assistance of Lemmas \ref{L4.1}--\ref{L4.5}, we present the time-weighted estimate at high frequencies as follows.

\begin{thm}\label{T4.6}
Under the assumptions of Theorem \ref{Th2}, we have
\begin{align} \label{G4.24}
&\sup_{t\geq 0} (1+t)^{\frac{1-\varepsilon-s_0}{2}}  \Big( 
 \|\widetilde f(t)\|_{L_v^2(\dot H^1\cap \dot H^{N-1})}+ \|\nabla_{x}\widetilde \phi(t)\|_{ \dot H^1\cap \dot H^{N-1}}+\|{\langle v\rangle}\widetilde f(t) \|_{L_v^2(\dot H^1\cap\dot H^{N-2})} +\|\{\mathbf{I}-\mathbf{P}\}\widetilde f(t)\|_{L_{x,v}^2}  \Big)\nonumber\\
& +\sup_{t\geq 0} (1+t)^{\frac{1-\varepsilon-s_0}{2}}   \sum_{\substack{ 1\leq |\beta|\leq N-1 \\|\alpha|+|\beta| \leq N-1}}\|\partial^\alpha_{x}\partial^\beta_{v}\{\mathbf{I}-\mathbf{P}\}\widetilde f(t) \|_{L_{x,v}^2 }    \nonumber\\
\lesssim&\,   \widetilde{\mathcal{D}}+\|\widetilde f_0\|_{L_v^2(\dot H^1\cap \dot H^{N-1})}+\|\nabla{\widetilde\phi_0}\|_{\dot H^1\cap\dot H^{N-1}}+\|\langle v\rangle\widetilde f_0 \|_{L_v^2(\dot H^1\cap\dot H^{N-2})} +\|\{\mathbf{I}-\mathbf{P}\}\widetilde f_0\|_{L_{x,v}^2} \nonumber\\
&+\sum_{\substack{ 1\leq |\beta|\leq N-1 \\|\alpha|+|\beta| \leq N-1}}\|\partial^\alpha_{x}\partial^\beta_{v}\{\mathbf{I}-\mathbf{P}\}\widetilde f_0\|_{L_{x,v}^2 } ,
\end{align}
for any $s\in\big[-\frac{1}{2}+\varepsilon,1-\varepsilon\big]$ with $s\geq s_0$ and $s_0\in\big(-\frac{1}{2},\frac{1}{2}\big]$, 
where  $\widetilde f_0(x,v)= f^{(1)}_0 (x,v)-f^{(2)}_0 (x,v)$, $\widetilde\phi_{0}(x)=\phi^{(1)}_{0}(x)-\phi^{(2)}_{0}(x) $,
and $\widetilde{\mathcal{D}} $ is given by
\begin{align*}
\widetilde{\mathcal{D}}:=  \sup_{t\geq 0}\,(1+t)^{\frac{1}{2}-\frac{\varepsilon+s_0}{2}} \Big(\|\widetilde f (t)\|_{L_v^2(\dot B_{2,\infty}^{1-\varepsilon} )}+ \|\nabla_{x}\widetilde \phi(t)\|_{\dot B_{2,\infty}^{1-\varepsilon}}  \Big) .   
\end{align*}    
\end{thm}
\begin{proof}
On the one hand, by  picking $\vartheta_3$ to be small enough, adding \eqref{G4.2} and $\vartheta_3\times$ \eqref{G4.10} yields
\begin{align}\label{G4.25}
&\frac{{\rm d}}{{\rm d}t}\Big(\| \widetilde  f(t)\|_{L_v^2(\dot H^1\cap \dot H^{N-1})}^2+\|\nabla_{x}\widetilde \phi(t)\|_{\dot H^1\cap\dot H^{N-1}}^2+\vartheta_3\sum_{1\leq k\leq N-2}\widetilde{\mathfrak E}_{k}(t) \Big)+\widetilde {\eta}_{10}\sum_{1\leq k\leq N-1}\|\nabla^k\{\mathbf{I}-\mathbf{P}\}\widetilde f(t)\|_{\nu}^2  \nonumber\\
&+\widetilde\eta_{10}\big( \|(\widetilde a,\widetilde b,\widetilde c)(t)\|_{\dot H^2\cap\dot H^{N-1}}^2+\|\widetilde \rho_{f}(t)\|_{\dot H^1\cap{\dot H^{N-2}}}^2 \big) \nonumber\\
\leq&\,  (\kappa+C\delta_0) \big( \|  (\widetilde a,\widetilde b,\widetilde c)(t)\|_{ \dot H^1}^2+\|\nabla\widetilde\phi(t)\|_{\dot H^1}^2 \big) +C\delta_0\sum_{\substack{ 1\leq |\beta|\leq N -1\\|\alpha|+|\beta| \leq N-1}} \| \partial^{\alpha}_{x}\partial^\beta_{v}\{\mathbf{I}-\mathbf{P}\}\widetilde f(t)\|_{\nu}^2,    
\end{align}
for some constant $\widetilde\eta_{10}>0$.
On the other hand,
we can deduce from  \eqref{G4.21}--\eqref{G4.23} that
\begin{align}\label{G4.26}
&\frac{{\rm d}}{{\rm d}t}\bigg(\|\nu\{\mathbf{I}-\mathbf{P}\}\widetilde f(t) \|_{L_v^2(\dot H^1\cap\dot H^{N-2})}^2 +\|\{\mathbf{I}-\mathbf{P}\}\widetilde f(t)\|_{L_{x,v}^2}^2+ C_{\beta}\sum_{\substack{ 1\leq |\beta|\leq N-1 \\|\alpha|+|\beta| \leq N-1}}\|\partial^\alpha_{x}\partial^\beta_{v}\{\mathbf{I}-\mathbf{P}\}\widetilde f(t) \|_{L_{x,v}^2 }^2\bigg)
 \nonumber\\
&+ \widetilde\eta_{11} \bigg(\|\nu^{\frac{3}{2}}\{\mathbf{I}-\mathbf{P}\}\widetilde f(t)\|_{L_v^2(\dot H^1\cap\dot H^{N-2})}^2+ \|\{\mathbf{I}-\mathbf{P}\}\widetilde f(t)\|_{\nu}^2+\sum_{\substack{ 1\leq |\beta|\leq N -1\\|\alpha|+|\beta| \leq N-1}} \| \partial^{\alpha}_{x}\partial^\beta_{v}\{\mathbf{I}-\mathbf{P}\}\widetilde f(t)\|_{\nu}^2\bigg)\nonumber\\    
&\quad\lesssim \|\widetilde f\|_{L_v^2(\dot H^1\cap\dot H^{N-1})}^2+\|\nu^\frac{1}{2}\{\mathbf{I}-\mathbf{P}\}\widetilde f\|_{L_v^2(\dot H^1\cap\dot H^{N-2})}^2 ,
\end{align}
for some constant $\widetilde{\eta}_{11}>0$.

We now introduce the following new temporal functional $\widetilde{\mathcal{E}}^H(t)$ and its corresponding dissipation rate $\widetilde{\mathcal{D}}^H(t)$:
\begin{align} \label{G4.27}
  \widetilde{\mathcal{E}}^H(t):=  &\, {\|\widetilde f(t)\|_{L_v^2(\dot H^1\cap \dot H^{N-1})}^2}+{\|\nabla_{x}\widetilde \phi(t)\|_{\dot H^1\cap\dot H^{N-1}}^2}+\vartheta_{3}\sum_{1\leq k\leq N-2}\widetilde{\mathfrak E}_{k}(t) +\vartheta_{4}\|\nu\{\mathbf{I}-\mathbf{P}\}\widetilde f(t) \|_{L_v^2(\dot H^1\cap\dot H^{N-2})}^2 \nonumber\\
& +\vartheta_{4}\|\{\mathbf{I}-\mathbf{P}\}\widetilde f(t)\|_{L_{x,v}^2}^2+ \vartheta_{4}C_{\beta}\sum_{\substack{ 1\leq |\beta|\leq N-1 \\|\alpha|+|\beta| \leq N-1}}\|\partial^\alpha_{x}\partial^\beta_{v}\{\mathbf{I}-\mathbf{P}\}\widetilde f(t) \|_{L_{x,v}^2 }^2,
\end{align}
where  $0<\vartheta_{4}\ll\vartheta_{3}\ll 1$ are sufficiently small constants that have been chosen,
and
\begin{align}\label{G4.28}
 \widetilde{\mathcal{D}}^H(t):= &\,\|\nu^{\frac{3}{2}}\{\mathbf{I}-\mathbf{P}\}\widetilde f(t)\|_{L_v^2(\dot H^1\cap\dot H^{N-2})}^2+ \|\{\mathbf{I}-\mathbf{P}\}\widetilde f(t)\|_{\nu}^2+{\|\nu^\frac{1}{2}\{\mathbf{I}-\mathbf{P}\}\widetilde f(t)\|_{L_v^2(\dot H^{N-1})}^2}\nonumber\\
& +\sum_{\substack{ 1\leq |\beta|\leq N -1\\|\alpha|+|\beta| \leq N-1}} \| \partial^{\alpha}_{x}\partial^\beta_{v}\{\mathbf{I}-\mathbf{P}\}\widetilde f(t)\|_{\nu}^2+\|(\widetilde
 a,\widetilde b,\widetilde c)(t)\|_{ \dot H^2\cap \dot H^{N-1}}^2+\|\widetilde\rho_{f}(t)\|_{\dot H^1\cap\dot H^{N-2}}^2.    
\end{align}
Moreover, from definitions of $\widetilde{\mathcal{E}}^H(t)$ and $\widetilde{\mathcal{D}}^H(t)$ in \eqref{G4.27}--\eqref{G4.28}, we have 
\begin{align*}
 \widetilde{\mathcal{E}}^H(t)\backsim&\, {\|\widetilde f(t)\|_{L_v^2(\dot H^1\cap \dot H^{N-1})}^2}+{\|\nabla_{x}\widetilde \phi(t)\|_{\dot H^1\cap\dot H^{N-1}}^2}  + \|\nu\{\mathbf{I}-\mathbf{P}\}\widetilde f(t) \|_{L_v^2(\dot H^1\cap\dot H^{N-2})}^2 \nonumber\\
 &+\|\{\mathbf{I}-\mathbf{P}\}\widetilde f(t)\|_{L_{x,v}^2}^2+  \sum_{\substack{ 1\leq |\beta|\leq N-1 \\|\alpha|+|\beta| \leq N-1}}\|\partial^\alpha_{x}\partial^\beta_{v}\{\mathbf{I}-\mathbf{P}\}\widetilde f(t) \|_{L_{x,v}^2 }^2,
\end{align*}
and
\begin{align}\label{G4.29}
 \widetilde{\mathcal{E}}^H(t) \lesssim   \widetilde{\mathcal{D}}^H(t)+ \|(\widetilde a,\widetilde b,\widetilde c)(t)\|_{\dot H^1}^2+ \|\nabla_{x}{\widetilde\phi}(t)\|_{\dot H^1}^2\lesssim \widetilde{\mathcal{D}}^H(t)+ \|\widetilde f\|_{L_v^2(\dot B_{2,\infty}^{1-\varepsilon})}^2+\|\nabla_{x}{\widetilde\phi}(t)\|_{\dot B_{2,\infty}^{1-\varepsilon}}^2.
\end{align}
Combining \eqref{G4.25} and \eqref{G4.26}, we arrive at
\begin{align*} 
&\frac{{\rm d}}{{\rm d}t} \widetilde{\mathcal{E}}^H(t)
 + \widetilde\eta_{12} \big(\widetilde{\mathcal{D}}^H(t) + \|(\widetilde a,\widetilde b,\widetilde c)(t)\|_{\dot H^1}^2+ \|\nabla_{x}{\widetilde\phi(t)}\|_{\dot H^1}^2\big)\lesssim   \|\widetilde f\|_{L_v^2(\dot B_{2,\infty}^{1-\varepsilon})}^2+\|\nabla_{x}{\widetilde\phi(t)}\|_{\dot B_{2,\infty}^{1-\varepsilon}}^2,
\end{align*}
for some constant $\widetilde\eta_{12}>0$, which together with \eqref{G4.29} yields
\begin{align*}
&\frac{{\rm d}}{{\rm d}t} \widetilde{\mathcal{E}}^H(t)
 + \widetilde\eta_{12}\widetilde{\mathcal{E}}^H(t)  \lesssim  \|\widetilde f(t)\|_{L_v^2(\dot B^{1-\varepsilon}_{2,\infty})}^2+\|\nabla_{x}{\widetilde\phi}(t)\|_{\dot B_{2,\infty}^{1-\varepsilon}}^2 \lesssim  (1+t)^{-(1-\varepsilon-s_0)}\widetilde{\mathcal{D}}^2.   
\end{align*}
Therefore, by leveraging Gr\"{o}nwall’s inequality, we obtain \eqref{G4.24} and thereby complete the proof of Theorem \ref{T4.6}
\end{proof}

\subsection{Semi-group estimates at low frequencies}
By Theorem \ref{T4.6}, the decay estimates for both $\langle v\rangle \widetilde{f}$ and $\nabla_v \widetilde{f}$ are already available. These estimates enable us to complete the semi-group argument and close the desired time-decay estimate for the low-frequency part.

\begin{thm}\label{T4.7}
Under the assumptions of Theorem \ref{Th2}, we have
\begin{align}\label{G4.30}
&\sup_{t\geq 0} (1+t)^{\frac{s-s_0}{2}}\big(\|\widetilde f(t)\|_{L_v^2(\dot B_{2,\infty}^s)}+\|\nabla_{x}\widetilde\phi(t)\|_{\dot B_{2,\infty}^s}\big) \nonumber\\
\lesssim&\, \delta_0 \big(\widetilde{\mathcal{D}}_{\varepsilon}+\widetilde{\mathcal{D}}\big)+\|\widetilde f_0\|_{L_v^2(\dot B_{2,\infty}^{s_0}\cap \dot B_{2,\infty}^{1-\varepsilon})}+\|\nabla_{x}\widetilde \phi_0\|_{\dot B_{2,\infty}^{s_0}\cap \dot B_{2,\infty}^{1-\varepsilon}}+\|\widetilde f_0\|_{L_v^2(\dot H^1\cap \dot H^{N-1})} \nonumber\\
&+\|{\nabla_{x}\widetilde\phi_0}\|_{\dot H^1\cap\dot H^{N-1}}+\|\langle v\rangle\widetilde f_0 \|_{L_v^2(\dot H^1\cap\dot H^{N-2})} +\|\{\mathbf{I}-\mathbf{P}\}\widetilde f_0\|_{L_{x,v}^2}+\sum_{\substack{ 1\leq |\beta|\leq N-1 \\|\alpha|+|\beta| \leq N-1}}\|\partial^\alpha_{x}\partial^\beta_{v}\{\mathbf{I}-\mathbf{P}\}\widetilde f_0\|_{L_{x,v}^2 } ,   
\end{align}
for any $s\in\big[-\frac{1}{2}+\varepsilon,1-\varepsilon\big]$ with $s\geq s_0$ and $s_0\in\big(-\frac{1}{2},\frac{1}{2}\big]$, 
where  $\widetilde f_0(x,v)= f_0^{(1)} (x,v)-f_0^{(2)} (x,v)$, $\widetilde\phi_{0}(x)=\phi^{(1)}_{0}(x)-\phi^{(2)}_{0}(x) $, and ${\widetilde{\mathcal{D}}_{\varepsilon}}$ is given by
\begin{align*}
\widetilde{\mathcal{D}}_{\varepsilon}:=\sup_{s_1\leq \bar s\leq 1-\varepsilon} \sup_{t\geq 0}\,(1+t)^{\frac{\bar s-s_0}{2}} \big(\|\widetilde f (t)\|_{L_v^2(\dot B_{2,\infty}^{\bar s}\cap \dot H^{N-1})}+\|\nabla_{x} \widetilde \phi (t)\|_{ \dot B_{2,\infty}^{\bar s}\cap \dot H^{N-1}}\big),\quad s_1=\max\{0,{s_0}\},   
\end{align*}
and $\delta_0$ is defined by \eqref{G1.6}.
\end{thm}  
\begin{proof}
By applying the Duhamel principle, in a similar manner to \eqref{G3.3}, the solution to system \eqref{G4.1} is expressed as
\begin{align} \label{G4.31}
\widetilde f(t)=&\,e^{t\mathcal B}\widetilde f_0
+\int_0^t e^{(t-\tau)\mathcal B} \mathbf{P}_1\big[ {-E}\cdot\nabla_{v}\widetilde f+\frac{1}{2}v\cdot E\widetilde f     \big](\tau) {\rm d}\tau\nonumber\\
&+\int_0^t e^{(t-\tau)\mathcal B} \{\mathbf{I}-\mathbf{P}\}\big[\Gamma(f^{(1)}+f^{(2)},\widetilde f)-E\cdot\nabla_{v}\widetilde f+\frac{1}{2}v\cdot E \widetilde f \big] (\tau){\rm d}\tau\nonumber\\
&+\int_0^t e^{(t-\tau)\mathcal B}   \mathbf{P}_1\big[-\nabla_{x}\widetilde \phi\cdot\nabla_{v}f^{(2)}+\frac{1}{2} \nabla_{x}\widetilde\phi\cdot v f^{(2)}-\nabla_{x}\phi^{(1)}\cdot\nabla_{v}\widetilde f+\frac{1}{2}v\cdot\nabla_{x}\phi^{(1)}\widetilde f     \big](\tau) {\rm d}\tau\nonumber\\
&+\int_0^t e^{(t-\tau)\mathcal B} \{\mathbf{I}-\mathbf{P}\} \big[ -\nabla_{x}\widetilde \phi\cdot\nabla_{v}f^{(2)}+\frac{1}{2} \nabla_{x}\widetilde\phi\cdot v f^{(2)} -\nabla_{x}\phi^{(1)}\cdot\nabla_{v}\widetilde f+\frac{1}{2}v\cdot\nabla_{x}\phi^{(1)}\widetilde f   \big](\tau) {\rm d}\tau \nonumber\\
\equiv:&\, e^{t\mathcal B}\widetilde f_0+\sum_{j=1}^4\widetilde K_{j}.
\end{align}
For the terms $\widetilde K_1$ and $\widetilde K_2$, similar to the argument in \cite[Lemma 4.1]{DLN-2026-arXiv} and \cite[Theorem 4.7]{DN-2026}, it can be deduced from \eqref{G2.23}--\eqref{G2.25} that
\begin{align}\label{G4.32}
&\|\widetilde K_1\|_{L_v^2(\dot B_{2,\infty}^s)}+\|\widetilde K_2\| _{L_v^2(\dot B_{2,\infty}^s)}\nonumber\\
\lesssim&\,   (1+t)^{-\frac{s-s_0}{2} }\big(\|\widetilde f_0 \|_{L_v^2(\dot B_{2,\infty}^{s_0}\cap \dot B_{2,\infty}^{1-\varepsilon})}+\|\nabla_{x}\widetilde \phi_0 \|_{ \dot B_{2,\infty}^{s_0}\cap \dot B_{2,\infty}^{1-\varepsilon}}\big) + \delta_0 {\widetilde {\mathcal{D}}_{\varepsilon}}(1+t)^{-\frac{s-s_0}{2}}\nonumber\\
&+ \delta_0 \Big(  \widetilde{\mathcal{D}}+\|\widetilde f_0\|_{L_v^2(\dot H^1\cap \dot H^{N-1})}+\|\nabla_{x}\widetilde\phi_0\|_{\dot H^1\cap\dot H^{N-1}}+\|\langle v\rangle\widetilde f_0 \|_{L_v^2(\dot H^1\cap\dot H^{N-2})} + \|\{\mathbf{I}-\mathbf{P}\}\widetilde f_0\|_{L_{x,v}^2} \nonumber\\
&\,\quad+\sum_{\substack{ 1\leq |\beta|\leq N-1 \\|\alpha|+|\beta| \leq N-1}}\|\partial^\alpha_{x}\partial^\beta_{v}\{\mathbf{I}-\mathbf{P}\}\widetilde f_0\|_{L_{x,v}^2 }\Big) \times(1+t)^{-\frac{s-s_0}{2}},
\end{align}
for any $s\in \big[-\frac{1}{2}+\varepsilon,1-\varepsilon\big]$.

With the assistance of \eqref{DNnew1}--\eqref{DNnew2}, we  next deal with the term $\widetilde K_4$. We analyze the cases for $s\in\big[-\frac{1}{2}+\varepsilon,1-\varepsilon \big]$ in three scenarios, specifically, $s\in\big[ -\frac{1}{2}+\varepsilon, \frac{1}{2} \big)$, $s = \frac{1}{2}$, and $s\in\big(\frac{1}{2},1-\varepsilon\big]$. First, we examine the case when $s\in \big[ -\frac{1}{2}+\varepsilon, \frac{1}{2} \big)$.  By utilizing \eqref{G2.16} and taking $g=\{\mathbf{I}-\mathbf{P}\}\widetilde f$ and $g=\{\mathbf{I}-\mathbf{P}\} f^{(2)}$ respectively in \eqref{A.6}, we have
\begin{align}\label{G4.33}
\|\widetilde K_4\|_{L_v^2(\dot B_{2,\infty}^s)}^2\lesssim &\,  \int_0^t(1+t-\tau)^{-\frac{1}{2}-s}   \Big(\big\|\nabla_{x}\widetilde\phi\cdot\big(\nabla_{v}f^{(2)}, v f^{(2)}\big)\big\|_{L_v^2(\dot B_{2,\infty}^{-\frac{1}{2}})}^2+\big\|\nabla_{x} \phi^{(1)}\cdot\big(\nabla_{v}\widetilde f , v \widetilde f\big)\big\|_{L_v^2(\dot B_{2,\infty}^{-\frac{1}{2}})}^2\Big){\rm d}\tau \nonumber\\
&+\int_0^t(1+t-\tau)^{-\frac{1}{2}-s}\Big(   \big\|\nabla_{x}\widetilde\phi\cdot\big(\nabla_{v}f^{(2)}, v f^{(2)}\big)\big\|_{L_v^2(\dot B_{2,\infty}^{-\frac{1}{2}})}^2+\big\|\nabla_{x} \phi^{(1)}\cdot\big(\nabla_{v}\widetilde f , v \widetilde f\big)\big\|_{L_v^2(\dot B_{2,\infty}^{-\frac{1}{2}})}^2\Big){\rm d}\tau \nonumber\\
\lesssim&\, \int_0^t(1+t-\tau)^{-\frac{1}{2}-s}    \big\|\nabla_{x}\widetilde\phi\cdot\big(\nabla_{v}f^{(2)}, v f^{(2)}\big)\big\|_{L_v^2(\dot B_{2,\infty}^{-\frac{1}{2}}\cap\dot B_{2,\infty}^{\frac{1}{2}})}^2 {\rm d}\tau \nonumber\\
&+\int_0^t(1+t-\tau)^{-\frac{1}{2}-s}    \big\|\nabla_{x} \phi^{(1)}\cdot\big(\nabla_{v}\widetilde f , v \widetilde f\big)\big\|_{L_v^2(\dot B_{2,\infty}^{-\frac{1}{2}}\cap\dot B_{2,\infty}^{\frac{1}{2}})}^2 {\rm d}\tau \nonumber\\
\lesssim&\,\sup_{t\geq  0} \|(\nabla_{v}f^{(2)},vf^{(2)})\|_{\dot B_{2,\infty}^{\frac{1}{2}}\cap\dot B ^{\frac{1}{2}+\varepsilon}_{{2,\infty}}}^2 \int_0^t(1+t-\tau)^{-\frac{1}{2}-s}  \|\nabla_{x}\widetilde\phi\|_{\dot B_{2,\infty}^{\frac{1}{2}}\cap\dot B_{2,\infty}^{\frac{3}{2}-\varepsilon}}^2   {\rm d}\tau \nonumber\\
&+\sup_{t\geq  0} \| \nabla_{x} \phi^{(1)}\|_{\dot B_{2,\infty}^{\frac{1}{2}}\cap\dot B ^{\frac{1}{2}+\varepsilon}_{2,\infty}}^2 \int_0^t(1+t-\tau)^{-\frac{1}{2}-s}  \|(\nabla_{v}\widetilde f,v\widetilde f)\|_{L_v^2(\dot B_{2,\infty}^{\frac{1}{2}}\cap\dot B_{2,\infty}^{\frac{3}{2}-\varepsilon})}^2   {\rm d}\tau \nonumber\\
\lesssim&\, \delta_0^2 \widetilde{\mathcal{D}}_{\varepsilon}^2 \int_0^t (1+\tau)^{-\frac{1}{2}-s} (1+t-\tau)^{-\frac{1}{2}+{s_0}} {\rm d}\tau+ \delta_0^2   \int_0^t (1+\tau)^{-\frac{1}{2}-s} {(1+t-\tau)}^{-\frac{1}{2}+s_0} {\rm d}\tau \nonumber\\
&\times \Big( \widetilde{\mathcal{D}}_{\varepsilon}+\|\widetilde f_0\|_{L_v^2(\dot H^1\cap \dot H^{N-1})}+\|\nabla\widetilde\phi_0\|_{\dot H^1\cap\dot H^{N-1}}+\|\langle v\rangle\widetilde f_0 \|_{L_v^2(\dot H^1\cap\dot H^{N-2})} +\|\{\mathbf{I}-\mathbf{P}\}\widetilde f_0\|_{L_{x,v}^2} \nonumber\\
&\quad+\sum_{\substack{ 1\leq |\beta|\leq N-1 \\|\alpha|+|\beta| \leq N-1}}\|\partial^\alpha_{x}\partial^\beta_{v}\{\mathbf{I}-\mathbf{P}\}\widetilde f_0\|_{L_{x,v}^2 }\Big) ^2\nonumber\\
\lesssim&\, { \delta_0^2 (1+t)^{- {(s-s_0)} }} \Big( \widetilde{\mathcal{D}}_{\varepsilon}+\widetilde{\mathcal{D}}+\|\widetilde f_0\|_{L_v^2(\dot H^1\cap \dot H^{N-1})}+\|\nabla\widetilde \phi_0\|_{\dot H^1\cap\dot H^{N-1}}+\|\langle v\rangle\widetilde f_0 \|_{L_v^2(\dot H^1\cap\dot H^{N-2})} \nonumber\\
&\quad+\|\{\mathbf{I}-\mathbf{P}\}\widetilde f_0\|_{L_{x,v}^2} +\sum_{\substack{ 1\leq |\beta|\leq N-1 \\|\alpha|+|\beta| \leq N-1}}\|\partial^\alpha_{x}\partial^\beta_{v}\{\mathbf{I}-\mathbf{P}\}\widetilde f_0\|_{L_{x,v}^2 }\Big)^2,
\end{align}
where we used the facts that ${-1+\varepsilon+s_0<-\frac{1}{2}+s_0}$  and
\begin{align*}  
 \int_0^t (1+\tau)^{-\frac{1}{2}-s} (1+t-\tau)^{-\frac{1}{2}+ {s_0} } {\rm d}\tau
 \lesssim&\, (1+t)^{-\frac{1}{2}+ {s_0} } \int_0^\frac{t}{2}(1+\tau)^{-\frac{1}{2}- {s} }{\rm d}\tau\nonumber\\
 &+(1+t)^{-\frac{1}{2}- {s} }\int_{\frac{t}{2}}^t (1+t-\tau)^{-\frac{1}{2}+ {s_0} } {\rm d}\tau\nonumber\\
 \lesssim&\, (1+t)^{-  {(s-s_0)} }.
\end{align*}
When $ s \geq \frac{1}{2}$ , given that $ 0 \leq \frac{1}{2} - s_0 < 1 $, we below adopt the following inequality instead on $\big[0,\frac{t}{2}\big]$:
\begin{align*}
&\int_0^{\frac{t}{2}}(1+t-\tau)^{-\frac12-s}
 (1+\tau)^{-\frac12+s_0} {\rm d}\tau
 +\sup_{\frac{t}{2}\leq\tau\leq t}(1+\tau)^{-(s-s_0)}\\
&\qquad\lesssim
(1+t)^{-\frac12-s}(1+t)^{\frac12+s_0}
+(1+t)^{-(s-s_0)}
\lesssim(1+t)^{-(s-s_0)}.
\end{align*}

For the case $s = \frac{1}{2}$, through direct calculation, one gets
\begin{align}\label{G4.34}
\|\widetilde K_4\|_{L_v^2(\dot B_{2,\infty}^{\frac{1}{2}})}^2\lesssim &\, 
{\int_0^{\frac{t}{2}}(1+t-\tau)^{-1}
 \big\|\nabla_x\widetilde\phi\cdot
( \nabla_v f^{(2)}, vf^{(2)})(\tau) \big\|_
{L_v^2(\dot B^{-\frac12}_{2,\infty}\cap\dot B^{\frac12}_{2,\infty})}^2}   {\rm d}\tau
\nonumber\\
& {+\int_0^{\frac{t}{2}}(1+t-\tau)^{-1}
\big\|\nabla_x\phi^{(1)}\cdot
(\nabla_v\widetilde f, v\widetilde f)(\tau) \big\|_
{L_v^2(\dot B^{-\frac12}_{2,\infty}\cap\dot B^{\frac12}_{2,\infty})}^2
 {\rm d}\tau}
\nonumber\\
&+\sup_{\frac{t}{2}\leq \tau\leq t} \big\|\nabla_{x}\widetilde\phi\cdot\big(\nabla_{v}f^{(2)}, v f^{(2)}\big)(\tau) \big\|_{L_v^2(\dot B_{2,\infty}^{-\frac{1}{2}}\cap\dot B_{2,\infty}^{\frac{1}{2}})}^2 \nonumber\\
&+\sup_{\frac{t}{2}\leq\tau\leq t}\big\|\nabla_{x} \phi^{(1)}\cdot\big(\nabla_{v}\widetilde f , v \widetilde f\big)(\tau) \big\|_{L_v^2(\dot B_{2,\infty}^{-\frac{1}{2}}\cap\dot B_{2,\infty}^{\frac{1}{2}})}^2 \nonumber\\
\lesssim&\, {\delta_0^2\Big(
\widetilde{\mathcal D}_\varepsilon+\widetilde{\mathcal D}
+\sqrt{\widetilde{\mathcal E}^H(0)}\Big)^2
\int_0^{\frac{t}{2}}(1+t-\tau)^{-1}(1+\tau)^{-\frac12+s_0}
{\rm d}\tau}
\nonumber\\
&+\sup_{\frac{t}{2}\leq \tau\leq t} \|(\nabla_{v}f^{(2)},vf^{(2)})(\tau)\|_{\dot B_{2,\infty}^{\frac{1}{2}}\cap\dot B ^{\frac{1}{2}+\varepsilon}_{2,\infty}}^2    \|\nabla_{x}\widetilde\phi(\tau)\|_{\dot B_{2,\infty}^{\frac{1}{2}}\cap\dot B_{2,\infty}^{\frac{3}{2}-\varepsilon}}^2    \nonumber\\
&+\sup_{\frac{t}{2}\leq \tau\leq t} \| \nabla_{x} \phi^{(1)}(\tau)\|_{\dot B_{2,\infty}^{\frac{1}{2}}\cap\dot B ^{\frac{1}{2}+\varepsilon}_{2,\infty}}^2    \|(\nabla_{v}\widetilde f,v\widetilde f)(\tau)\|_{\dot B_{2,\infty}^{\frac{1}{2}}\cap\dot B_{2,\infty}^{\frac{3}{2}-\varepsilon}}^2   \nonumber\\
\lesssim&\, \delta_0^2   \Big(\widetilde{\mathcal{D}}_{\varepsilon}+ \widetilde{\mathcal{D}}+\|\widetilde f_0\|_{L_v^2(\dot H^1\cap \dot H^{N-1})}+\|{\nabla\widetilde\phi_0}\|_{\dot H^1\cap\dot H^{N-1}}+\|\langle v\rangle\widetilde f_0 \|_{L_v^2(\dot H^1\cap\dot H^{N-2})} +\|\{\mathbf{I}-\mathbf{P}\}\widetilde f_0\|_{L_{x,v}^2} \nonumber\\
&\quad+\sum_{\substack{ 1\leq |\beta|\leq N-1 \\|\alpha|+|\beta| \leq N-1}}\|\partial^\alpha_{x}\partial^\beta_{v}\{\mathbf{I}-\mathbf{P}\}\widetilde f_0\|_{L_{x,v}^2 }\Big) ^2 \times(1+t)^{\varepsilon-\frac{1}{2}}+\delta_0^2 \widetilde{\mathcal{D}}_{\varepsilon} (1+t)^{ -{(s-s_0)} }\nonumber\\
\lesssim&\, { \delta_0^2 (1+t)^{- {(s-s_0)} }} \Big( \widetilde{\mathcal{D}}_{\varepsilon}+\widetilde{\mathcal{D}} +\|\widetilde f_0\|_{L_v^2(\dot H^1\cap \dot H^{N-1})}+\|{\nabla\widetilde\phi_0}\|_{\dot H^1\cap\dot H^{N-1}}+\|\langle v\rangle\widetilde f_0 \|_{L_v^2(\dot H^1\cap\dot H^{N-2})} \nonumber\\
&\quad+\|\{\mathbf{I}-\mathbf{P}\}\widetilde f_0\|_{L_{x,v}^2} +\sum_{\substack{ 1\leq |\beta|\leq N-1 \\|\alpha|+|\beta| \leq N-1}}\|\partial^\alpha_{x}\partial^\beta_{v}\{\mathbf{I}-\mathbf{P}\}\widetilde f_0\|_{L_{x,v}^2 }\Big)^2,
\end{align}
for $s=\frac{1}{2}$ with $s_0\in\big(-\frac{1}{2},\frac{1}{2}\big]$.
For the remaining case $s\in\big(\frac{1}{2},1-\varepsilon\big] $, using \eqref{G2.26} and \eqref{G2.27},
we have
\begin{align}\label{G4.35}
\|\widetilde K_4\|_{L_v^2(\dot B_{2,\infty}^{s})}^2\lesssim&\, {\delta_0^2
\Big(\widetilde{\mathcal D}_\varepsilon+\widetilde{\mathcal D}
+\sqrt{\widetilde{\mathcal E}^H(0)}\Big)^2
\int_0^{\frac{t}{2}}(1+t-\tau)^{-\frac12-s}(1+\tau)^{-\frac12+s_0}
{\rm d}\tau} \nonumber\\
&+ \sup_{\frac{t}{2}\leq \tau\leq t}   \|(\nabla_{v}f^{(2)},vf^{(2)})(\tau)\|_{\dot B_{2,\infty}^{\frac{1}{2}}\cap\dot B ^{\frac{3}{2} }_{2,1}}^2    \|\nabla_{x}\widetilde\phi(\tau)\|_{ \dot B_{2,\infty}^{s}}^2\nonumber\\
&+ \sup_{\frac{t}{2}\leq \tau\leq t}  \| \nabla_{x} \phi^{(1)}(\tau)\|_{\dot B_{2,\infty}^{\frac{1}{2}}\cap\dot B ^{\frac{3}{2}}_{2,1}}^2    \|(\nabla_{v}\widetilde f,v\widetilde f)(\tau)\|_{\dot B_{2,\infty}^{s}}^2    \nonumber\\
\lesssim&\, { \delta_0^2 (1+t)^{- {(s-s_0)} }} \Big( \widetilde{\mathcal{D}}_{\varepsilon}+\widetilde{\mathcal{D}} +\|\widetilde f_0\|_{L_v^2(\dot H^1\cap \dot H^{N-1})}+\|\nabla{\widetilde\phi_0}\|_{\dot H^1\cap\dot H^{N-1}}+\|\langle v\rangle\widetilde f_0 \|_{L_v^2(\dot H^1\cap\dot H^{N-2})} \nonumber\\
&\quad+\|\{\mathbf{I}-\mathbf{P}\}\widetilde f_0\|_{L_{x,v}^2} +\sum_{\substack{ 1\leq |\beta|\leq N-1 \\|\alpha|+|\beta| \leq N-1}}\|\partial^\alpha_{x}\partial^\beta_{v}\{\mathbf{I}-\mathbf{P}\}\widetilde f_0\|_{L_{x,v}^2 }\Big)^2.
\end{align}
By combining all the estimates \eqref{G4.33}--\eqref{G4.35}, we   further obtain
\begin{align}\label{G4.36}
\|\widetilde K_4\|_{L_v^2(\dot B_{2,\infty}^{s})}
\lesssim&\, { \delta_0 (1+t)^{- \frac{(s-s_0)}{2} }} \Big( \widetilde{\mathcal{D}}_{\varepsilon}+\widetilde{\mathcal{D}} +\|\widetilde f_0\|_{L_v^2(\dot H^1\cap \dot H^{N-1})}+\|\nabla\widetilde \phi_0\|_{\dot H^1\cap\dot H^{N-1}}+\|\langle v\rangle\widetilde f_0 \|_{L_v^2(\dot H^1\cap\dot H^{N-2})} \nonumber\\
&\quad+\|\{\mathbf{I}-\mathbf{P}\}\widetilde f_0\|_{L_{x,v}^2} +\sum_{\substack{ 1\leq |\beta|\leq N-1 \\|\alpha|+|\beta| \leq N-1}}\|\partial^\alpha_{x}\partial^\beta_{v}\{\mathbf{I}-\mathbf{P}\}\widetilde f_0\|_{L_{x,v}^2 }\Big) .
\end{align}

Finally, we handle the difficult term $\widetilde K_{3}$. Similar to \eqref{G3.10}, we have
\begin{align} \label{G4.37}
  &\mathbf{P}_1\big[-\nabla_{x}\widetilde \phi\cdot\nabla_{v}f^{(2)}+\frac{1}{2} \nabla_{x}\widetilde\phi\cdot v f^{(2)}-\nabla_{x}\phi^{(1)}\cdot\nabla_{v}\widetilde f+\frac{1}{2}v\cdot\nabla_{x}\phi^{(1)}\widetilde f     \big]  \nonumber\\
=&\,  \big(\widetilde \rho_{f} \nabla_{x}\phi^{(1)}+\rho_{f}^{(2)}\nabla_{x}\widetilde\phi\big)\cdot v\sqrt{M}+ \frac13\big(\widetilde b\cdot \nabla_{x}\phi^{(1)}+b^{(2)}\cdot\nabla_{x}\widetilde\phi\big)(|v|^2-3)\sqrt{M} \nonumber\\
=:&\, \widetilde L_1+\widetilde L_2.
\end{align}    
For the term $\widetilde L_1$ in equation \eqref{G4.37}, through direct computation, it can be shown that
\begin{align*}
\widetilde L_1=\nabla_x\cdot\Big[\nabla_{x}\widetilde\phi\otimes \nabla_{x}\phi^{(1)}+\nabla_{x} \phi^{(2)}\otimes \nabla_{x}\widetilde\phi {-\big(\nabla_{x}\phi^{(2)}\cdot\nabla_{x}\widetilde\phi+\frac12|\nabla_{x}\widetilde\phi|^2\big){\rm Id}}\Big]\cdot v\sqrt{M}.  
\end{align*}
Therefore, for $s\in \big[ -\frac{1}{2}+\varepsilon,\frac{1}{2}\big)$, by using \eqref{DNnew3}--\eqref{DNnew4}, we derive  
\begin{align}\label{G4.38}
\bigg\|\int_0^t e^{(t-\tau)\mathcal B} \widetilde L_1{\rm d}\tau  \bigg\|_{L_v^2(\dot B_{2,\infty}^s)}^2\lesssim&\, \int_0^t (1+t-\tau)^{-\frac{1}{2}-s} \big\|\big(\nabla_{x}\phi^{(1)},\nabla_{x}\phi^{(2)}\big)\big\|_{\dot B_{2,\infty}^{\frac{1}{2}}}^2 \|\nabla_{x}\widetilde\phi\|_{\dot B_{2,\infty}^{\frac{1}{2}}}^2{\rm d}\tau  \nonumber\\
\lesssim&\, \delta_0^2 \widetilde{\mathcal{D}}_{\varepsilon}^2(t) \int_0^t (1+\tau)^{-\frac{1}{2}-s} (1+t-\tau)^{-\frac{1}{2}+{s_0}} {\rm d}\tau\nonumber\\
\lesssim&\, { \delta_0^2 \widetilde{\mathcal{D}}_{\varepsilon}^2(t)(1+t)^{- {(s-s_0)} }}.
\end{align}
Similar to \eqref{G4.34} and \eqref{G4.35}, we can conclude that
\begin{align}\label{G4.39}
\bigg\|\int_0^t e^{(t-\tau)\mathcal B} \widetilde L_1{\rm d}\tau  \bigg\|_{L_v^2(\dot B_{2,\infty}^s)}^2 
\lesssim&\, {\delta_0^2\widetilde{\mathcal D}_\varepsilon^2
\bigg[\int_0^{\frac{t}{2}}(1+t-\tau)^{-\frac12-s}(1+\tau)^{-\frac12+s_0}
 {\rm d}\tau
+\sup_{\frac{t}{2}\leq\tau\leq t}(1+\tau)^{-(s-s_0)}\bigg]}\nonumber\\
\lesssim&\, { \delta_0^2{\widetilde{\mathcal{D}}_{\varepsilon}^2}(1+t)^{- {(s-s_0)} }},     
\end{align}
 for $s\in \big[ \frac{1}{2},1-\varepsilon\big]$.

Notice that
\begin{align*}
3\widetilde L_{2}=&\, \big[\widetilde b^{\parallel}\cdot \nabla_{x}\Delta^{-1}_{x}\rho_{f}^{(1)}+  (b^{(2)})^{\parallel}\cdot \nabla_{x}\Delta^{-1}_{x}\widetilde \rho_{f}+\widetilde b^{\perp}\cdot \nabla_{x}\Delta^{-1}_{x}\rho_{f}^{(1)}+  (b^{(2)})^{\perp}\cdot \nabla_{x}\Delta^{-1}_{x}\widetilde \rho_{f} \big] (|v|^2-3)\sqrt{M}   \nonumber\\
=&\,  -\partial_{t}\big(\nabla_{x}\Delta^{-1}_{x}\widetilde \rho_{f}\cdot \nabla_{x}\Delta^{-1}_{x} \rho_{f}^{(2)}+\frac{1}{2} |\nabla_{x}\Delta^{-1}_{x}\widetilde \rho_{f}|^2  \big)                         (|v|^2-3)\sqrt{M} \nonumber\\
&+\big[ \widetilde b^{\perp}\cdot \nabla_{x}\Delta^{-1}_{x}\rho_{f}^{(1)}+  (b^{(2)})^{\perp}\cdot \nabla_{x}\Delta^{-1}_{x}\widetilde \rho_{f} \big] (|v|^2-3)\sqrt{M}.
\end{align*}
By leveraging the new structure of the VPB system \eqref{I1.4} and adopting a similar argument as in {\bf Step 2} and {\bf Step 3} of Section 3.2, we can infer that
\begin{align}\label{G4.40}
&\bigg\|\int_0^t e^{(t-\tau)\mathcal B} \widetilde L_2{\rm d}\tau  \bigg\|_{L_v^2(\dot B_{2,\infty}^s)} \nonumber\\
\lesssim&\, { \delta_0 (1+t)^{- \frac{(s-s_0)}{2} }} \Big( \widetilde{\mathcal{D}}_{\varepsilon}+\widetilde{\mathcal{D}} +\|\widetilde f_0\|_{L_v^2(\dot H^1\cap \dot H^{N-1})}+\|\nabla\widetilde\phi_0\|_{\dot H^1\cap\dot H^{N-1}}+\|\langle v\rangle\widetilde f_0 \|_{L_v^2(\dot H^1\cap\dot H^{N-2})} \nonumber\\
&\quad+\|\{\mathbf{I}-\mathbf{P}\}\widetilde f_0\|_{L_{x,v}^2} +\sum_{\substack{ 1\leq |\beta|\leq N-1 \\|\alpha|+|\beta| \leq N-1}}\|\partial^\alpha_{x}\partial^\beta_{v}\{\mathbf{I}-\mathbf{P}\}\widetilde f_0\|_{L_{x,v}^2 }\Big),
\end{align}
for any $s\in \big[ -\frac{1}{2}+\varepsilon,1-\varepsilon\big]$.
By combining the estimates \eqref{G2.22}, \eqref{G4.32},  \eqref{G4.36}, \eqref{G4.38}--\eqref{G4.40}, we consequently obtain
\begin{align*}
&\|\widetilde f(t)\|_{L_v^2(\dot B_{2,\infty}^{s})}+\|\nabla_{x}\widetilde\phi(t)\|_{\dot B^s_{2,\infty}}\nonumber\\
\lesssim&\, { \delta_0 (1+t)^{- \frac{(s-s_0)}{2} }} \Big( \widetilde{\mathcal{D}}_{\varepsilon}+\widetilde{\mathcal{D}} +\|\widetilde f_0\|_{L_v^2(\dot H^1\cap \dot H^{N-1})}+\|\nabla\widetilde\phi_0\|_{\dot H^1\cap\dot H^{N-1}}+\|\langle v\rangle\widetilde f_0 \|_{L_v^2(\dot H^1\cap\dot H^{N-2})} \nonumber\\
&\quad+\|\{\mathbf{I}-\mathbf{P}\}\widetilde f_0\|_{L_{x,v}^2} +\sum_{\substack{ 1\leq |\beta|\leq N-1 \\|\alpha|+|\beta| \leq N-1}}\|\partial^\alpha_{x}\partial^\beta_{v}\{\mathbf{I}-\mathbf{P}\}\widetilde f_0\|_{L_{x,v}^2 }\Big)+(1+t)^{-\frac{s-s_0}{2}}\|\widetilde f_0\|_{L_v^2(\dot B_{2,\infty}^{s_0}\cap \dot B_{2,\infty}^{{1-\varepsilon}})} \nonumber\\
&+(1+t)^{-\frac{s-s_0}{2}}\|\nabla_{x}\widetilde \phi_0\|_{ \dot B_{2,\infty}^{s_0}\cap \dot B_{2,\infty}^{{1-\varepsilon}}} , 
\end{align*}
which yields \eqref{G4.30}, and thus we complete the proof of Theorem \ref{T4.7}.
\end{proof}

\subsection{Proof of asymptotic stability}
In this subsection, having established Theorems \ref{T4.6} and \ref{T4.7}, we are ready to prove the asymptotic stability of the solution to the Cauchy problem \eqref{I1.4}.
\begin{proof}[Proof of Theorem \ref{Th2}]
Combining  \eqref{G4.24} and \eqref{G4.30} yields
\begin{align*}
&\sup_{t\geq 0} (1+t)^{\frac{1-\varepsilon-s_0}{2}}  \Big( 
 \|\widetilde f(t)\|_{L_v^2(\dot H^1\cap \dot H^{N-1})}+\|\nabla_{x}\widetilde \phi(t)\|_{\dot H^1\cap\dot H^{{N-1}}}+\|\langle v\rangle\widetilde f(t) \|_{L_v^2(\dot H^1\cap\dot H^{N-2})} +\|\{\mathbf{I}-\mathbf{P}\}\widetilde f(t)\|_{L_{x,v}^2}  \Big)\nonumber\\
& +\sup_{t\geq 0} (1+t)^{\frac{1-\varepsilon-s_0}{2}}   \sum_{\substack{ 1\leq |\beta|\leq N-1 \\|\alpha|+|\beta| \leq N-1}}\|\partial^\alpha_{x}\partial^\beta_{v}\{\mathbf{I}-\mathbf{P}\}\widetilde f(t) \|_{L_{x,v}^2 }  +\sup_{t\geq 0} (1+t)^{\frac{s-s_0}{2}}\big(\|\widetilde f(t)\|_{L_v^2(\dot B_{2,\infty}^s)} +\|\nabla_{x}\widetilde \phi(t)\|_{ \dot B_{2,\infty}^s}\big) \nonumber\\
\lesssim&\,   \|\widetilde f_0\|_{L_v^2(\dot B_{2,\infty}^{s_0}\cap \dot B_{2,\infty}^{1-\varepsilon})}+\|\nabla_{x}\widetilde \phi_0\|_{ \dot B_{2,\infty}^{s_0}\cap \dot B_{2,\infty}^{1-\varepsilon}}+\|\widetilde f_0\|_{L_v^2(\dot H^1\cap \dot H^{N-1})}+\|\nabla\widetilde \phi_0\|_{ \dot H^1\cap \dot H^{N-1}}+\|\langle v\rangle\widetilde f_0 \|_{L_v^2(\dot H^1\cap\dot H^{N-2})}  \nonumber\\
&+\|\{\mathbf{I}-\mathbf{P}\}\widetilde f_0\|_{L_{x,v}^2}+\sum_{\substack{ 1\leq |\beta|\leq N-1 \\|\alpha|+|\beta| \leq N-1}}\|\partial^\alpha_{x}\partial^\beta_{v}\{\mathbf{I}-\mathbf{P}\}\widetilde f_0\|_{L_{x,v}^2 } \nonumber\\
\lesssim&\,  \|\widetilde f_0\|_{L_v^2(\dot B_{2,\infty}^{s_0}\cap \dot  H^{N-1})}+\|\nabla_{x}\widetilde \phi_0\|_{ \dot B_{2,\infty}^{s_0}\cap \dot H^{N-1} }+\|\langle v\rangle \widetilde f_0 \|_{L_v^2(\dot H^1\cap\dot H^{N-2})} +\|\{\mathbf{I}-\mathbf{P}\}\widetilde f_0\|_{L_{x,v}^2} \nonumber\\
&+\sum_{\substack{ 1\leq |\beta|\leq N-1 \\|\alpha|+|\beta| \leq N-1}}\|\partial^\alpha_{x}\partial^\beta_{v}\{\mathbf{I}-\mathbf{P}\}\widetilde f_0\|_{L_{x,v}^2 },
\end{align*}
for any $s\in\big[ -\frac{1}{2}+\varepsilon,1-\varepsilon\big]$, where we utilized the smallness of $\delta_0$.  Then, by using \eqref{G1.5}, we   obtain \eqref{G1.9}.

Finally, we pay attention to the uniform stability estimate \eqref{DNYG1.9}.
Let $\frac12<s\leq1-\varepsilon$. By the same argument as in Theorem \ref{T4.6},
with $1-\varepsilon$ replaced by $s<1$ and the
high-frequency terms absorbed into
$\widetilde{\mathcal D}^{H}(t)$, we obtain
\begin{align*}
\frac{{\rm d}}{{\rm d}t}\widetilde{\mathcal E}^{H}(t)
+\lambda^\prime\widetilde{\mathcal E}^{H}(t)
\lesssim
\|\widetilde f_L(t)\|_{L_v^2(\dot B^s_{2,\infty})}^2
+\|\nabla_x\widetilde\phi_L(t)\|_{\dot B^s_{2,\infty}}^2
\end{align*}
for some constant $\lambda^\prime>0$, which implies that
\begin{align}\label{DNY4.42}
\widetilde{\mathcal E}^{H}(t)
\lesssim 
e^{-\lambda^\prime t}\widetilde{\mathcal E}^{H}(0)
+\int_0^t e^{-\lambda^\prime(t-\tau)}
\big(
\|\widetilde f_L(\tau)\|_{L_v^2(\dot B^s_{2,\infty})}^2
+\|\nabla_x\widetilde\phi_L(\tau)\|_{\dot B^s_{2,\infty}}^2
\big){\rm d}\tau,
\end{align}
Applying  \eqref{G2.26}, \eqref{DNG2.28}, \eqref{DNY4.42}, and the argument for
$\widetilde L_2$ in {\bf Step 2} and {\bf Step 3} of Section 3.2,
we can obtain that
\begin{align*}
&\sup_{t\geq 0}
\big(
\|\widetilde f(t)\|_{\CE^{s,N-1}}
+\|\nabla_x\widetilde\phi(t)\|_
{\dot B^s_{2,\infty}\cap\dot H^{N-1}}
\big)\nonumber\\
\lesssim&\,
\sqrt{\widetilde{\mathcal E}^H(0)}
+\sup_{t\geq 0}
\Big(
\|\widetilde f_L(t)\|_{L_v^2(\dot B^s_{2,\infty})}
+\|\nabla_x\widetilde\phi_L(t)\|_{\dot B^s_{2,\infty}}
\Big)
\\
\lesssim&\,
\|\widetilde f_0\|_{\CE^{s,N-1}}
+\|\nabla_x\widetilde\phi_0\|_
{\dot B^s_{2,\infty}\cap\dot H^{N-1}}+\sup_{t\geq 0}
\big\|
\nu^{-1/2}\{\mathbf I-\mathbf P\}
\big[\Gamma(f^{(1)},\widetilde f)
+\Gamma(\widetilde f,f^{(2)})\big](t)
\big\|_{L_v^2(\dot B^{s-1}_{2,\infty})}\\
&+\sup_{t\geq 0}
\| (E+\nabla_x\phi^{(1)})
\cdot(v\widetilde f,\nabla_v\widetilde f)(\tau)\|_{L_v^2(\dot B^{s-1}_{2,\infty})}
 +\sup_{t\geq 0}
 \|\nabla_x\widetilde\phi
\cdot(vf^{(2)},\nabla_vf^{(2)})(t)
 \|_{L_v^2(\dot B^{s-1}_{2,\infty})}\\
&+\sup_{t\geq 0}
\|E(\widetilde\rho_f,\widetilde b)(t)\|_
{\dot B^{s-2}_{2,\infty}}+\sum_{i=1}^{2}\sup_{t\geq 0}
\|\nabla_x\phi^{(i)}\otimes
\nabla_x\widetilde\phi(t)\|_{\dot B^{s-1}_{2,\infty}}
\\
&+\sup_{t\geq 0}
\bigg(
\bigg\|\int_0^t e_L^{(t-\tau)\mathcal B}
\widetilde L_2(\tau){\rm d}\tau
\bigg\|_{L_v^2(\dot B^s_{2,\infty})}+
 \bigg\|\nabla_x\Delta_x^{-1}\mathbf P_0
\int_0^t e_L^{(t-\tau)\mathcal B}
\widetilde L_2(\tau)\,{\rm d}\tau
\bigg\|_{L_v^2(\dot B^s_{2,\infty})}
\bigg)
\\
\lesssim&\,
\|\widetilde f_0\|_{\CE^{s,N-1}}
+\|\nabla_x\widetilde\phi_0\|_
{\dot B^s_{2,\infty}\cap\dot H^{N-1}}+\sup_{t\geq 0}
\big(
\|\widetilde f(t)\|_{\CE^{s,N-1}}
+\|\nabla_x\widetilde\phi(t)\|_
{\dot B^s_{2,\infty}\cap\dot H^{N-1}}
\big) \\
&+\bigg[
\sup_{t\geq 0}
\|E(\tau)\|_{\dot B^{-\frac{3}{2}}_{2,\infty}\cap\dot H^N}
+\sum_{i=1}^{2}\sup_{t\geq 0}
\Big(
\|f^{(i)}(t)\|_{\CE^{\frac{1}{2},N}}
+\|\nabla_x\phi^{(i)}(t)\|_
{\dot B^{\frac{1}{2}}_{2,\infty}\cap\dot H^N}
\Big)
\bigg]
\\\lesssim&\,
\|\widetilde f_0\|_{\CE^{s,N-1}}
+\|\nabla_x\widetilde\phi_0\|_
{\dot B^s_{2,\infty}\cap\dot H^{N-1}}
 + \delta_0\sup_{t\geq 0}
\Big(
\|\widetilde f(t)\|_{\CE^{s,N-1}}
+\|\nabla_x\widetilde\phi(t)\|_
{\dot B^s_{2,\infty}\cap\dot H^{N-1}}
\Big).
\end{align*}
Taking the supremum over $t \geq 0$ and leveraging the smallness of $\delta_0$, we conclude that
\begin{align*}
&\sup_{t\geq0}
\Big(
\|\widetilde f(t)\|_{\CE^{s,N-1}}
+\|\nabla_x\widetilde\phi(t)\|_
{\dot B^s_{2,\infty}\cap\dot H^{N-1}}
\Big) \lesssim
\|\widetilde f_0\|_{\CE^{s,N-1}}
+\|\nabla_x\widetilde\phi_0\|_
{\dot B^s_{2,\infty}\cap\dot H^{N-1}},
\end{align*}
which proves \eqref{DNYG1.9}.
Therefore, the proof of Theorem \ref{Th2} is completed.    
\end{proof}

\section{Existence and stability of time-periodic solutions}\label{sec5}
This section is dedicated to the existence and stability analysis of time-periodic solutions to problem \eqref{G1.10}. The proof is inspired by Serrin's framework; for example, refer to \cite{DL-2015-AMSSB,Mp-1991-Nonlinearity,Serrin-1959-ARMA}. Below, we present our argument for the construction of time-periodic solutions to the Vlasov-Poisson-Boltzmann system.

\begin{proof}[Proof of Theorem \ref{Th1.3}]
Let $N\geq 4$. Suppose that the external force $E(t, x)$ is time-periodic with a period $T > 0$ and satisfies 
\begin{align}\label{G5.1}
\|E (t)\|_{\mathcal{C} (\mathbb{R}; \dot B_{2,\infty}^{-\frac{3}{2}}\cap \dot H^N )}\leq \delta, 
\end{align}
where $0<\delta<\delta_0$ is a constant to be determined later. Let $(f^{*}(t,x,v),\phi^*(t,x))$ be the global solution of the Cauchy problem \eqref{I1.4} with the initial data $(f^{*}(0, x, v),\phi^*(0,x))$ as stated in Theorem \ref{Th1}. Specifically, we set $f^{*}(0, x, v) \equiv 0$ and hence $\phi^*(0,x) \equiv 0$. It then follows from Theorem \ref{Th1} that
\begin{equation*}
\left\{
\begin{aligned}
&  f^*\in \mathcal{C}\big([0,\infty);L_v^2(\dot B_{2,\infty}^{\frac{1}{2}}\cap \dot H^N)\big),\quad \langle v\rangle f^*\in \mathcal{C}\big([0,\infty);L_v^2( \dot H^1\cap \dot H^{N-1})\big),  \\
&\{\mathbf{I}-\mathbf{P}\}f^*\in \mathcal{C}\big([0,\infty);H^N_{x,v}\big),\quad \nabla_{x}\phi^* \in \mathcal{C}\big([0,\infty); \dot B_{2,\infty}^{\frac{1}{2}}\cap \dot H^N\big) ,
\end{aligned}
\right.
\end{equation*}
and 
\begin{align}\label{G5.2}
\sup_{t\geq 0}\Big\{\|f^*(t)\|_{\CE^{\frac{1}{2},N}}+\|\nabla_{x} \phi^*(t)\|_{\dot B_{2,\infty}^{\frac{1}{2}}\cap \dot H^N } \Big\}
\leq C_0  \|E (t)\|_{\mathcal{C} (\mathbb R; \dot B_{2,\infty}^{-\frac{3}{2}}\cap \dot H^N)  }. 
\end{align}
Assuming that $(1 + C_0)\delta\leq\delta_0$, we  deduce from   \eqref{G5.1}--\eqref{G5.2} that
\begin{align}\label{G5.3}
 \sup_{t\geq 0}\Big\{\|f^*(t)\|_{\CE^{\frac{1}{2},N}}+\|\nabla_{x} \phi^*(t)\|_{\dot B_{2,\infty}^{\frac{1}{2}}\cap \dot H^N } \Big\}+\|E (t)\|_{\mathcal{C} (\mathbb{R}; \dot B_{2,\infty}^{-\frac{3}{2}}\cap \dot H^N )}\leq    \delta_0.
\end{align}

For the sake of notational simplicity,  we set
\begin{align*}
X_\varepsilon
:=
L_v^2\big(\dot B^{1-\varepsilon}_{2,\infty}\cap \dot H^{N-1}\big),\qquad   Y_\varepsilon
:=
 \dot B^{1-\varepsilon}_{2,\infty}\cap \dot H^{N-1}.  
\end{align*}
Let \(m\geq k\geq1\) be integers. In view of  \eqref{G5.3} and Theorem \ref{Th1},  $\big(f^*(t+(m-k)T,x,v),\phi^*(t+(m-k)T,x)\big)$
solves the Cauchy problem \eqref{I1.4} with initial data
\(\big(f^*((m-k)T,x,v),\phi^*((m-k)T,x)\big)\). On the other hand, \(\big(f^*(t,x,v),\phi^*(t,x) \big)\) is the solution
corresponding to the initial data \((f^*(0,x,v),\phi^*(0,x))\). Hence, by applying the stability estimate \eqref{G1.9} in Theorem \ref{Th2}, we obtain that 
\begin{align}\label{G5.4}
&\|f^*(t+(m-k)T)-f^*(t)\|_{X_\varepsilon}+\|\nabla_{x}\phi^*(t+(m-k)T)-\nabla_{x}\phi^*(t)\|_{Y_\varepsilon}
\nonumber\\
\lesssim&\,
(1+t)^{-\frac{1}{4}+\frac{\varepsilon}{2}}    \Big(\|f^*((m-k)T)-f^{*}(0)\|_{L_v^2(\dot B_{2,\infty}^{\frac{1}{2}}\cap\dot H^{N-1})}+ \|\nabla_{x}\phi^*((m-k)T)-\nabla_{x}\phi^{*}(0)\|_{ \dot B_{2,\infty}^{\frac{1}{2}}\cap\dot H^{N-1}}  \nonumber\\
&+\big\|\{\mathbf{I}-\mathbf{P}\}(f^*((m-k)T)-f^{*}(0))\big\|_{L_{x,v}^2}+\big\|\langle v\rangle (f^*((m-k)T)-f^{*}(0))\big\|_{L_v^2(\dot H^1\cap \dot H^{N-2})}\nonumber\\
&+\sum_{\substack{ 1\leq |\beta|\leq N-1 \\|\alpha|+|\beta| \leq N-1}}\big\|\partial^\alpha_{x}\partial^\beta_{v}\{\mathbf{I}-\mathbf{P}\}(  f^*((m-k)T)-f^{*}(0))\big\|_{L_{x,v}^2 }\Big) \nonumber\\
\lesssim&\,  \delta_0(1+t)^{-\frac{1}{4}+\frac{\varepsilon}{2}} ,
\end{align}
for  all $t\geq 0$.
Taking \(t=kT\) in \eqref{G5.4} gives
\begin{align}\label{G5.5}
\|f^*(mT)-f^*(kT)\|_{X_\varepsilon}+\|\nabla_{x}\phi^*(mT)-\nabla_{x}\phi^* (kT)\|_{Y_{\varepsilon}}
\lesssim
\delta_0(1+kT)^{-\frac14+\frac{\varepsilon}{2}},
\end{align}
for all integers $m\geq k\geq 1$.
Since \(0<\varepsilon<\frac12\), we have $
-\frac14+\frac{\varepsilon}{2}<0$.
Therefore, the right-hand side of \eqref{G5.5} tends to zero as \(k\to\infty\), uniformly for \(m\geq k\). It follows that $\{f^*(nT)\}_{n\geq1}$ is a Cauchy sequence in \(X_\varepsilon\), and $\{\nabla_{x}\phi^*(nT)\}_{n\geq1}$ is a Cauchy sequence in \(Y_\varepsilon\).
Since \(X_\varepsilon\) and \(Y_\varepsilon\) are complete, there exist \(f^*_\infty\in L_v^2(\dot B^{1 - \varepsilon}_{2,\infty}\cap\dot H^{N - 1})\) and \(\nabla_{x}\phi^*_\infty\in\dot B^{1 - \varepsilon}_{2,\infty}\cap\dot H^{N - 1}\) such that
\begin{align}\label{G5.6}
\lim_{n\to\infty}
\|f^*(nT)-f^*_\infty\|_{ X_{\varepsilon}}
=0 ,\qquad  \lim_{n\to\infty}
\|\nabla_{x}\phi^*(nT)-\nabla_{x}\phi^*_\infty\|_{  Y_{\varepsilon}}
=0.
\end{align}
Moreover, by the uniform estimate \eqref{G5.2} and the Fatou's lemma, we further have
\begin{align}\label{G5.7}
\|f^*_\infty\|_{\CE^{\frac12,N}}+\|\nabla_{x}\phi^*_{\infty}\|_{\dot B^{\frac{1}{2}}_{2,\infty}\cap\dot H^N}
\leq
\liminf_{n\to\infty}
\|f^*(nT)\|_{\CE^{\frac12,N}}+\liminf_{n\to\infty}
\|\nabla_{x}\phi^*(nT)\|_{\dot B^{\frac{1}{2}}_{2,\infty}\cap\dot H^N}
\leq
C_0\delta .
\end{align}
Therefore, based on \eqref{G5.1} and \eqref{G5.7}, the condition \eqref{G1.6} is satisfied when $f_0 = f_{\infty}^*$ and $\phi_0 = \phi_{\infty}^*$. By applying Theorem \ref{Th1} once more, let $(f_{T}(t),\phi_{T}(t))$ be the global solution of the Cauchy problem \eqref{I1.4} with the initial data $f_{T}(0,x,v)=f_{\infty}^{*}(x,v)$ and ${\phi^*_{T}(0,x)}=\phi_{\infty}^{*}(x)$.
According to the uniform estimate \eqref{DNYG1.9} in Theorem \ref{Th2}, it follows that
\begin{align}\label{G5.8}
&\|f_{T}(t)-f^{*}(t+(n-1)T)\|_{ X_{\varepsilon}}+\|\nabla_{x}\phi_{T}(t)-\nabla_{x}\phi^{*}(t+(n-1)T)\|_{ Y_{\varepsilon}}\nonumber\\
&\quad\leq C_1\|f_{\infty}^*-{f^*}((n-1)T)\|_{ X_{\varepsilon}}+ C_1\|\nabla_{x}\phi_{\infty}^*-{\nabla_{x}\phi^*}((n-1)T)\|_{ Y_{\varepsilon}},
\end{align}
for any $t\geq 0$ and $n\geq 1$.
Substituting $t = T$ into \eqref{G5.8} and then taking the infimum over $n$ gives
\begin{align*} 
 &\|f_{T}({T})- f_{T}(0)\|_{ X_{\varepsilon}}+\|\nabla_{x}\phi_{T}({T})- \nabla_{x}\phi_{T}(0)\|_{ Y_{\varepsilon}} \nonumber\\
=&\,\|f_{T}({T})- f_{\infty}^*\|_{X_{\varepsilon}}+\|\nabla_{x}\phi_{T}({T})- \nabla_{x}\phi_{\infty}^*\|_{Y_{\varepsilon}}\\
{=}&\,
 \liminf_{n\rightarrow\infty} \|f_{T}(T)-f^{*}(nT)\|_{X_{\varepsilon}}+\liminf_{n\rightarrow\infty} \|\nabla_{x}\phi_{T}(T)-\nabla_{x}\phi^{*}(nT)\|_{Y_{\varepsilon}}\nonumber\\ 
\leq&\, C_1 \liminf_{n\rightarrow\infty} \|f_{\infty}^*-{f^*}((n-1)T)\|_{X_{\varepsilon}}+C_1 \liminf_{n\rightarrow\infty} \|\nabla_{x}\phi_{\infty}^*-{\nabla_{x}\phi^*}((n-1)T)\|_{Y_{\varepsilon}}\nonumber\\
=&\, C_1 \lim_{n\rightarrow\infty} \|f_{\infty}^*-{f^*}((n-1)T)\|_{X_{\varepsilon}}+C_1 \lim_{n\rightarrow\infty} \|\nabla_{x}\phi_{\infty}^*-{\nabla_{x}\phi^*}((n-1)T)\|_{Y_{\varepsilon}}\nonumber\\
= &\,0,
\end{align*}
where we have  utilized \eqref{G5.6} in the last line.
Therefore, it can be concluded that
\begin{align*}
f_{T}(T,x,v)= f_{T}(0,x,v), \ \forall\,(x,v)\in \mathbb{R}^3\times \mathbb{R}^3,    
\end{align*}
and
\begin{align*}
\phi_{T}(T,x)= \phi_{T}(0,x ), \ \forall\,x\in \mathbb{R}^3,    
\end{align*}
which, by uniqueness of the Cauchy problem and periodicity of $E$, implies that $(f_T(t,x,v),\phi_{T}(t,x))$ extends periodically to a time-periodic solution with the same period $T > 0$ for the problem \eqref{G1.10}, and further, \eqref{fT.bdd} follows. Applying \eqref{G1.9} with $s_0=\frac12$ and $s=1-\varepsilon$
at times $t+nT$ and letting $n\to\infty$ yields uniqueness of
$(f_T,\nabla_x\phi_T)$ in the small $T$-periodic class.
If the force $E$ is time-independent, the
uniqueness implies that $(f_T,\nabla_x\phi_T)$ is stationary.
    
Finally, we prove the time decay estimate \eqref{G1.15}. 
For $\frac32<p\leq2$, set
$s_0=-3(\frac1p-\frac12)\in(-\frac12,0]$
and choose
$0<\varepsilon_1<\min\{\varepsilon,s_0+\frac12\}$.
In fact, by leveraging the embedding $L^p (\mathbb R^3) \hookrightarrow \dot B_{2,\infty}^{-3( \frac{1}{p}-\frac{1}{2})}(\mathbb R^3)$ and \eqref{G1.9} with $\varepsilon$ replaced by $\varepsilon_1$, we find that
\begin{align*}
 &\|(f - f_T)(t)\|_{L_v^2(\dot H^{ s} )}+\|(\nabla_{x}\phi - \nabla_{x}\phi_T)(t)\|_{ \dot H^{ s}  } \nonumber\\
 \lesssim&\, \|(f - f_T)(t)\|_{L_v^2(\dot B^{ s}_{2,1} )}+\|(\nabla_{x}\phi - \nabla_{x}\phi_T)(t)\|_{\dot B_{2,1}^s}\nonumber\\
 \lesssim&\,   \|(f - f_T)(t)\|_{L_v^2(\dot B_{2,\infty}^{{s_0}} )}^{\zeta^\prime} \|(f - f_T)(t)\|_{L_v^2(\dot B_{2,\infty}^{1-{\varepsilon_1}} )}^{1-\zeta^\prime}+ \|(\nabla_{x}\phi - \nabla_{x}\phi_T)(t)\|_{ \dot B_{2,\infty}^{{s_0}} }^{\zeta^\prime} \|(\nabla_{x}\phi - \nabla_{x}\phi_T)(t)\|_{ \dot B_{2,\infty}^{1-{\varepsilon_1}} }^{1-\zeta^\prime}\nonumber\\
  \lesssim&\, (1 + t)^{-\frac{s}{2}-\frac{3}{2}(\frac{1}{p}-\frac{1}{2})}\Big(\|(f - f_T)(0)\|_{L_v^2( L^p\cap\dot H^N)}+\|(\nabla_{x}\phi- \nabla_{x}\phi_T)(0)\|_{  L^p\cap\dot H^N} +\|\langle v\rangle (f - f_T)(0)\|_{L_v^2(\dot H^1\cap \dot H^{N-2})}   \nonumber\\
&+ \| \{\mathbf{I}-\mathbf{P}\}(f - f_T)(0)\|_{L_{x,v}^2}+\sum_{\substack{ 1\leq |\beta|\leq N-1 \\|\alpha|+|\beta| \leq N-1}}\|\partial^\alpha_{x}\partial^\beta_{v}\{\mathbf{I}-\mathbf{P}\}(f - f_T)(0)\|_{L_{x,v}^2 }\Big)  ,
\end{align*}
for any $s\in{\big[-\frac12+\varepsilon,1-\varepsilon\big]}$ with $s>s_0$, where $\zeta^\prime\in(0,1)$ satisfies ${s_0}\zeta^\prime +  (1-{\varepsilon_1})(1 - \zeta^\prime)=s$.
Therefore, we complete the proof of Theorem \ref{Th1.3}.
\end{proof}

\appendix
\section{Analytic tools}\label{sec.app}
First, we present some useful estimates related to the linear collision operator $\mathcal{L}$ and the nonlinear collision operator $\Gamma$.

\begin{lem}[{\!\!\cite[Lemmas 3.2--3.3]{Gy-CPAM-2006}}]\label{LA.1}
It holds that $\langle \mathcal{L} h_1, h_2\rangle=\langle h_1,\mathcal{L} h_2\rangle$, and $\langle-\mathcal{L}h,h\rangle\geq 0 $, with $\mathcal{L}h=0$ if and only if $h=\mathbf{P}h$. Moreover, it holds that
\begin{equation}\label{A.2}
\begin{aligned}
&\langle -\nu^{2l} \mathcal{L}h,h\rangle\geq \frac{1}{2}\lambda_0|\nu^{l} h|_{\nu}^2-C|h|_{\nu}^2,\ l\geq 0,\\
&\langle - \partial^\alpha_{x}\partial^\beta_{v} \mathcal{L}h,\partial^\alpha_{x}\partial^\beta_{v}h\rangle\geq \frac{1}{2}\lambda_0|\partial^\alpha_{x}\partial^\beta_{v}h|_{\nu}^2-C\sum_{|\beta'|\leq |\beta|-1}|\partial^\alpha_x\partial^{\beta'}_{v} h|_{\nu}^2,\ \ |\beta|\geq 1.
\end{aligned}
\end{equation}
\end{lem}

\begin{lem}[{\!\!{\cite[Lemma 2.3]{Gy-CPAM-2002}}} and  {\cite[Lemma 2.7]{UY-AA-2006}}] \label{LA.2}
There exists $C>0$ such that
\begin{align*} 
|\langle\Gamma(g_1,g_2),g_3\rangle|+|\langle\Gamma(g_2,g_1),g_3\rangle|\leq C\sup_{v}\{\nu^3{|g_3|}\} |g_1|_2|g_2|_2.   
\end{align*}
Moreover, for any $0\leq\eta\leq1$, we have
\begin{align}\label{A.3}
|\nu^{-\eta}\Gamma(g,h)|_2\leq C \big(|\nu^{1-\eta}g|_{2}|h|_{2}+|\nu^{1-\eta }h|_{2}|g|_{2}\big).   
\end{align}
\end{lem}

\begin{lem}[{\!\!\cite[Lemma 3.3]{Gy-CPAM-2006}}]\label{LA.3}
Let $g_{i}(x,v)$, $i=1,2,3,$ be smooth functions. Then we have
\begin{align}\label{A.4}
&|\langle \partial^\alpha_{x}\partial^\beta_{v} \Gamma (g_1,g_2), \partial^\alpha_{x}\partial^\beta_{v}  g_3\rangle |\nonumber\\ 
&\quad \leq    C\sum_{\substack{{\alpha_1+\alpha_2=\alpha},\,\beta_1+\beta_{2}\leq\beta \\|\alpha|+|\beta| \leq N}}\big ( |\partial_{x}^{\alpha_1}\partial^{\beta_1}_v g_1|_2|\partial^{\alpha_2}_{x}\partial^{\beta_2}_{v}g_2|_{\nu} + |\partial_{x}^{\alpha_1}\partial^{\beta_1}_v g_2|_2|\partial^{\alpha_2}_{x}\partial^{\beta_2}_{v}g_1|_{\nu}\big)|\partial^{\alpha}_{x}\partial^{\beta}_{v}g_3|_{\nu} .
\end{align}
\end{lem}

\begin{prop}[{\!\!\cite[Lemma 3.3]{Gy-CPAM-2006}}]\label{LA.4}
There exists $C>0$ such that
\begin{align}\label{A.5}
&|\langle\nu^2\partial^\alpha_{x} \Gamma(g_1,g_2),\partial^\alpha_x g_3\rangle|\nonumber\\
&\quad \leq  C \big (
|\nu   \partial^{\alpha_1}g_1|_{2} |\nu   \partial^{\alpha_2}g_2  |_{\nu}+ |\nu   \partial^{\alpha_1}g_1|_{\nu} |\nu   \partial^{\alpha_2}g_2  |_{2} \big) | \nu\partial^\alpha g_3|_{\nu}.
\end{align}
\end{prop}

Next, we collect Sobolev interpolation inequalities and embeddings.

\begin{lem} [{\!\!\cite[Lemma A.1]{GW-CPDE-2012}}]\label{LA.5}
Let \(2\leq p\leq \infty\) and \(0\leq m,\alpha\leq\ell\). When \(p = \infty\), we further require that \(m\leq \alpha + 1\) and \(\ell\geq \alpha + 2\). Then, for any \(g\in C_0^\infty(\mathbb{R}^3)\), we have
\begin{align*}
\|\nabla^\alpha g\|_{{L^p_{x}}}\lesssim \|\nabla^m g\|_{L^2_{x}}^{1 - \vartheta}\|\nabla^\ell g\|_{L^2_{x}}^{\vartheta},
\end{align*}
where \(0\leq \vartheta\leq 1\) and \(\alpha\) satisfies
\begin{align*}
 \frac{\alpha}{3}-\frac{1}{p}=\left(\frac {m}{3}-\frac{1}{2}\right)(1 - \vartheta)+\left( \frac{\ell}{3}-\frac{1}{2}\right)\vartheta.
\end{align*}
\end{lem}

\begin{lem} [{\!\!\cite[Theorem 1.38]{BCD-Book-2011}}]\label{LA.6}
If $s\in\big[0,\frac{3}{2}\big)$,  then the homogeneous Sobolev space $\dot H^s(\mathbb R^3)$ is continuously embedded in $L^{\frac{6}{3-2s}}(\mathbb R^3)$.
\end{lem}

Then, we  introduce the following form of Minkowski's inequality: 
\begin{lem}[{\!\!\cite[Lemma A.8]{GW-CPDE-2012}}]\label{LA.7}
Let $1\leq p<\infty$. Suppose $g $ is a measurable function defined on $\mathbb{R}_{y}^3\times \mathbb{R}_{z}^3$. Then, we have
\begin{align*}
\bigg(   \int_{\mathbb{R}_{z}^3} \bigg(\int_{\mathbb{R}_{y}^3} |g(y,z)|{\rm d}y\bigg)^p  {\rm d}z  \bigg)^{ {1}/{p}}\leq \int_{\mathbb{R}_{y}^3}     \bigg(\int_{\mathbb{R}_{z}^3} |g(y,z)|^p{\rm d}z\bigg)^{{1}/{p}}  {\rm d}y. 
\end{align*}
In particular, for $1\leq p\leq q\leq\infty$, we have
\begin{align*}
\|g\|_{L^q_{z}L^p_{y}}\leq \|g\|_{L^p_{y}L^q_{z
}}.    
\end{align*}
\end{lem}

Finally, we present the following velocity-weighted and Besov interpolation:
\begin{lem} \label{LA.8}
Let \(g=g(x,v)\) satisfy 
$g\in L^2_{x,v}$ and  $\langle v\rangle g\in L^2_v(\dot H^1)$.
Then, it holds that
\begin{align}\label{A.6}
\|\langle v\rangle^{1/2}g\|_{L^2_v(\dot B^{1/2}_{2,\infty})}
\lesssim
\|g\|_{L^2_{x,v}}^{1/2}
\|\langle v\rangle g\|_{L^2_v(\dot H^1_{x})}^{1/2}.    
\end{align}
\end{lem}

\begin{proof}
For each fixed \(v\), the homogeneous Besov interpolation \eqref{G2.15} gives
\begin{align*}
\|g(\cdot,v)\|_{\dot B^{1/2}_{2,\infty}}
\lesssim
\|g(\cdot,v)\|_{L^2_{x}}^{1/2}
\|g(\cdot,v)\|_{\dot H^1_{x}}^{1/2}.    
\end{align*}
Multiplying by \(\langle v\rangle^{1/2}\), squaring the result, and integrating with respect to \(v\), we obtain  
\begin{align}\label{A.7}
\|\langle v\rangle^{1/2}g\|_{L^2_v(\dot B^{1/2}_{2,\infty})}^2
=&\,
\int_{\mathbb R^3_v}
\langle v\rangle
\|g(\cdot,v)\|_{\dot B^{1/2}_{2,\infty}}^2 {\rm d}v
\nonumber\\
\lesssim&\,
\int_{\mathbb R^3_v}
\|g(\cdot,v)\|_{L^2_{x}}
\langle v\rangle
\|g(\cdot,v)\|_{\dot H^1_{x}}{\rm d}v
\nonumber\\
\lesssim&\,
\bigg(
\int_{\mathbb R^3_v}
\|g(\cdot,v)\|_{L^2_{x} }^2 {\rm d}v
\bigg)^{1/2}
\bigg(
\int_{\mathbb R^3_v}
\langle v\rangle^2
\|g(\cdot,v)\|_{\dot H^1 _{x}}^2 {\rm d}v
\bigg)^{1/2}
\nonumber\\
=&\,
\|g\|_{L^2_{x,v}}
\|\langle v\rangle g\|_{L^2_v(\dot H^1_{x} )}.    
\end{align}
Taking the square root of \eqref{A.7} yields \eqref{A.6}.
\end{proof}

\bigskip 
\noindent{\bf Acknowledgements:} 
The research of Renjun Duan was partially supported by the General Research Fund (Project No.~14301720) from RGC of Hong Kong, the grant from the National Natural Science Foundation of China (Project No.~12425109), and a Direct Grant from CUHK. The research of Tong Yang was supported by the General Research Fund of Hong Kong (Project No. 11302723), a start-up grant from The Hong Kong Polytechnic University (Project No. P0043962), and the Kuok Group Foundation.

\vspace{2mm}

\noindent\textbf{Conflict of interest.} The authors do not have any possible conflicts of interest.

\vspace{2mm}

\noindent\textbf{Data availability statement.}
 Data sharing is not applicable to this article as no data sets were generated or analyzed during the current study.

\bibliographystyle{plain}

\begin{thebibliography}{aaa}

\bibitem{AM-JSP-1994} L. Arkeryd, N. Maslova, 
On diffuse reflection at the boundary for the Boltzmann equation and related equations,
{\it J. Statist. Phys.} {\bf 77 (5--6)} (1994)  1051--1077.

\bibitem{BCD-Book-2011} H. Bahouri, J. Chemin, R. Danchin,
{\it  Fourier Analysis and Nonlinear Partial Differential Equations,}
Grundlehren Math. Wiss, 343,
Springer, Heidelberg, 2011. 

\bibitem{Bb-ARMA-2005}H. Beirão da Veiga,  
Time periodic solutions of the Navier-Stokes equations in unbounded cylindrical domains—Leray's problem for periodic flows,
{\it Arch. Ration. Mech. Anal.} {\bf 178 (3)} (2005)  301--325.


\bibitem{CIP-1994}  C. Cercignani, R. Illner,  M. Pulvirenti, 
{\it The Mathematical Theory of Dilute Gases}, 
Applied  Mathematical Sciences, 106,
Springer-Verlag, New York, 1994. 

\bibitem{DL-AM-1989} R. J. DiPerna,  P.-L. Lions,  
On the Cauchy problem for Boltzmann equations: global existence and weak stability,
{\it Ann. of Math. (2)} {\bf 130 (2)} (1989) 321--366.

\bibitem{Deguchi-2026} N. Deguchi, Stability analysis of time-periodic solutions to the Navier-Stokes-Fourier system in 3D whole space, arXiv:2601.00034.

\bibitem{DV-IM-2005}L. Desvillettes, C. Villani,  
On the trend to global equilibrium for spatially inhomogeneous kinetic systems: the Boltzmann equation,
{\it Invent. Math.} {\bf 159 (2)} (2005)  245--316.

\bibitem{D-Physical-2009}R.-J. Duan,   
Stability of the Boltzmann equation with potential forces on torus,
{\it Phys. D} {\bf 238 (17)} (2009) 1808--1820.

\bibitem{D-2011-Nonlinearity} R.-J. Duan,  
Hypocoercivity of linear degenerately dissipative kinetic equations,
{\it Nonlinearity} {\bf 24 (8)} (2011)  2165--2189.

\bibitem{DLN-2026-arXiv} R.-J. Duan, Y. Lei, J.-K. Ni, Global dynamics of the Boltzmann equation driven by a time-periodic source
in $\mathbb {R}^3$, preprint, 2026.

\bibitem{DLN-2026} R.-J. Duan, F. Li, J.-K. Ni, Z. Zhang,
Solutions to the Fokker-Planck equation with time-dependent force and source in  $\mathbb R^3$, arXiv:2610.23191.

\bibitem{DL-2015-AMSSB}R.-J. Duan, S. Liu,  
Time-periodic solutions of the Vlasov-Poisson-Fokker-Planck system,
{\it Acta Math. Sci. Ser. B (Engl. Ed.)} {\bf 35 (4)} (2015)  876--886.

\bibitem{DN-2026} R.-J. Duan, J.-K. Ni, Three-dimensional time-periodic problem on the Boltzmann equation with external force, arXiv:2604.21339.


 


\bibitem{RS-ARMA-2011}R.-J. Duan,  R. M. Strain,  
Optimal time decay of the Vlasov-Poisson-Boltzmann system in  $\mathbb{R}^3$,
{\it Arch. Ration. Mech. Anal.} {\bf 199 (1)} (2011)  291--328.

\bibitem{DUYZ-CMP-2008} R.-J. Duan, S. Ukai, T. Yang,  H. Zhao, 
Optimal decay estimates on the linearized Boltzmann equation with time dependent force and their applications,
{\it Comm. Math. Phys.} {\bf 277 (1)} (2008)  189--236.


\bibitem{DYZ-DCDS-2006} R.-J. Duan, T. Yang,  C. Zhu,  
Boltzmann equation with external force and Vlasov-Poisson-Boltzmann system in infinite vacuum,
{\it Discrete Contin. Dyn. Syst.} {\bf 16 (1)} (2006) 253--277.

\bibitem{EP-JMPA-1975} R. S. Ellis,  M. A. Pinsky, 
The first and second fluid approximations to the linearized Boltzmann equation,
{\it J. Math. Pures Appl. (9)} {\bf 54} (1975) 125--156.

\bibitem{EMPS-ARMA-1999} E. Feireisl,  Š. Matuš${\dot {\rm u}}$-Nečasov\'{a}, H. Petzeltov\'{a},   I. Straškraba,  
On the motion of a viscous compressible fluid driven by a time-periodic external force.
{\it Arch. Ration. Mech. Anal.} {\bf 149 (1)} (1999)  69--96.

\bibitem{Glassey 1996} R. T. Glassey, {\it The Cauchy Problem in Kinetic Theory}, SIAM, Philadelphia, 1996.


\bibitem{GMM-MSM-2017}  M. P. Gualdani, S. Mischler,  C. Mouhot,  
Factorization of non-symmetric operators and exponential H-theorem,
{\it M\'{e}m. Soc. Math. Fr. (N.S.)}   {\bf  153} (2017).

\bibitem{Gy-CPAM-2002} Y. Guo,  
The Vlasov-Poisson-Boltzmann system near Maxwellians,
{\it Comm. Pure Appl. Math.} {\bf 55 (9)} (2002) 1104--1135.

\bibitem{GY-iumj-2004} Y. Guo, The Boltzmann equation in the whole space,
 {\it Indiana Univ. Math. J.} {\bf 53} (2004) 1081--1094.

\bibitem{Gy-CPAM-2006}
Y. Guo,  
Boltzmann diffusive limit beyond the Navier-Stokes approximation,
{\it Comm. Pure Appl. Math.} {\bf 59 (5)} (2006)  626--687.

\bibitem{GW-CPDE-2012}  Y. Guo,  Y. Wang, 
Decay of dissipative equations and negative Sobolev spaces,
{\it Comm. Partial Differential Equations} {\bf 37 (12)}  (2012)  2165--2208.


\bibitem{Hamdache-ARMA-1992} K. Hamdache,  
Initial-boundary value problems for the Boltzmann equation: global existence of weak solutions,
{\it Arch. Rational Mech. Anal.} {\bf 119 (4)} (1992)  309--353.


\bibitem{KaTs}
Y. Kagei, K. Tsuda, 
Existence and stability of time periodic solution to the compressible Navier-Stokes equation for time periodic external force with symmetry, {\it J. Differential Equations} {\bf 258 (2)} (2015), 399–444.

\bibitem{Ks-JJAM-1990} S. Kawashima,  
The Boltzmann equation and thirteen moments,
{\it Japan J. Appl. Math.} {\bf 7 (2)} (1990)  301--320.

\bibitem{LMZ-ARMA-2010} H.-L. Li,  A. Matsumura,  G. Zhang, 
Optimal decay rate of the compressible Navier-Stokes-Poisson system in $\mathbb R^3$,
{\it Arch. Ration. Mech. Anal.} {\bf 196 (2)} (2010) 681--713.

\bibitem{LYY-PD-2004} T.-P. Liu,  T. Yang,  S.-H. Yu,  
Energy method for Boltzmann equation,
{\it Phys. D} {\bf 188 (3--4)} (2004)  178--192.


\bibitem{Mp-1991-Nonlinearity} P. Maremonti, 
Existence and stability of time-periodic solutions to the Navier-Stokes equations in the whole space,
{\it Nonlinearity} {\bf 4 (2)} (1991)  503--529.


\bibitem{Mischler-CMP-2000} S. Mischler,  
On the initial boundary value problem for the Vlasov-Poisson-Boltzmann system,
{\it Comm. Math. Phys.} {\bf 210 (2)} (2000) 447--466.

\bibitem{MN-nonlinearity-2006}C. Mouhot, L. Neumann,  
Quantitative perturbative study of convergence to equilibrium for collisional kinetic models in the torus,
{\it Nonlinearity} {\bf 19 (4)} (2006)  969--998.





\bibitem{Serrin-1959-ARMA}
J. Serrin, 
A note on the existence of periodic solutions of the Navier-Stokes equations,
{\it Arch. Rational Mech. Anal.} {\bf 3} (1959) 120--122.


\bibitem{SK-HMJ-1985} Y. Shizuta, S. Kawashima, 
Systems of equations of hyperbolic-parabolic type with applications to the discrete Boltzmann equation,
{\it Hokkaido Math. J.} {\bf 14 (2)} (1985) 249--275.

\bibitem{SG-CPDE-2006} M. R. Strain,  Y. Guo,  
Almost exponential decay near Maxwellian,
{\it Comm. Partial Differential Equations} {\bf 31 (1--3)}  (2006)  417--429.


\bibitem{SG-ARMA-2008}M. R. Strain,  Y. Guo,
Exponential decay for soft potentials near Maxwellian,
{\it Arch. Ration. Mech. Anal.} {\bf 187 (2)} (2008) 287--339.

\bibitem{Tk-ARMA-2016} K. Tsuda, 
On the existence and stability of time periodic solution to the compressible Navier-Stokes equation on the whole space,
{\it Arch. Ration. Mech. Anal.} {\bf 219 (2)} (2016)  637--678.

\bibitem{Ukai-1974} S. Ukai,  
On the existence of global solutions of mixed problem for non-linear Boltzmann equation,
{\it Proc. Japan Acad.} {\bf 50} (1974) 179--184.

\bibitem{Ukai-1976} S. Ukai,  
Les solutions globales de l'\'{e}quation de Boltzmann dans l'espace tout entier et dans le demi-espace,
{\it C. R. Acad. Sci. Paris S\'{e}r. A-B} {\bf 282 A} (1976)  317--320.

\bibitem{Ukai-2006} S. Ukai,  
Time-periodic solutions of the Boltzmann equation,
{\it Discrete Contin. Dyn. Syst.} {\bf 14 (3)} (2006)  579--596.

\bibitem{UY-AA-2006}  S. Ukai,  T. Yang,  
The Boltzmann equation in the space $L^2\cap L^\infty_{\beta}$: global and time-periodic solutions,
{\it Anal. Appl. (Singap.)}   {\bf 4 (3)} (2006)  263--310.

\bibitem{Valli-1983} A. Valli,  
Periodic and stationary solutions for compressible Navier-Stokes equations via a stability method,
{\it Ann. Scuola Norm. Sup. Pisa Cl. Sci. (4)} {\bf 10 (4)} (1983)  607--647.




\bibitem{VZ-CMP-1986} A. Valli, W. M.  Zajaczkowski, 
Navier-Stokes equations for compressible fluids: global existence and qualitative properties of the solutions in the general case,
{\it Comm. Math. Phys.} {\bf 103 (2)} (1986)  259--296.

\bibitem{Villani-MAMS-2009} C. Villani,  
Hypocoercivity,
{\it Mem. Amer. Math. Soc.} {\bf 202 (950)} (2009).




\bibitem{YYZ-ARMA-2006}T. Yang, H. Yu,  H. Zhao,  
Cauchy problem for the Vlasov-Poisson-Boltzmann system,
{\it Arch. Ration. Mech. Anal.} {\bf 182 (3)} (2006)  415--470.

\bibitem{YZ-CMP-2006} T. Yang, H. Zhao,  
Global existence of classical solutions to the Vlasov-Poisson-Boltzmann system,
{\it Comm. Math. Phys.} {\bf 268 (3)} (2006)  569--605.




\end{thebibliography}

\end{document}